\documentclass[12pt,twoside,openany]{report}  

\usepackage{subfiles}
\usepackage{authblk}
\usepackage{amsmath,amssymb,amsthm}
\usepackage[left=2.5cm, right=2.5cm, top=1in, bottom=1in]{geometry}
\usepackage{framed}
\usepackage{mdframed}
\usepackage{graphicx}
\usepackage{bm}
\usepackage{theoremref}
\usepackage{enumitem}
\usepackage{fancyhdr}
\usepackage{color}
\usepackage{float}
\usepackage{dsfont}
\usepackage{mathtools}
\usepackage{breqn}
\usepackage{bbm}
\usepackage{mathrsfs}
\usepackage{amstext}
\usepackage{hyperref}
\usepackage{accents}
\usepackage{tikz}
\usepackage{caption}
\usepackage{subcaption}
\usepackage{setspace}
\usepackage[level, nodayofweek]{datetime}
\newdateformat{monthyeardate}{%
	\monthname[\THEMONTH], \THEYEAR}

\hypersetup{
	colorlinks   = true,
	citecolor    = blue,
	linkcolor    = blue
}
\usepackage{pdfsync}
\usepackage{mathrsfs}
\usepackage{nicefrac}

\usepackage{verbatim}
\usepackage[utf8]{inputenc}
\usepackage[english]{babel}

\newcommand{\notinsubfile}[1]{}

\DeclarePairedDelimiter\floor{\lfloor}{\rfloor}

\makeatletter
\def\thm@space@setup{%
	\thm@preskip=0.5cm plus 0.25cm minus 0cm
	\thm@postskip=\thm@preskip 
}
\makeatother

\numberwithin{equation}{section}
\newtheorem{theorem}{Theorem}[section]
\newtheorem{corollary}[theorem]{Corollary}
\newtheorem{lemma}[theorem]{Lemma}
\newtheorem{proposition}[theorem]{Proposition}
\newtheorem{assumption}[theorem]{Assumption}
\theoremstyle{remark}
\newtheorem{definition}[theorem]{Definition}
\newtheorem{remark}[theorem]{Remark}
\newtheorem{example}[theorem]{Example}

\let\oldproofname=\proofname
\renewcommand{\proofname}{\rm\bf{\oldproofname}}

\newcommand{\mcB}{\mathcal{B}}
\newcommand{\mcC}{\mathcal{C}}

\newcommand{\mcE}{\mathcal{E}}
\newcommand{\mcF}{\mathcal{F}}
\newcommand{\mcG}{\mathcal{G}}
\newcommand{\mcH}{\mathcal{H}}

\newcommand{\mcK}{\mathcal{K}}
\newcommand{\mcL}{\mathcal{L}}
\newcommand{\mcM}{\mathcal{M}}
\newcommand{\mcN}{\mathcal{N}}

\newcommand{\mcP}{\mathcal{P}}

\newcommand{\mcS}{\mathcal{S}}

\newcommand{\mcX}{\mathcal{X}}

\newcommand{\indic}{\mathds{1}}

\newcommand{\mbC}{\mathbb{C}}

\newcommand{\mbE}{\mathbb{E}}

\newcommand{\mbN}{\mathbb{N}}

\newcommand{\mbP}{\mathbb{P}}

\newcommand{\mbR}{\mathbb{R}}

\newcommand{\mbT}{\mathbb{T}}

\newcommand{\mbZ}{\mathbb{Z}}

\newcommand{\mfB}{\mathfrak{B}}

\newcommand{\mfR}{\mathfrak{R}}

\newcommand{\msH}{\mathscr{H}}

\newcommand{\msL}{\mathscr{L}}

\newcommand{\msN}{\mathscr{N}}

\newcommand{\msP}{\mathscr{P}}

\let\div\relax
\DeclareMathOperator{\div}{div}

\DeclareMathOperator*{\esssup}{ess\,sup}
\DeclareMathOperator*{\Tr}{Tr}

\newcommand{\supp}{\text{supp}}

\newcommand{\dd}{\mathop{}\!\mathrm{d}}

\newcommand{\Op}{\text{Op}}

\newcommand{\tzero}{|_{t=0}}

\newcommand{\blue}[1]{{\color{blue} #1}}

\newcommand{\vertiii}[1]{{\left\vert\kern-0.25ex\left\vert\kern-0.25ex\left\vert #1 
		\right\vert\kern-0.25ex\right\vert\kern-0.25ex\right\vert}}


\makeatletter
\newcommand*{\circledleq}{%
	\mathrel{%
		\tiny\mathpalette\@mathcircledtikz{\leqslant}%
	}%
}

\newcommand*{\circledless}{%
	\mathrel{%
		\tiny\mathpalette\@mathcircledtikz{<}%
	}%
}

\newcommand*{\circledgeq}{%
	\mathrel{%
		\tiny\mathpalette\@mathcircledtikz{\geqslant}%
	}%
}

\newcommand*{\circledgreater}{%
	\mathrel{%
		\tiny\mathpalette\@mathcircledtikz{>}%
	}%
}

\newcommand*{\circledequal}{%
	\mathrel{%
		\tiny\mathpalette\@mathcircledtikz{=}%
	}%
}

\newcommand*{\@mathcircledtikz}[2]{%
	\tikz[
	baseline=(X.base),
	inner sep=0.5\pgflinewidth,
	line width={%
		0.4pt%
		\ifx#1\scriptstyle -.1pt\fi
		\ifx#1\scriptscriptstyle -.2pt\fi
	},%
	]
	\node[circle,draw] (X) {$#1#2\m@th$};%
}
\makeatother

\makeatletter
\newcommand*\bigcdot{\mathpalette\bigcdot@{.5}}
\newcommand*\bigcdot@[2]{\mathbin{\vcenter{\hbox{\scalebox{#2}{$\m@th#1\bullet$}}}}}
\makeatother

\newlength{\myleftlen}

\newcommand{\RN}[1]{%
	\textup{\uppercase\expandafter{\romannumeral#1}}%
}

\long\def\avi#1{{\color{red}A:\ #1}}

\title{\vspace{-5em} Variational Methods in Stochastic PDEs}  
\author{A. Mayorcas} 
\affil{Mini-course given at \emph{Probability and Analysis Summer School }\\\vspace{0.8em}
Feza G\"ursey Institute, Istanbul\\ \vspace{0.8em}
Problem Classes: \emph{D. O'Kane}
}
\date{\today}

\begin{document}
	\maketitle
	\tableofcontents

	\chapter{ Introduction and Preliminaries}\label{ch:Background}
This mini-course aims to serve as a brief introduction to the study of stochastic partial differential equations. The main goal of these notes is to provide a useful resource for both PDE and stochastic analysts. We aim to demonstrate both how some standard PDE techniques can be extended to the stochastic setting and how the tools of stochastic analysis can be extended to infinite dimensional settings, common to the analysis of PDE. The main focus of this course is on the variational method for coercive-monotone SPDE. The tools we introduce along the way should be familiar to researchers in both areas and hopefully this provides a sufficiently broad base of tools for applications to problems beyond the scope of this course. In Chapter \ref{ch:Extras} we give a brief introduction to the pathwise approach to SPDE, which has been a topic of significant research in recent years, sparked by the development of novel approaches to singular SPDE, \cite{hairer_14_RegStruct,gubinelli_imkeller_perkowski_15_GIP}.\\ \par 
The variational method for SPDE stems from early work by E. Pardoux, \cite{pardoux_72}, which extended the variational approach to PDE developed by J. Lions, \cite{lions_11}. The variational approach to SPDE has been developed in a number of directions and we refer to \cite{lototsky_rozovsky_17,prevot_rockner_07} for a more comprehensive treatment of the literature and material. In particular, however, we mention the early works by I. Gyongy \& N. Krylov, \cite{gyongy_krylov_81_1,gyongy_krylov_81_2,gyongy_82_3} in which the authors extend the results presented here to the case of general semi-martingale drivers. Alongside the variational approach there are broadly two other, non-pathwise, frameworks for solving SPDE; the perspective of martingale problems, \cite{metivier_viot_88,pardoux_21_spde_introduction} and the semi-group approach, \cite{daprato_zabczyk_14,hairer_09}. We also wish to mention the book by F. Flandoli, \cite{flandoli_11} that focuses on stochastic perturbations of fluid equations, the collected work, \cite{dalang_koshnevisan_mueller_nualart_xiao_09}, that contains an overview of a number of different problems in the field of SPDE and the book by L. Zambotti on stochastic obstacle problems, \cite{zambotti_15}.\\ \par
These notes were written to accompany a $6$ hour mini-course delivered at the University of Oxford in April 2021. The four chapters broadly follow the material presented across four lectures, with more background, references and details to some arguments. In the remainder of this chapter we first list some conventions, notation and terminology that we use throughout, Subsection \ref{subsec:Conventions}. Then in Section \ref{sec:RandomnessPDE} we give some background on various motivations for studying SPDE problems, focussing in particular on a presentation of a simplified version of the Kraichnan model for a passive scalar in a turbulent fluid. In the second half of that section, Subsection \ref{subsec:WorkedExample}, we give a worked example of the variational method that is presented abstractly in Chapter \ref{ch:SPDE}. The remainder of the first chapter is dedicated to recalling some necessary background material and developing the theory of stochastic integration in Hilbert spaces. In Chapter \ref{ch:PDE} we return to deterministic problems and present a summary of J. Lions variational approach to coercive PDE. Chapter \ref{ch:SPDE} contains the main focus of the course, the variational approach to coercive-monotone SPDE. We present the abstract method along with an extension of the worked example given in Subsection \ref{subsec:WorkedExample}. Finally in Chapter \ref{ch:Extras} we give a short introduction to the pathwise approach to non-singular SPDE, using a stochastic Burger's equation with white noise forcing as an example.
\subsection{Conventions, Notation and Terminology}\label{subsec:Conventions}
Throughout these notes we use the conventional notation for stochastic equations, that is we use the shorthand
	\begin{equation}\label{eq:GenParSDPEIntegral}
		\dd u_t =A(u_t)\dd t + B(u_t)\dd W_t, \quad \iff \quad u_t = u_0  +\int_0^t A(u_s)\dd s + \int_0^t B(u_s)\dd W_s.
	\end{equation}
Furthermore, when we refer to \textit{weak solutions}, to problems such as \eqref{eq:GenParSDPEIntegral}, we mean a weak solution in the PDE sense, that is we ask \eqref{eq:GenParSDPEIntegral} to hold in the sense of distributions. This is not to be confused with the notion of weak solution in stochastic analysis, which typically means a solution not adapted to a given filtration. The solutions we consider will always be weak PDE solutions but strong stochastic solutions. We adopt a slightly unusual convention and write $[X]_t$ for the quadratic variation of a stochastic process, instead of the usual angle brackets. This is to avoid confusion with the standard notation for inner products on Hilbert spaces and duality pairings between Banach spaces.\\ \par 
Unless otherwise specified, all Banach and Hilbert spaces are separable and whenever we refer to a \text{basis} $(e_k)_{k\geq 1}$ of a Hilbert space we mean a complete, orthonormal system. For a Banach space $E$ and a Hilbert space $U$ we typically write $\|\,\cdot\,\|_E$ for the norm on $E$ and $\|\,\cdot\,\|_U,\,\langle \,\cdot\,,\,\cdot\,\rangle_U$ for the norm and inner product on $U$. If such an expression appears without an explicit label it will always be explained what is meant. We write $E^\ast$ for the dual of $E$, the set of bounded, linear functionals $\varphi:E \rightarrow \mbR$. We say that a sequence $(u_n)_{n\geq 1}\subset E$ converges weakly in $E$, if $(\varphi(u_n))_{n\geq1} \subset \mbR$ converges in $\mbR$ for every $\varphi \in E^\ast$. We say that a sequence $(\varphi_n)_{n\geq 1}\subset E^\ast$ converges in weak-$*$ if the evaluations $(\varphi_n(u))_{n\geq 1}\subset \mbR$ are a convergent sequence for every $u \in E$. Through the Riesz representation theorem we identify the dual of a Hilbert space with itself, so that we say a sequence $(u_n)_{n\geq 1}\subset U$ converges weakly if $(\langle h,u_n\rangle_U)_{n\geq 1}\subset \mbR$ is convergent for every $h\in U$.
We say that a map $[0,\infty) \ni t\mapsto u_t \in E$ is continuous if it is continuous with respect to the strong topology on $E$. Given a $T>0$, we equip the space of continuous maps $u:[0,T]\rightarrow E$ with the structure of a Banach space, by defining the norm
\begin{equation*}
	\|u\|_{C_TE} := \sup_{t \in [0,T]} \|u_t\|_E.
\end{equation*}
Given a measure space $(E,\mcB(E),\mu)$, a measurable space $(F,\mcB(F))$ and a measurable map $X:E\rightarrow F$, the push-forward of $\mu$ by $X$, $X\#\mu = \mu(X^{-1}(\,\,\cdot\,))$ defines a measure on $F$ and for any integrable function $g:F\rightarrow \mbR$ and $A\in \mcB(F)$, we have the identity
\begin{equation*}
	\int_A g(x)\dd X\# \mu(x) = \int_{X^{-1}(A)} g(X(x))\dd\mu(x).
\end{equation*}
Given a family of sets $\{A\}$, we write $\sigma(\{A\})$ for the smallest $\sigma$-algebra containing the collection. Given a measurable space $(E,\mcB(E))$ we write $\mcP(E)$ for the set of probability measures on $E$. We say that a sequence of probability measures $(\mu_n)_{n\geq 1}\subset \mcP(E)$ converges weakly to $\mu\in \mcP(E)$ if $\int_E f \dd\mu^n\rightarrow \int_E f\dd \mu$ for all $f:E\rightarrow \mbR$, continuous and bounded. We say that a sequence of random variables converges weakly if their laws converge weakly. The strong topology on $\mcP(E)$ is the topology induced by the total variation distance.\\ \par 
For two scalars, $s,\,t \in \mbR$, we use the notation $s\wedge t$ to denote their minimum and $s\vee t$ to denote their maximum. Typically we use this when considering functions $u:[0,\infty)\rightarrow E$, where we write $u_{t\wedge s}$ to mean the value of $u$ at the minimum of $s,\,t$. For functions of multiple variables, $u:\mbR^d\rightarrow \mbR$, we use standard multi-index notation to denote partial derivatives, i.e for $k \in \mbN^d$, $\partial_x^{k}u$ denotes the $|k|^{\text{th}}$ partial derivative in the coordinates describe by the multi-index $k =(k_1,\,\ldots,k_d)$. If we write $D^{|k|}$ we mean the vector of all possible $|k|^{\text{th}}$ order derivatives.
\section{Randomness in PDE}\label{sec:RandomnessPDE}
In very abstract terms we may think of a typical PDE problem as being constituted of the following elements; a finite dimensional metric space, $\Gamma$, a partial differential operator $L(u,Du,\ldots)$, acting on functions $u:\Gamma\rightarrow \mbR^n$, a non-linear functional, $F$, and some data $g$ defined on the boundary $\partial \Gamma$. The associated PDE problem is to find $u$, a solution in a suitable sense to the problem
\begin{equation}\label{eq:GenPDE}
	\begin{cases}
		L u = F(u,D u, D^2 u,\ldots), &\text{inside }\Gamma,\\
		u|_{\partial \Gamma} = g,
	\end{cases}
\end{equation}
The distinction between the linear part, $L$, and non-linear part, $F$, is not necessary but provides a useful framing for the equations considered in these notes.\\ \par 
In this setting there are essentially four ways in which noise can be included in a PDE problem; through the linear operator $L$, through the non-linearity $F$, through the domain, $\Gamma$ or the boundary data $u|_{\partial \Gamma}=g$. We will not explicitly consider the latter two cases; some treatment of random data equations can be found in \cite{flandoli_17, tzvetkov_16} while some treatment of equation on random domains can be found in \cite{xiu_tartakovsky_06, harbrecht_peters_siebenmorgen_16}. In fact we will mostly restrict ourselves to cases where randomness occurs in the \textit{right hand side} of the equation, $F$.\\ \par 
Motivations for studying SPDE come from many directions in both mathematics and the applied sciences. To mention only a few references, they have found applications in: constructive quantum field theory, \cite{damgaard_huffel_87,gubinelli_hofmanova_18_pde,hairer_14_Phi43,kupiainen_16}; the study of disordered media, \cite{labbe_19,allez_chouk_15}; non-linear filtering, \cite{zakai_69,pardoux_79}; population dynamics, \cite{dawson_72,fleming_75} and stochastic fluid dynamics including models of turbulence, \cite{kraichnan_68,flandoli_11}. In particular we refer to the introductions of \cite{daprato_zabczyk_14,pardoux_21_spde_introduction,dalang_koshnevisan_mueller_nualart_xiao_09} for more detailed summaries of some of these applications along with others not mentioned here. In order to motivate this course a little better we proceed to describe how an SPDE model can arise in the study of a passive scalar transported by a turbulent flow. This model is related to the Kraichnan model of turbulence, \cite{kraichnan_68}, mentioned above.
\subsection{Passive Scalar in a Turbulent Vector Field}\label{subsec:Turbulence}
We present a simplified version of a model developed by separately by R. Kraichnan, \cite{kraichnan_68} and A. Kazantsev \cite{kazantzev_68}, both of which were inspired by the work of G. Batchelor on turbulence, \cite{batchelor_00,moffatt_02}. More detailed presentations and analysis of this model and its variants can be found among the non-exhaustive references, \cite{flandoli_11,flandoli_luo_21_high, flandoli_galeati_luo_21_delayed,gess_yaroslavtsev_21, majda_kramer_99,chetrite_delannoy_gawedzki_07}. We will use SPDE related to this model as recurring examples throughout the course.\\ \par
Let us imagine that we are interested in the diffusion of a given concentration in a two dimensional fluid. We may model this situation by the simple transport equation, on $[0,\infty)\times \mbR^d$, with $\Gamma_0\subset \mbR^2$ closed, bounded, $\partial \Gamma_0$ sufficiently smooth, and $(v_t)_{t\geq 0}$ a family of divergence free vector fields,
\begin{equation}\label{eq:SimpleTransport}
	\begin{cases}
		\partial_t u_t - v_t \cdot \nabla u_t = 0, & \text{ in }\Gamma\\
		u\tzero = \frac{1}{|\Gamma_0|}\mathds{1}_{\Gamma_0}.&
	\end{cases}
\end{equation}
The solution $u_t:\Gamma \rightarrow \mbR$ describes the concentration at time $t\geq 0$. Formally speaking, due to the divergence free assumption, for any solution to \eqref{eq:SimpleTransport}, one has $\frac{1}{|\Gamma|}\int_{\Gamma} u_0(x)\dd x = \frac{1}{|\Gamma|}\int_{\Gamma} u_t(x)\dd x$. That is mass cannot be lost during the evolution, but the local density may vary. For simplicity, let us assume that $\Gamma=\mbT^2$, the two dimensional torus and as a toy example let us consider the simple vector fields,
\begin{equation*}
	v_t(x,y) = \begin{bmatrix}
		- y\\
		x
	\end{bmatrix}\theta_t,
\end{equation*}
for $t\mapsto \theta_t \in \mbR$ a differentiable function. The characteristics of $\eqref{eq:SimpleTransport}$ are the system of ODEs,
\begin{equation}\label{eq:TransportODE}
	\begin{cases}
		\frac{d}{dt}X_t =  Y_t \frac{d}{dt}\theta_t,\\
		\frac{d}{dt} Y_t = -X_t\frac{d}{dt}\theta_t.
	\end{cases}
\end{equation}
Provided the paths $t\mapsto \theta_t \in \mbR$ are sufficiently regular (and we ignore the issue of boundary data) we can hope to solve \eqref{eq:TransportODE} and using a flow map give the solution to \eqref{eq:SimpleTransport}, $u(t,x,y) = u_0(X^{-1}_t(x),Y^{-1}_t(y))$, where $(X^{-1}_t,Y^{-1}_t)$ denotes the backwards flow associated to \eqref{eq:TransportODE}, see \cite[Sec. 3.2]{evans_10}. \\ \par
Now we imagine that there is some uncertainty over the coefficient $\theta_t$. For example the concentration might be a dye diffusing far out at sea and we only have intermittent readings on the local fluid velocity. Alternatively, as in \cite{kraichnan_68}, we might intentionally introduce some randomness to model a turbulent velocity field. We can account for both situations by fixing a probability space $(\Omega,\mcF,\mbP)$ and now considering random coefficients, given by a measurable map, $(\omega,t)\mapsto \theta_t(\omega) \in \mbR$.  We can also see the map $\Omega \ni \omega \mapsto \theta(\omega)$ as a random mapping into the space of paths. If we assume that for $\mbP$-a.a. $\omega \in \Omega$, we still have $\theta(\omega) \in C^1(\mbR_+;\mbR)$, then we may replace \eqref{eq:TransportODE} with the random system of characteristics,
\begin{equation}\label{eq:RandomODE}
	\begin{cases}
		\frac{d}{dt}X_t(\omega) =  Y_t(\omega) \frac{d}{dt}\theta_t(\omega),\\
		\frac{d}{dt} Y_t(\omega) = -X_t(\omega)\frac{d}{dt}\theta_t(\omega).
	\end{cases}
\end{equation}
Again, up to solving this system, we provide a solution the random PDE,
\begin{equation}\label{eq:RandomTransport}
	\begin{cases}
		\partial_t u_t(\omega) - v_t(\omega) \cdot \nabla u_t(\omega) = 0, & \text{ in }\mbT^2\\
		u\tzero = \frac{1}{|\Gamma_0|}\mathds{1}_{\Gamma_0}.
	\end{cases}
\end{equation}
By integrating over $\omega \in \Omega$ we can now answer more nuanced questions regarding the evolution of $u$. For example we could study the expected density in a certain region, or the expected time to total mixing. The idea is that if the random coefficients are described by a sufficiently rich probability space, and of a form suitable to some statistical properties of the situation at hand, then our model can be made robust to many realised scenarios.\\ \par
To this end, we note that the model considered so far is very simple and unlikely to describe many situations of interest sufficiently well. In order to enrich the model let us first develop the form of vector fields we consider. We now define,
\begin{equation}\label{eq:DivFreeVectorFlow}
	v_t(x,y) := \sum_{k=1}^\infty \sigma_k v_k(x,y)\theta^k_t,\quad v_k(x,y) = \begin{bmatrix}
		-(y-a_k)\\
		x-b_k 
	\end{bmatrix},\quad \theta^k_t\sim \theta_t \quad \text{i.i.d,}
\end{equation}
with $(a_k,b_k)_{k\geq 1}\subset
\Gamma$ given centres of rotation, the coefficients $\sigma_k$ chosen so as to ensure the sum converges suitably, $\mbP$-a.s. (note that this is possible since we specified that $\Gamma$ be bounded) and $(\theta^k_t)_{k\geq 1}$ a family of i.i.d random scalar paths. Without loss of generality we may assume that $\mbE[\theta^k_t]=0$ for all $k,\,t$. This superposition of random, divergence free, vector fields now gives a lot of scope to accommodate an expected fluid flow and with randomness built in to account for uncertainty. In the Kraichnan model the spatial features of the vector field flow is achieved in a different manner, but the end results are qualitatively similar, see \cite[Sec. III]{chetrite_delannoy_gawedzki_07} and \cite[Sec. 5]{gess_yaroslavtsev_21} for presentations of this more detailed model.\\ \par 
While the spatial properties of the flow \eqref{eq:DivFreeVectorFlow} can now be adjusted to a large extent, the time dependence is constrained by requiring that $t\mapsto \theta^k_t$ be at least $C_1$. For example, if we choose the $\sigma_k$ so that the series is $\ell^2$ convergent, we have,
\begin{align*}
	\mbE[|v_t(x,y)-v_{s}(x,y)|^2] &= \sum_{k=1}^\infty \sigma_k^2 |v_k(x,y)|^2 \mbE[|\theta_t^k -\theta_s^k|^2] \\
	&\leq \sup_{k\geq 1} \mbE[|\partial_t \theta^k_s|^2]|t-s|^2 \|\sigma_{\,\cdot\,} v_{\,\cdot\,}(x,y)\|^2_{\ell^2(\mbN)}.
\end{align*}
That is, the values $v_t(x,y)$ and $v_s(x,y)$ must be relatively well correlated. So if we expect the vector fields at a given location to oscillate faster than linearly in time we will be unable to capture this by our current model. Following similar considerations, in \cite{kraichnan_68}, Kraichnan proposed an extension, allowing for arbitrarily short correlation times for the vector flows. In order to achieve this one can replace the $(\theta^k)_{k\geq 1}$ with less regular trajectories, a common choice being Brownian motions, $(W^k)_{k\geq 1}$, where one has,
\begin{equation*}
	\mbE[|W^k_t-W^k_s|^2] = |t-s|.
\end{equation*}
More pertinently, when we consider the characteristic equations, we replace $\frac{\dd}{\dd t}\theta_t$ with the differential $\dd W_t$ and if we compute the covariation between $\dd W_t$ and $\dd W_s$ we formally have,
\begin{equation*}
	\mbE[\dd W_t\dd W_s] = \delta(t-s).
\end{equation*}
So we see that replacing $(\theta^k)_{k\geq 1}$ with $(W^k)_{k\geq1}$ allows for a fluid flow with arbitrarily small correlation times. While this extreme may also not be physically well justified it provides more flexibility than the previous regularity requirements, gives way to a mathematically rich theory and can be further adapted to more physically relevant settings.\\ \par 
The price to pay for introducing vector fields with this more singular behaviour is that the ODEs, \eqref{eq:RandomODE}, are no longer trivially well-posed. In particular, one has that $\mbP$-a.s., the paths $t\mapsto W^k_t$ are only almost $1/2$-H\"older continuous and so $t\mapsto \dd W^k_t$ represents a proper distribution. In order to solve \eqref{eq:RandomODE} in this case then we require a definition of integrals $\int_0^t X_s\dd W^k_s$ for paths $t\mapsto X_t$ of no better regularity than $t\mapsto W_t$. This is a primary aim of stochastic analysis, which the remainder of this chapter is dedicated to. While we refer to \eqref{eq:RandomTransport} as a random PDE, we refer to,
\begin{equation}\label{eq:TransportSPDE}
	\begin{cases}
		\dd u_t - \sum_{k=1}^\infty \sigma_k v_k \cdot \nabla u_t \dd W_t^k=0,&\text{ in }\mbT^2,\\
		u\tzero = u_0,&
	\end{cases}
\end{equation}
as a stochastic PDE (SPDE). This is to emphasise the fact that one requires some additional theory in order to handle the time integral in \eqref{eq:TransportSPDE}.
\subsection{A Worked Example of the Variational Method}\label{subsec:WorkedExample}
In the final subsection of this introductory portion of the notes let us present a model example of the variational method that we will cover in more detail in the rest of the course. This example serves both as a motivation for the model described at the end of Chapter \ref{ch:SPDE} and an example walk through of the abstract method in a relatively simple case. Let $(\Omega,\mcF,(\mcF_t)_{t\geq 0},\mbP)$ be a filtered probability space, carrying  $(W_t)_{t\geq 0}$ a standard, $\mbR^d$ valued Brownian motion. We consider the semi-linear, parabolic SPDE, for simplicity on the $d$-dimensional torus, $\mbT^d$,
\begin{equation}\label{eq:BMTransportPDE}
	\begin{cases}
		\dd u =\Delta u \dd t+ \sqrt{\sigma} \nabla u \cdot \dd W_t , & \text{ on }\mbR_+\times \mbT^d,\\
		u\tzero=u_0.
	\end{cases}
\end{equation}
\begin{remark}
	Note that \eqref{eq:BMTransportPDE} has the structure of a diffusive equivalent of the transport equation \eqref{eq:SimpleTransport} with $v_t = \dd W_t$. This can be generalised to the superposition example discussed in that context $v_t = \sum_{k} \sqrt{\sigma}_kv_k \dd W^k_t$. 
\end{remark}
For the sake of presentation we restrict to the one dimensional case here, although the method applies almost \textit{mutatis mutandi} in higher dimensions. Let $(e_k)_{k\in \mbZ}$ be the standard Fourier basis of $L^2(\mbT)$. Then defining $u_{k;t} := \langle u_t,e_k\rangle$, we derive from \eqref{eq:BMTransportPDE} the system of SDEs,
\begin{equation}\label{eq:FourierCoeffSDEs}
	\dd u_{k;t} = - (2\pi  k)^2 u_{k;t} \dd t - 2\sqrt{\sigma} \pi i k u_{k;t} \dd W_t, \quad u_{k;0}= \langle u_0,e_k\rangle.
\end{equation}
The linear system, \eqref{eq:FourierCoeffSDEs}, is seen to be well-posed by a stochastic analogue of the standard Cauchy--Lipschitz theory, see \cite[Thm. 5.2.1]{oksendal_13}. Granting that for any $n\geq 1$, we have global, unique solutions, $(u_k)_{k=1}^n$, we set
\begin{equation*}
	u_{n;t} = \sum_{|k|\leq n} u_{k;t}e_k, \quad u_{n;0} := \sum_{|k|\leq n} u_{k;0}e_k,
\end{equation*}
So, at least formally,  $u_{n}$, is a smooth approximation to $u$  solving \eqref{eq:BMTransportPDE}. The question is whether we can really take the limit and define a solution to \eqref{eq:BMTransportPDE} as $u := \lim_{n \rightarrow \infty} u_n$. The key is to obtain a suitable \textit{a priori} bound on the solutions $u_n$, uniform in $n\geq 1$. By the usual Parseval's identity, it holds that
\begin{equation*}
	\|u_{n;t}\|^2_{L^2(\mbT)} = \sum_{|k|\leq n} |u_{k;t}|^2.
\end{equation*}
Applying It\^o's formula (the chain rule for It\^o processes, see \cite[Thm. 4.1.2 \& 4.2.1]{oksendal_13} and Section \ref{sec:QVandIto} below), we see that for each $k=1,\ldots n$, $|u_{k;t}|^2$ solves the SDE,
\begin{align*}
	\dd |u_{k;t}|^2 &= 2 u_{k;t}\dd u_{k;t} +  \dd [ u_k ]_t\\
	&=  -2(2\pi k)^2u_{k;t}u_{k;t} \dd t+ 2\sqrt{\sigma}  (2\pi i k u_{k;t}) u_{k;t} \dd W_t + \sigma(2\pi i k u_{k;t})^2\dd t\\
	&= 2\partial_{xx} u_{k;t} u_{k;t} \dd t + 2\sqrt{\sigma} \partial_x  u_{k;t} u_{k;t} \dd W_t + \sigma |\partial_x u_{k;t}|^2 \dd t.
\end{align*}
Here $[u_{k}]_t$ denotes the quadratic variation of the process $t\mapsto u_{k;t}$, in our case this is just $4\sigma \pi^2|u_{k;t}|^2 t$, since $[W]_t= t$. For more details see \cite[Ex. 2.17]{oksendal_13} and Section \ref{sec:QVandIto} below. Summing these equations up, or equivalently applying It\^o's formula directly to the PDE, and converting to integral form, we get the identity,
\begin{equation*}
	\|u_{n;t}\|^2_{L^2}  = \|u_{n;0}\|^2_{L^2}+\int_0^t\left(  2\langle \partial_{xx} u_{n;t},u_{n;t}\rangle  + \sigma \langle \partial_x u_{n;s},\partial_xu_{n;s}\rangle \right)\dd s + 2\sqrt{\sigma} \int_0^t \langle u_{n;s},\partial_x u_{n;s}\rangle \dd W_s.
\end{equation*}
The final integral may, at this stage be understood, as a finite dimensional stochastic integral. Integrating by parts in the first term and applying the product rule in the final term we write,
\begin{equation}\label{eq:IntroEnergyBnd1}
	\|u_{n;t}\|^2_{L^2}  = \|u_{n;0}\|^2_{L^2}-(2-\sigma)\int_0^t\|\partial_x  u_{n;s}\|^2_{L^2}\dd s + \sqrt{\sigma} \int_0^t \int_{\mbT }\partial_x (u_{n;s}(x))^2\dd x \dd W_s.
\end{equation}
We see that the last term contains the integral of a derivative on $\mbT$  and so it disappears for all $\omega \in \Omega$ such that the integral is well defined, that is,
\begin{equation}\label{eq:StochIntVanish}
	\int_0^t \int_{\mbT }\partial_x (u_{n;s}(x))^2\dd x \dd W_s(\omega) =0 \quad \text{for }\mbP\text{-a.a. } \omega \in \Omega.
\end{equation}
This is no accident and will also be true in higher dimensions if this term is of the form $\nabla \cdot (v_t u_t )\dd W_t$ with $\nabla \cdot v_t =0$. However, in many cases, even if the stochastic integral does not vanish $\mbP$-a.s. it will be zero in expectation, so long as it is well-defined. In this case a stopping time argument can be run under the expectation.\\ \par
With this in hand, we now define, for $f \in C^\infty(\mbT)$, the norm $\|f\|^2_{H^1} := \|f\|^2_{L^2} + \|\partial_x f\|^2_{L^2}$, and rewrite \eqref{eq:IntroEnergyBnd1} as,
\begin{equation}\label{eq:IntroEnergyBnd2}
	\begin{aligned}
		\|u_{n;t}\|^2_{L^2} + (2-\sigma)\int_0^t \|u_{n;s}\|^2_{H^1}\dd s &= \|u_{n;0}\|^2_{L^2}+ (2-\sigma)\int_0^t \|u_{n;s}\|^2_{L^2}\dd s.
	\end{aligned}
\end{equation}
which by \eqref{eq:StochIntVanish} holds for $\mbP$-a.a. $\omega\in \Omega$. So now, if $(2-\sigma)>0$ and we fix some $T>0$, we may apply Gr\"onwall's inequality, to obtain the estimate,
\begin{equation}\label{eq:IntroEnergyBnd3}
	\sup_{t \in [0,T]}\|u_{n;t}\|^2_{L^2} + (2-\sigma)\int_0^T \|u_{n;s}\|^2_{H^1}\dd s  \leq \|u_{n;0}\|^2_{L^2} \,e^{(2-\sigma)T}
\end{equation}
 So, assuming $(2-\sigma )>0$, we have shown,
\begin{equation*}
	u_{n}(\omega) \in C([0,T];L^2(\mbT))\cap L^2 ([0,T];H^1(\mbT)), \quad \text{for }\mbP\text{-a.a. }\omega \in \Omega.
\end{equation*} 
Furthermore, the quantity $\|u_{n}(\omega)\|_{C_TL^2}+ \|u_{n}(\omega)\|_{L^2_TH^1}$ depends only on $\|u_{n;0}\|^2_{L^2}$ - which we assume to be deterministic for ease. Now if $u_0 \in L^2(\mbT)$ we easily have,
\begin{equation*}
	\|u_{n;0}\|^2_{L^2} = \sum_{|k|\leq n} |u_{k;0}|^2 \leq \sum_{k}|u_{k;0}|^2 = \|u_0\|^2_{L^2}.
\end{equation*} 
Therefore the bound \eqref{eq:IntroEnergyBnd2}, and hence the norm $\|u_{n}(\omega)\|_{C_TL^2}+ \|u_{n}(\omega)\|_{L^2_TH^1}$ is controlled independently of $n\geq 1$. Concretely,
\begin{equation*}
	\sup_{n \geq 1} \left(\|u_{n}(\omega)\|_{C_TL^2}+ \|u_{n}(\omega)\|_{L^2_TH^1}\right) \leq \|u_0\|_{L^2} e^{(2-\sigma)T}.
\end{equation*}
This uniform bound is essentially enough to obtain a weak solution to \eqref{eq:BMTransportPDE} as the limit $u_{n}\rightarrow u$. We will detail this final step in the Chapter \ref{ch:SPDE} since it requires some functional analysis which we present in a more general setting.\\ \par 
The main takeaway, however, is that the above argument applies only for $\sqrt{\sigma}< \sqrt{2}$. Philosophically, this strategy is based on the ellipticity of the operator $u\mapsto (-\Delta\,\, \dd t+ \sqrt{\sigma} \nabla \,\,\cdot\,\dd W_t)u$ which we have seen does not hold for $\sqrt{\sigma} \geq \sqrt{2}$. In fact, using the notion of a Stratonovich integral (a stochastic integral which does obey the usual chain rule), we can re-write \eqref{eq:BMTransportPDE} in the form,
\begin{equation}\label{eq:BMStratTransport}
	\begin{cases}
		\dd u_t = \left(1-\frac{\sigma}{2} \right)\Delta  u \dd t  + \sqrt{\sigma} \,\nabla u \circ \dd W_t& \text{ on }\mbR_+\times \mbT^d,\\
		u\tzero=u_0.
	\end{cases}
\end{equation}
Now if we take $\sigma >2$ we see that the deterministic part of \eqref{eq:BMStratTransport} becomes a backwards heat equation.\\ \par
In the second half of this introductory chapter we recall some standard material and prepare some preliminaries that we need throughout the course. The majority of this background material is presented without proof and in  general we refer to \cite{brezis_11, evans_10,daprato_zabczyk_14,cohen_elliott_15} for details where they are omitted.

\section{Bounded Operators on Infinite Dimensional Spaces}
Given Banach spaces $E,\,F$ we write $L(E;F)$ for the set of bounded, linear operators $O:E\rightarrow F$. We define the operator norm of $O$ by the expression
$$\|O\|_{\Op(E;F)} := \sup_{\|x\|_E \leq 1}\|Ox\|_{F} = \sup_{x \in F} \frac{\|Ox\|_{F}}{\|x\|_{E}}.$$
$L(E;F)$ is itself a Banach space, although non-separable if both $E,\,F$ are infinite dimensional.
\begin{definition}[Finite Rank and Compact Operators]\label{def:FRandCompOps}
	Given Banach spaces, $E,\,F$ we say that an operator $O\in L(E;F)$ has \textit{finite rank} if $O(E)$ is a finite dimensional subspace of $F$. We say that an operator $O\in L(E;F)$ is \textit{compact} if for any bounded set $B\subset E$, $O(B)\subset E$ is compact.
\end{definition}
By definition, any finite rank operator is also compact. The following theorem says that in fact all compact operators can be approximated by finite rank operators.
\begin{theorem}[Representation of Compact Operators]\label{th:CompactOps}
	An operator $O\in L(E;F)$ is compact if there exists a sequence $(O_n)_{n\geq 1}\subset L(E;F)$ of finite rank operators such that $\|O-O_n\|_{\Op(E;F)} \rightarrow 0$. If $U,\,H$ are a Hilbert spaces and $O \in L(U;H)$ is compact, then there exist orthonormal families, $(\tilde{e}_k)_{k\geq 1}$ of $U$ and $(\tilde{f}_k)_{k\geq 1}$ of $H$, and a sequence of real numbers, $(\lambda_k)_{k\geq 1}$, which if they converge must converge to zero, such that
	\begin{equation}\label{eq:CompactOpRep}
		O x = \sum_{k=1}^\infty \lambda_k \langle x,\tilde {e}_k\rangle_U \tilde{f}_k, \quad \text{for all } x\in U.
	\end{equation}
\end{theorem}
\begin{proof}
	See \cite[Ch. 6]{brezis_11} in particular Corollary 6.2 and Theorem 6.8.
\end{proof}
Note that the families $(\tilde{e}_k)_{k\geq 1},\,(\tilde{f}_k)_{k\geq 1}$ need not be complete, in the sense that they need not form bases of $U,\,H$. Consider $O$ a finite rank operator, for example. If $U,\,H$ are Hilbert spaces and $O \in L(U;H)$, then there is an operator $O^\ast \in L(H;U)$ such that,
\begin{equation*}
	\langle O x,h\rangle_{H} = \langle x,O^\ast h\rangle_{U},\quad \text{ for all } x \in U,\, h\in H.
\end{equation*}
The operator $O^\ast$ is known as the adjoint (or Hermitian adjoint) of $O$. If $O \in L(H):= L(H;H)$ is such that $O=O^\ast$, then $O$ is said to be self-adjoint. If $O$ is a compact, self-adjoint operator then there exists a representation of the form \eqref{eq:CompactOpRep} with $(f_k)_{k\geq 1}=(e_k)_{k\geq 1}$ an orthonormal basis of  $H$. See \cite[Thm. 6.11]{brezis_11}.\\ \par
An operator $O\in L(H)$ is said to be positive if for any $h\in H$,
\begin{equation*}
	\langle Oh,h\rangle \geq 0.
\end{equation*}
The class of non-negative, symmetric operators $O\in L(H)$ will be of particular importance. To every non-negative operator, $O\in L(H)$, there exists a second non-negative operator, $O^{1/2},\in L(H)$, known as the square root of $O$, such that $O=O^{1/2}(O^{1/2})^\ast$. See \cite[Thm. 3.1.28]{lototsky_rozovsky_17}. If $O\in L(H)$ is symmetric and non-negative, then $O^{1/2}=(O^{1/2})^\ast \in L(H)$.
\subsection{Trace Class and Hilbert--Schmidt Operators}\label{subsec:TraceAndHSOps}
We refer to \cite[App. C]{daprato_zabczyk_14} for more details on much of the material in this subsection.
\begin{definition}[Trace Class Operators]\label{def:TraceClass}
	Let $E,\,F$ be Banach spaces and $T\in L(E;F)$. We say that $T$ is a \textit{trace class} (or \textit{nuclear}) operator, if there exist two sequences, $(a_k)_{k\geq 1}\subset E^\ast,\,(b_k)_{k\geq 1} \subset F$, such that 
	\begin{equation*}
		\sum_{k=1}^\infty \|a_k\|_{E^*}\|b_k\|_{F} <\infty,
	\end{equation*}
	and for any $x \in E$,
	\begin{equation*}
		Tx = \sum_{k=1}^\infty a_k(x) b_k.
	\end{equation*}
	We write $L_1(E;F)$ for the set of trace class operators and equip this space with the norm,
	\begin{equation}\label{eq:TraceNorm}
		\|T\|_{L_1} := \inf \left\{ \sum_{k=1}^\infty \|a_k\|_{E^\ast}\|b_k\|_{F}<\infty\,:\, Tx = \sum_{k=1}^\infty a_k(x)b_k \right\},
	\end{equation}
	under which $L_1(E;F)$ becomes a Banach space. 
\end{definition}
Let $U,\,H$ be Hilbert spaces and $T\in L_1(U;H)$, then by making the identification, $U^\ast\cong U$, there exist sequences $(a_k)_{k\geq 1} \subset U,\, (b_k)_{k\geq 1} \subset H$ such that, for all $x \in U$,
\begin{equation*}
	Tx = \sum_{k=1}^\infty \langle a_k,x\rangle b_k.
\end{equation*}
If $T\in L_1(H):=L_1(H;H)$, and $(e_k)_{k\geq 1}\subset H$ is a basis of $H$, then we define the trace of $T$ by the expression,
\begin{equation}
	\Tr T := \sum_{k=1}^\infty \langle Te_k,e_k\rangle \in \mbR.
\end{equation}
It is relatively straightforward to check that the definition of $\Tr T$ is independent of the choice of basis $(e_k)_{k\geq 1}\subset H$, see \cite[Prop. C.1]{daprato_zabczyk_14}, and that $|\Tr T|\leq \|T\|_{L_1}$. The following proposition says that the equality $|\Tr T|= \|T\|_{L_1}$ holds when $T\in L_1(H)$ is symmetric and non-negative. In this case, the property of finite trace is equivalent to being trace class.
\begin{proposition}[Symmetric Non-negative Trace Class Operators]\label{prop:SymmetricNonNegativeTraceClass}
	Let $T\in L(H)$, be symmetric and non-negative. Then $T\in L_1(H)$ if and only if, for some (equivalently any) basis $(e_k)_{k\geq 1}\subset H$, 
	\begin{equation*}
		|\Tr T| =\Tr T = \sum_{k=1}^\infty \langle Te_k,e_k\rangle <\infty.
	\end{equation*}
	Furthermore, in this case, $\Tr T = \|T\|_{L_1}$.
\end{proposition}
\begin{proof}
	Let $T^{1/2}$, denote the non-negative square root of $T$, and $(e_k)_{k\geq 1}$ be a basis of $H$, so that for any $x \in H$,
	\begin{equation*}
		T^{1/2}x = \sum_{k=1}^\infty \langle T^{1/2}x,e_k\rangle e_k.
	\end{equation*}
	Therefore, for any $x \in H$ and $N\geq 1$,
	\begin{align*}
		\left\| T^{1/2}x - \sum_{k=1}^N \langle T^{1/2}x,e_k\rangle e_k \right\|_H^2 = \sum_{k=N+1}^\infty |\langle T^{1/2}x,e_k\rangle|^2 &\leq |x|^2\sum_{k=N+1}^\infty \|T^{1/2}e_k\|^2_{H} \\
		&= |x|^2\sum_{k=N+1}^\infty \langle T e_k,e_k\rangle\\
		&\leq |x|^2 \Tr T. 
	\end{align*}
	It follows that $T^{1/2}$ is the limit in operator norm of a sequence of finite rank operators and so $T^{1/2}$ is a compact operator and hence $T= T^{1/2}T^{1/2}$ is also compact. Therefore, by Theorem \ref{th:CompactOps}, there exists a basis $(f_k)_{k\geq 1} \subset H$ and a sequence of non-negative, real numbers, $(\lambda_k)_{k\geq 1}\subset \mbR_+$, accumulating to $0$, such that
	\begin{equation}\label{eq:TraceClassRep}
		Tx =\sum_{k=1}^\infty \lambda_k \langle x,f_k\rangle f_k.
	\end{equation}
	Using this formula, we have that $\langle T e_k,e_k\rangle = \sum_{k=1}^\infty \lambda_k \langle e_k,f_k\rangle^2$ and so,
	\begin{equation*}
		\Tr T = \sum_{k=1}^\infty \langle Te_k,e_k\rangle = \sum_{k=1}^\infty \sum_{l=1}^\infty \lambda_k |\langle  e_k,f_l\rangle|^2 = \sum_{k=1}^\infty \lambda_k <\infty.
	\end{equation*}
	Thus we have shown that $T\in L_1(H)$, with sequences $(a_k)_{k\geq 1} = (f_k)_{k\geq 1}$ and $(b_k)_{k\geq 1}= (\lambda_k f_k)_{k\geq 1}$, and that $\Tr T \geq \|T\|_{L_1}$. Hence $\Tr T = \|T\|_{L_1}$. 
\end{proof}
\begin{definition}[Hilbert--Schmidt Operator]\label{def:HSOperator}
	Let $U,\,H$ be Hilbert spaces with bases $(e_k)_{k\geq 1},\,(f_k)_{k\geq 1}$ respectively. Then $T\in L(U;H)$ is said to be a \textit{Hilbert--Schmidt} operator, if,
	\begin{equation}\label{eq:HSNorm}
		\|T\|^2_{L_2} := \sum_{k=1}^\infty \|T e_k\|^2_{H} = \sum_{k=1}^\infty \sum_{l=1}^\infty \langle Te_k,f_l\rangle_H^2 = \sum_{l=1}^\infty \|T^\ast f_l\|_U^2 \,<\,\infty.
	\end{equation}
\end{definition}
\begin{remark}\label{rem:HilbSchmidtBasis}
		The space $L_2(U;H)$ is a separable Hilbert space when equipped with the norm $\|T\|_{L_2}$ and scalar product,
	\begin{equation*}
		\langle T,S\rangle = \sum_{k=1}^\infty \langle Te_k,Se_k\rangle_H, \quad \text{ for }T,\,S \in L_2(U;H),
	\end{equation*}
	and the sequence $(e_k\otimes f_k)_{k\geq 1}\subset U\otimes H$ provides a complete orthonormal basis for $L_2(U;H)$. It follows that $L(U;H)$ is densely embedded in $L_2(U;H)$.
\end{remark}
\begin{remark}
	It follows from \eqref{eq:HSNorm} that the definition of $\|T\|_{L_2}$ is independent of the choice of bases for $U,\,H$ and that $\|T^\ast\|_{L_2}= \|T\|_{L_2}$.
\end{remark}
We see that $T\in L(H)$ is trace class if and only if $|T|^{1/2} := (T T^\ast)^{1/2} \in L_2(H)$. Furthermore, since both trace class and Hilbert--Schmidt operators are seen to be compact, from Theorem \ref{th:CompactOps}, for any $T \in L_2(U;H)$ and $TT^\ast \in L_1(H)$, there exist bases $(e_k)_{k\geq 1}\subset U$ and $(f_k)_{k\geq 1}\subset H$ and a sequence of positive real numbers, $(\lambda_k)_{k\geq 1}$ such that, for all $x \in U$, $h\in H$,
\begin{equation}\label{eq:HSandTraceClassRep}
	\begin{aligned}
		Tx &= \sum_{k\geq 1} \sqrt{\lambda_k} \langle x,e_k\rangle_Uf_k,\\
		(TT^\ast)h &= \sum_{k\geq 1}\lambda_k \langle h,f_k\rangle_H f_k .
	\end{aligned}
\end{equation}
\subsection{Pseudo Inverses of Linear Operators}
In the following sections we turn to the study of Gaussian random variables and Wiener processes taking values in infinite dimensional spaces. Instead of being described by a covariance matrix these random variables are described by covariance operators. In developing the theory of stochastic integration with respect to generalised Wiener processes, Subsection \ref{subsec:GenStochInt}, we will be particularly concerned with the pre-images of these operators. It will therefore be useful to recall some facts regarding the images and pre-imagines of linear operators and their pseudo-inverses. Throughout we let $U,\,H$ be two Hilbert spaces and recall that $L(U;H)$ is the set of bounded linear operators from $U$ to $H$. 
\begin{lemma}\label{lem:ImagesAndPreImages}
	Let $O\in L(U;H)$ and $R\geq 0$. Then,
	\begin{itemize}
		\item  the set $ O \left(\overline{B_U(0,R)}\right):=\{  O x\in H\,:\, x \in U,\, \|x\|_{U}\,\leq\, R\}$ is convex and closed,
		\item the set $O^{-1}\left(\overline{B_H(0,R)}\right):=\{ x\in U\,:\, \|Ox\|_H\,\leq\, R \}$ is convex and closed.
	\end{itemize}
\end{lemma}
\begin{proof}
	Convexity of both sets follows from linearity of $O$. We only show that $O \left(\overline{B_U(0,R)}\right)$ is closed, since it follows directly by continuity of $O$ that $O^{-1}\left(\overline{B_H(0,R)}\right)$, as the pre-image of a closed set, is closed. Recall that a convex subset of a Hilbert space is closed if and only if it is weakly compact. By the Banach--Alaoglou theorem, any bounded set in $U$ is weakly compact, and since $O$ is strongly, and therefore also weakly, continuous, the image of a weakly compact set under $O$ is also weakly compact in $H$. It therefore follows from convexity that $O \left(\overline{B_U(0,R)}\right)$ is also closed.
\end{proof}

For $h \in H\setminus \{0\}$, let us define the set,
\begin{equation*}
	O^{-1}(h):= \{ x \in U\,:\, Ox = h\}.
\end{equation*}
It follows from Lemma \ref{lem:ImagesAndPreImages} that $O^{-1}(h)$ is convex and closed. By continuity of $O$ it also does not contain $0 \in U$. Therefore, there is a unique element of $O^{-1}(h)$ that minimizes the distance to $0\in U$. This leads to the following definition.
\begin{definition}[Pseudo-inverse]\label{def:PseudoInverse}
	Let $O\in L(U;H)$ and the map $H\setminus\{0\}\ni h\mapsto O^{-1}h$ be defined by setting,
	\begin{equation*}
		O^{-1} h :=  \left\{y \in O^{-1}(h)\,:\, \|y\|_{U}\leq \|x\|_{U},\, \forall\, x \in O^{-1}(h)  \right\}.
	\end{equation*}
	Then $O^{-1}\in L(H\setminus\{0\};U )$ is the \textit{pseudo-inverse} of $O$.
\end{definition}
Recall that for a linear operator $O\in L(U;H)$, we define $\text{Ker}(O):= \{ x\in U\,:\, Ox=0\in H \}$ and $\text{Ker}(O)^\perp = \{ y\in U\,:\, \langle x,y\rangle_U =0,\forall\, x \in \text{Ker}(O) \} = O^\ast(H)$.
%
%
\begin{proposition}
	Let $O\in L(U;H)$ and $O^{-1}$ be its pseudo-inverse.
	\begin{enumerate}[label=\roman*)]
		\item The space $O(U)$ is a Hilbert space when equipped with the inner product,
		\begin{equation*}
			\langle h,g\rangle_{O(U)} := \langle O^{-1}h,O^{-1}h\rangle_{U},\quad \text{ for all } h,\,g \in O(U).
		\end{equation*}
		\item Let $(e_k)_{k\geq 1}$ be a basis of $\text{Ker}(O)^\perp$. Then, $(Oe_k)_{k\geq 1}$ is a basis of $(O(U),\langle \,\cdot\,,\,\cdot\,\rangle_{O(U)})$.
	\end{enumerate}
\end{proposition}
\begin{proof}
	Both conclusions are a consequence of the fact that $O : \text{Ker}(O)^\perp \rightarrow O(U)$ is an isometry when $\text{Ker}(O)^\perp$ is equipped with the inner product of $U$ and $O(U)$ is equipped with the inner product $\langle \,\cdot\,,\,\cdot\,\rangle_{O(U)}$ defined above.
\end{proof}
\begin{proposition}\label{prop:OpImageCompare}
	Let $(U_1,\langle \,\cdot\,,\,\cdot\,\rangle_{1}),\,(U_2,\langle \,\cdot\,,\,\cdot\,\rangle_{2})$ be two Hilbert spaces and $O_1\in L(U_1,H)$, $O_2\in L(U_2,H)$. Then the following both hold,
	\begin{enumerate}[label=\roman{*})]
		\item \label{it:OpInclusion} If there exists a $c\geq 0$ such that for all $h\in H$, $\|O^{\ast}_1h\|_{1} \leq c\|O^{\ast}_2h\|_2$ then
		$$\{ O_1x \,:\, x \in U_1,\, \|x\|_{1}\leq 1\}\subseteq \{ O_2y \,:\, y \in U_2,\, \|x\|_{2}\leq 1\}.$$
		In particular, $\text{Im}O_1\subseteq \text{Im}O_2$.
		\item \label{it:OpIsometry} If  for all $h\in H$, $\|O_1^\ast h\|_1=\|O_2^\ast h\|_2$, then $\text{Im}O_1=\text{Im}(O_2)$ and for all $h \in \text{Im}O_2$,  $\|O^{-1}_1h\|_1=\|O^{-1}_2h\|_{2}$. 
	\end{enumerate}
\end{proposition}
\begin{proof}
	See the proof of \cite[Prop. C.0.5]{prevot_rockner_07}
\end{proof}
\begin{corollary}\label{cor:SquareRootOfSquare}
	Let $O\in L(U;H)$ and set $Q = OO^\ast \in L(H)$. Then it holds that,
	\begin{equation*}
		\text{Im}Q^{1/2}=\text{Im}O,\quad \text{and}\quad \|Q^{-1/2}h\|_H = \|O^{-1}h\|_{U},\quad \text{for all}\quad h\in \text{Im}O,
	\end{equation*}
	where $Q^{-1/2}$ is the pseudo-inverse of $Q^{1/2}$.
\end{corollary}
\begin{proof}
	First note that $Q^{1/2}$ is well defined since $OO^\ast$ is always a non-negative operator. Furthermore, since $Q^{1/2}$ is symmetric, for all $h\in H$, we have that
	\begin{equation*}
		\|(Q^{1/2})^\ast h\|^2_H = \|Q^{1/2}h\|^2_{H} = \langle Qh,h\rangle_H = \langle OO^\ast h,h\rangle_H= \|O^\ast h\|^2_U.
	\end{equation*}
	Therefore, the conclusion follows by Item \ref{it:OpIsometry} of Proposition \ref{prop:OpImageCompare}.
\end{proof}
\section{Lebesgue and Bochner Integration}\label{sec:LebBochInt}
We very briefly recall the construction of the Lebesgue and Bochner integrals. We will build the stochastic integral in a similar manner as we build the Lebesgue integral below. The expectation of infinite dimensional random variables gives a recurring example of a Bochner integral in our analysis.\\ \par
Let $(E,\mcB(E),\mu)$ be a topological, measure space, equipped with its Borel sigma algebra, $\mcB(E)$. We define the set of simple functions, $\mcE$, to be all functions, $g:E\rightarrow \mbR$ of the form
\begin{equation*}
	g = \sum_{k=1}^N c_k \mathds{1}_{A_k}, \quad c_k \in \mbR, \,\, A_k \in \mcB(E),\,N\in \mbN.
\end{equation*}
Define the integral of $g\in \mcE$ by the expression
\begin{equation*}
	\int_E g(x) \dd \mu(x) = \sum_{k=1}^N c_k \mu(A_k) \,\in \,\mbR,
\end{equation*}
and then note that for any $p\in [1,\infty]$, the space $\mcE$ equipped with the norm,
\begin{equation*}
	\|g\|_{L^p} := \begin{cases}
		\displaystyle	\sum_{k=1}^N |c_k|^p|A_k,& p \in [1,\infty)\\
		\displaystyle	\sup_{ k=1,\ldots,N} |c_k|, & p=\infty,
	\end{cases}
\end{equation*}
defines a complete, normed, vector space. We set $L^p(\mu)$ to be the completion of $\mcE$ under the norm $\|\,\cdot\,\|_{L^p}$ and the Lebesgue integral of a measurable function $f:E \rightarrow \mbR$ as the equivalence class of limit points of all simple approximating sequences $g_n\rightarrow f \in L^p(\mu)$. Extending to vector valued functions can be done component by component.\\ \par
A similar notion of integral can be defined for maps taking values in infinite dimensional spaces, this is the Bochner integral. Let $(V,\|\,\cdot\,\|_V)$ be a Banach space and define by $\mcE_V$ the set of $V$ valued simple maps $g :E\rightarrow V$ of the form
\begin{equation*}
	g = \sum_{k=1}^N c_k \mathds{1}_{A_k},\quad c_k \in V, A_k \in \mcB(E),\, N\in \mbN,
\end{equation*}
and we now set
\begin{equation*}
	\int_E g(x)\dd \mu(x) = \sum_{k=1}^N c_k \mu(A_k) \,\in V.
\end{equation*}
This is \textit{mutatis mutandi} the same construction as for the Lebesgue integral. However, in order to extend to general measurable maps $f :E\rightarrow V$ we now say that $f$ is Bochner integrable if there exists a sequence of simple function $(g_n)_{n\geq 1} \subset \mcE_V$ such that 
\begin{equation*}
	\lim_{n\rightarrow \infty} \int_E \|f(x)-g_n(x)\|_V \dd \mu (x)= 0,
\end{equation*}
where each integral is understood as the Lebesgue integral of the map $E \ni x \mapsto \|f(x)-g_n(x)\|_{V}\,\in \mbR$.\\ \par 
A version of the dominated convergence theorem holds for the Bochner integral.
\begin{theorem}[DCT for Bochner Integrals]\label{th:DCTBochner}
	Let $(E,\mcB(E),\mu)$ be a topological measure space and $V$ be a Banach space, $f:E\rightarrow V$ be Bochner integrable and $(f_n)_{n\geq 1}$ a sequence of Bochner integral maps, such that for $\mu$-a.e $x\in E$, $f(x)=\lim_{n\rightarrow \infty} f_n(x)$. Then, if there exists a real, integrable function $g \in L^1(\mu;\mbR)$ such that for $\mu$-a.e $x\in E$, $\|f_n(x)\|_V\leq g(x)$, it holds that
	\begin{equation*}
		\lim_{n\rightarrow \infty} \int_E\|f-f_n\|_V\,\dd\mu  =0,
	\end{equation*}
	and so
	\begin{equation*}
		\int_E f_n \dd \mu \rightarrow \int_E f \dd \mu.
	\end{equation*}
\end{theorem}
\begin{proof}
	The first assertion is a consequence of the DCT for Lebesgue integrals, since by assumption $\|f_n(x)\|_V\rightarrow \|f(x)\|_V\in \mbR$ for $\mu$-a.e $x\in E$. The second assertion follows from the first and the discussion above regarding construction of the Bochner integral.
\end{proof}
\section{Probability and Random Variables}\label{subsec:Probabillity}
Let $(\Omega,\mcF,\mbP)$ be a probability space, consisting of an abstract measurable space $(\Omega,\mcF)$, equipped with $\mbP$ a probability. We also assume we are given, $(E,\mcB(E))$ a measurable, separable, Banach space, with $\mcB(E)$ the Borel sigma algebra induced by the norm $\|\,\cdot\,\|_E$. We will assume both these to be in place without further comment. We say that a measurable map $X:\Omega\rightarrow E$ is an $E$ valued random variable. We write $\mcL(X) := X_\#\mbP \in \mcP(E)$ that is, for any $A \in \mcB(E)$,
\begin{equation*}
	\mcL(X)(A) = \mbP[\omega \in \Omega \,:\, X(\omega) \in A].
\end{equation*}
We say that a property holds $\mbP$-a.s. if the $\mbP$ measure of the set where the property does not hold is zero. The following proposition says that the Borel $\sigma$-algebra, $\mcB(E)$ is generated by the cylinder sets of $E^\ast$. We recall that given a family of sets $A \in 2^E$, the notation $\sigma(A)$ denotes the smallest $\sigma$-algebra containing $A$.
\begin{proposition}\label{prop:CylinderSets}
	Let $(E,\mcB(E))$ be a separable, Banach space equipped with its Borel $\sigma$-algebra. Firstly, there exists a sequence $(\varphi_n)_{n\geq 1} \subset E^\ast$ such that,
	\begin{equation}\label{eq:NormApproxSequence}
		\|x\|_E = \sup_{n\geq 1} |\varphi_n(x)|,\quad \text{for any}\quad x \in E.
	\end{equation}
	As a result, 
	\begin{equation}\label{eq:CylinderBorel}
		\mcB(E)= \sigma(\{ x \in E\,:\, \varphi(x)\leq R,\,R \in \mbR,\, \varphi \in E^\ast \}),
	\end{equation}
	and so $X$ is an $E$ valued random variable, if and only if $\varphi(X)$ is a real random variable for all $\varphi\in E^\ast$.
\end{proposition}
\begin{proof}
	We show \eqref{eq:NormApproxSequence} first. Let $(x_n)_{n\geq 1}\subset E$ be a dense set. By the Hahn--Banach theorem, \cite[Thm. 1.1]{brezis_11}, for all $n\geq 1$ there exists a $\varphi_n\in E^\ast$ such that $\varphi_n(x_n)=\|x_n\|_E$ and $\|\varphi_n\|_{E^\ast}=1$. We show that the sequence $(\varphi_n)_{n\geq 1}$ has the property \eqref{eq:NormApproxSequence}. Fix $x \in E$ and since $\|\varphi_n\|_{E^\ast}\leq 1$ it holds that $|\varphi_n(x)|\leq \|x\|_E$ for each $n\geq 1$ and therefore $\sup_{n\geq 1}|\varphi_n(x)|\leq \|x\|_E$. On the other hand, for any $\varepsilon>0$, there exists an $x_n$, such that $\|x-x_n\|_E <\varepsilon$ and so $|\varphi_n(x)-\|x_n\|_E|= |\varphi_n(x-x_n)|\leq \varepsilon$. Therefore, $|\varphi_n(x)|>\|x\|_E -2\varepsilon$, and since $\varepsilon$ was arbitrary we conclude that \eqref{eq:NormApproxSequence} holds.\\ \par
	We use \eqref{eq:NormApproxSequence} to prove \eqref{eq:CylinderBorel}. It follows from \eqref{eq:NormApproxSequence}, that for $a \in E$,\, $R>0$, 
	\begin{equation*}
		\overline{B}_E(a,R) = \{ x \in E\,:\, \|x-a\|_{E}\leq R \} = \bigcap_{n\geq 1}\{ x \in E\,:\, |\varphi_n(x)-\varphi_n(a)|\leq R \}.
	\end{equation*}
	Therefore, the smallest $\sigma$-algebra containing all sets of the form $\{  x \in E\,:\,\varphi(x)\leq R,\, R\in \mbR, \,\varphi\in E^\ast\}$ contains the closed balls of $E$ and thus contains $\mcB(E)$. On the other hand, every $\varphi \in E^\ast$ is a continuous map from $(E,\mcB(E))\rightarrow (\mbR,\mcB(\mbR))$ and so is measurable. Therefore, the reverse inclusion holds from which \eqref{eq:CylinderBorel} follows. Since a map $X:\Omega\rightarrow E$ is an $E$ valued random variable if and only if it is measurable with respect to $\mcF$ and $\mcB(E)$, the final assertion follows since \eqref{eq:CylinderBorel} is equivalent to the statement $\mcB(E)= \sigma(\{\varphi \in E^\ast \})$.
\end{proof}
The following corollary of Proposition \ref{prop:CylinderSets} allows us to turn some questions regarding $E$ valued random variables into questions regarding real valued random variables through the dual mapping $\omega \mapsto \varphi(X(\omega))$ for $\varphi \in E^\ast$.
\begin{corollary}\label{cor:EValuedRVs}
	Let $(E,\mcB(E))$ be a separable Banach space equipped with its Borel $\sigma$-algebra. Then the following all hold,
	\begin{enumerate}[label = \roman*)]
		\item \label{it:VectoSpace}The set of $E$ valued random variables is a vector space. That is, if $X,\, Y$ are $E$ valued random variables, then for any $\alpha,\,\beta \in \mbR$, $\alpha X + \beta Y$ is an $E$ valued random variable.
		\item \label{it:WeaklyClosed}If $(X_n)_{n\geq 1}$ is a sequence of $E$ valued random variables such that $X_n(\omega)\rightharpoonup X(\omega)\in E$, $\mbP$-a.s. then $X:\Omega\rightarrow E$ is an $E$ valued random variable.
		\item \label{it:NormRealRV} If $X$ is an $E$ valued random variable, then $\Omega \ni \omega\mapsto \|X(\omega)\|_E$ is a real random variable.
		\item \label{it:Modification} If $X,\,Y$ are two $E$ valued random variables. Then $X=Y$, $\mbP$-a.s. if and only if $\varphi(X)=\varphi(Y)$ $\mbP$-a.s. for all $\varphi\in E^\ast$.
	\end{enumerate}
\end{corollary}
\begin{proof}
	Items \ref{it:VectoSpace} and \ref{it:WeaklyClosed} follow from Proposition \ref{prop:CylinderSets} and the fact that both properties hold for real random variables. Item \ref{it:NormRealRV} follows from the fact that by definition $\|\,\cdot\,\|_E \in E^\ast$. In order to show item \ref{it:Modification}, by linearity of the space $E$ it suffices to show that if $\varphi(X)=0\in \mbR$, $\mbP$-a.s. for all $\varphi\in E^\ast$, then $X=0\in E$, $\mbP$-a.s.. Let, $(\varphi_n)_{n\geq 1}$ be the sequence from \eqref{eq:NormApproxSequence}, then if by assumption $\varphi_n(X)=0$, $\mbP$-a.s. it holds that $\sup_{n\geq 1}|\varphi_n(X)|=\|X\|=0$, $\mbP$-a.s., which implies the conclusion.
\end{proof}
We say that an $E$ valued random variable, $X$, is simple if there exists an $N\geq 1$ and sequences $(x_k)_{k=1}^N\subset E$, $(A_k)_{k=1}^N\subset \mcF$, such that,
\begin{equation}\label{eq:SimpleRV}
	X(\omega) = \sum_{k=1}^N x_k \indic_{A_k}(\omega), \quad \text{for}\quad \mbP\text{-a.a. } \omega \in \Omega.
\end{equation}
It follows from, Corollary \ref{cor:EValuedRVs}, that for any sequence of simple random variables $(X_n)_{n\geq 1}$ converging, weakly or strongly, to $X\in E$, defines an $E$ valued random variable. Conversely, it also holds that to any $E$ valued random variable, there exists a sequence of simple random variables converging strongly to it. We give this result in the more general setting of  separable metric spaces.
\begin{proposition}\label{prop:SimpleApprox}
	Let $(E,\rho_E)$ be a separable metric space and $X$ be an $E$ valued random variable. Then there exists a sequence of simple random variables $(X_n)_{n\geq 1}$, of the form, \eqref{eq:SimpleRV}, such that $\lim_{n\rightarrow \infty}\rho_E(X(\omega),X_n(\omega))\rightarrow 0$, $\mbP$-a.s..
\end{proposition}
\begin{proof}
	Let $(x_k)_{k\geq 1}\subset E$ be a countable dense set. Then for every $n\geq 1$ and $\mbP$-a.a. $\omega \in \Omega$, we define
	\begin{align*}
		\rho_n(\omega) &= \min\{ \rho(X(\omega),x_k),\,k=1,\ldots,n \},\\
		k_n(\omega) &= \min\{ k\leq n\,:\, \rho_n(\omega)= \rho(X(\omega),x_k) \},\\
		X_n(\omega)&= x_{k_n(\omega)}.
	\end{align*}
	The random variables $X_n$ are simple, since each $X_n$ takes values only in the set $\{ x_1,\ldots,x_n\}$. Furthermore, since the sequence $(x_k)_{k\geq 1}$ is dense, one has that $\lim_{n\rightarrow \infty} \rho_n(\omega)=0$, $\mbP$-a.s. Since $\rho_n(\omega) = \rho_E(X(\omega),X_n(\omega))$, the conclusion follows.
\end{proof}
For a simple random variable we define the Bochner integrals,
\begin{equation*}
	\mbE[X]:=\int_{\Omega} X(\omega) \dd \mbP(\omega) = \sum_{k=1}^N \mbE[x_k] \mbP[A_k],
\end{equation*}
and
\begin{equation*}
	\mbE[\|X\|_E] = \sum_{k=1}^N \mbE[\|x_k\|_E] \mbP[A_k].
\end{equation*}
By completeness of $(E,\,\|\,\cdot\,\|_E)$, an $E$ valued random variable is Bochner integrable if and only if 
\begin{equation*}
	\int_\Omega \|X(\omega)\|_E \dd \mbP(\omega) <\infty.
\end{equation*}
Therefore, using Proposition \ref{prop:SimpleApprox} we may extend the Bochner integral, $\mbE[X]$, to all $E$ valued random variables, as the limit of the Bochner integral along a sequence of approximating simple random variables. For $p\in [1,\infty]$ we define the spaces $L^p(\Omega;E)$ as the completions of the simple random variables under the norms,
\begin{equation*}
	\|X\|_{p}:= \begin{cases}
		\mbE\left[\|X\|_E^p\right]^{\frac{1}{p}}, & \text{ if } p\in [1,\infty),\\
		\displaystyle \esssup_{\omega \in \Omega} \|X(\omega)\|_E, & \text{ if } p =\infty.
	\end{cases}
\end{equation*}
Note that $\mbE[X] \in E$ while $\mbE[\|X\|_E^p] \in \mbR$. For concision and where it will not cause confusion we sometimes write the Bochner integral as
\begin{equation*}
	\int_{\Omega} X \dd \mbP := \int_{\Omega}X(\omega)\dd \mbP(\omega).
\end{equation*}
The Bochner integral of a random variable obeys analogues of the dominated convergence and monotone convergence theorems for the Lebesgue integral.
Let $(F,\|\,\cdot\,\|_E)$ be a second separable Banach space and consider a linear operator $A:D(A) \subseteq E \rightarrow F$. Recall that we say $A$ is closed if the set,
\begin{equation*}
	\text{graph}(A):= \{ (x,y) \in E\times F\,:\, x\in D(A),\, y = Ax \},
\end{equation*}
is closed in the product space $E\times F$. We equip $D(A)$ with the graph norm, $\|x\|_{D(A)}:= \|x\|_E+\|Ax\|_F$. We have the following result regarding $D(A)\subseteq E$ valued random variables.
\begin{theorem}\label{th:OpExpectationCommute}
	Let $E,\,F$ be separable Banach spaces and $A:D(A)\subseteq E\rightarrow F$ be a closed operator such that $D(A)\in \mcB(E)$. If $X:\Omega\rightarrow E$ is $\mbP$-a.s. $D(A)$ valued, then $AX$ is an $F$ valued random variable and $X$ is a $D(A)$ valued random variable (seen as a Banach space equipped with the graph norm). If $X$ is Bochner integrable and
	\begin{equation*}
		\int_{\Omega}\|AX(\omega)\|_E\dd \mbP(\omega)<\infty,
	\end{equation*}
	then
	\begin{equation}\label{eq:OpExpectationCommute}
		A \int_{\Omega} X(\omega)\dd \mbP(\omega) = \int_\Omega X(\omega) \dd \mbP(\omega).
	\end{equation}
\end{theorem}
\begin{proof}
	The first statements essentially follow from the definitions, using the fact that $(D(A),\|\,\cdot\,\|_{D(A)})$ is itself a Banach space. To prove \eqref{eq:OpExpectationCommute} consider a sequence of simple random variables $(X_n)_{n\geq 1}\subset D(A)$ approximating $X$ in $D(A)$, that is
	$$\|X-X_n\|_{D(A)}:= \|X-X_n\|_{E} + \|AX - AX_n\|_F\rightarrow 0.$$
	It follows directly that we have
	\begin{equation*}
		\int_\Omega \|X(\omega)-X_n(\omega)\|_{D(A)}\dd \mbP(\omega) \rightarrow 0.
	\end{equation*} 
	Thus, by definition of the Bochner integral, we have that both
	\begin{equation*}
		\int_{\Omega} X_n(\omega)\dd \mbP(\omega) \rightarrow \int_{\Omega}X(\omega)\dd \mbP(\omega), \qquad \int_{\Omega} AX_n(\omega)\dd \mbP(\omega) \rightarrow \int_{\Omega} AX(\omega)\dd \mbP(\omega).
	\end{equation*}
	However, it also follows from the definition of the Bochner integral of simple random variables that
	\begin{equation*}
		\int_{\Omega} AX_n(\omega)\dd \mbP(\omega)= A\int_{\Omega} X_n(\omega)\dd \mbP(\omega),
	\end{equation*}
	therefore \eqref{eq:OpExpectationCommute} follows from the closedness of $A$.
\end{proof}
\begin{remark}\label{rem:LinOperatorExpecationCommute}
	It follows from Theorem~\ref{th:OpExpectationCommute} that $\varphi \in L(E,F)$ and if $X$ is $E$ valued random variable that is Bochner integrable then,
	\begin{equation*}
		\varphi \left( \int_{\Omega} X(\omega)\dd \mbP(\omega)\right) = \int_{\Omega} \varphi(X(\omega))\dd \mbP(\omega).
	\end{equation*} 
	That is $\varphi(\mbE[X])= \mbE[\varphi(X)]$. In fact, the expectation is characterised as the unique element $x\in E$ such that,
	\begin{equation*}
		\varphi(x) = \mbE[\varphi(X)],\quad \text{ for all }\varphi \in E^\ast.
	\end{equation*}
\end{remark}
In the sequel we employ the notion of a conditional expectation, in particular when studying martingales taking values in Hilbert spaces. Existence and uniqueness of the conditional expectation for $E$ valued random variables follows in a similar manner as in the finite dimensional case. 
\begin{theorem}\label{th:CondExpect}
	Let $X$ be a Bochner integrable, $E$ valued random variable and let $\mcG \subseteq \mcF$ be a sub sigma-algebra of $\mcF$. Then there exists a unique, up to a $\mbP$ measure zero set, integrable, $E$ valued random variable, $Z$, that is $\mcG$ measurable and such that 
	\begin{equation}\label{eq:CondExpDefine}
		\int_A X \dd \mbP = \int_A Z \dd \mbP,\quad \text{ for all }A \in \mcG.
	\end{equation}
	We write $\mbE[X|\mcG]$ for the random variable $Z$ and refer to it as the conditional expectation of $X$ with respect to $\mcG$.
\end{theorem}
\begin{remark}
	Often, we wish to talk about the conditional expectation of a random variable, $X$, with respect to another random variable $Y$. In this case we use the abuse of notation $\mbE[X|Y]$ by which we really mean $\mbE[X|\sigma(Y)]$, where $\sigma(Y)$ is the $\sigma$-algebra generated by $Y$.
\end{remark}
\begin{proof}[Proof of Theorem \ref{th:CondExpect}]
	We begin by showing uniqueness. Assume that there exists two random variables $Z,\, \tilde{Z}$ such that both equal $\mbE[X|\mcG]$. Therefore, using Remark \ref{rem:LinOperatorExpecationCommute} it holds that for any $\varphi \in E^\ast$,
	\begin{equation*}
		\int_A \varphi(Z)\dd \mbP = \varphi \left(\int_A Z \dd \mbP\right) =  \int_A \varphi (X)\dd \mbP, \quad \text{ for all } A \in \mcG.  
	\end{equation*}
	Similarly, $\int_A \varphi(\tilde{Z})\dd \mbP=\int_A\varphi(X)\dd \mbP$. Therefore, for all $\varphi \in E^\ast$, $\varphi(Z),\, \varphi(\tilde{Z})$ both equal the real, conditional expectation $\mbE[\varphi(X)|\mcG]$. Since uniqueness holds for the finite dimensional conditional expectation we have $\varphi(Z)=\varphi(\tilde{Z})$ for all $\varphi \in E^\ast$ and so by Item \ref{it:Modification} of Corollary \ref{cor:EValuedRVs} we have $Z=\tilde{Z}$ $\mbP$-a.s.\\ \par 
	To prove existence, let $X = \sum_{k=1}^N x_k\mathds{1}_{A_k}$ be a simple random variable, then we define,
	\begin{equation*}
		Z = \sum_{k=1}^N x_k \mbP[A_k|\mcG],
	\end{equation*}
	where $\mbP[A_k|\mcG]$ is the usual conditional probability of $A_k$ with respect to $\mcG$. It is clear that $Z$ satisfies \eqref{eq:CondExpDefine} and by the law of total expectation we have
	\begin{equation}\label{eq:SimpleCondExpBnd}
		\mbE[\|Z\|_E] \leq \sum_{k=1}^N \mbE[\|x_k\|_E]\mbP[A_k] =  \mbE[\|X\|_E].
	\end{equation}
	Now let $(X_n)_{n\geq 1}$ be a sequence of simple random variables, approximating $X$ and $Z_n = \mbE[X_n|\mcG]$. So then for $n,\,m \in \mbN$, by \eqref{eq:SimpleCondExpBnd}, we have
	\begin{equation*}
		\mbE[\|Z_n-Z_m\|_E] \leq \mbE[\|X_n-X_m\|_E] \rightarrow 0,\quad \text{ as }n,\,m \rightarrow \infty.
	\end{equation*}
	Therefore $Z_n$ is a Cauchy sequence in $L^1(\Omega,\mcG,\mbP;E)$ and thus there exists a $\mcG$-measurable random variable $Z:= \lim_{n\rightarrow \infty}Z_n$ which by continuity of the Bochner integral also satisfies \eqref{eq:CondExpDefine}.
\end{proof}
\subsection{Gaussian Measures on Hilbert Spaces}
Recall that a real random variable $X:\Omega\rightarrow \mbR$ is called Gaussian if there exist $m \in \mbR$, $q \in [0,\infty)$ such that for any $A \in \mcB(\mbR)$,
\begin{equation*}
	\mcL(X)(A):=\mbP[X\in A] = \frac{1}{\sqrt{4\pi \sigma}}\int_A e^{-\frac{|x-m|^2}{2q}}\dd x.
\end{equation*}
In the multivariate case we say that an $\mbR^d$ valued random variable is Gaussian if there exist $m \in \mbR^d$ and $Q\in \mbR^{d\times d}$, a non-negative, symmetric matrix such that for any $A \in \mcB(\mbR^d)$,
\begin{equation*}
	\mcL(X)(A):=\mbP[X\in A] = \frac{1}{(4\pi )^{d/2}\sqrt{\det{Q}}} \int_A e^{-\frac{1}{2} \langle Q^{-1}(x-m),x-m\rangle}\dd x
\end{equation*}
In both cases, if the variance is zero we interpret the measure, $\mcL(X)$, as a dirac mass centred at $m$. Both definitions can also be concisely given in terms of characteristic functions (Fourier transforms) of measures. A measure $\mu \in \mcP(\mbR^d)$ is a Gaussian measure if and only if its characteristic function $\mbR^d \ni \lambda  \mapsto \hat{\mu}(\lambda) := \mbE[e^{i\langle \lambda, X\rangle }]$ is given by the expression
\begin{equation}\label{eq:FiniteDimGaussMuHat}
	\hat{\mu}(\lambda)= e^{i \langle m,\lambda \rangle  - \frac{1}{2}\langle Q\lambda,\lambda\rangle } \quad \text{ for all }\lambda \in \mbR^d,
\end{equation}
where $m\in \mbR^d$ and $Q\in\mbR^{d\times d}$ is a symmetric, non-negative matrix. We say that an $\mbR^d$ valued random variable is Gaussian if its law, $\mcL(X)$ is a Gaussian measure in this sense.\\ \par 
These definitions motivate the notion of Hilbert space valued random variables in possibly infinite dimensions. We fix a separable, Hilbert space $U$ and identify $U^\ast$ with $U$ from now on. As a result we change notation from the previous section, instead of writing $g(h)$ for $g \in U^\ast,\, f\in U$, we write $\langle f,g\rangle$, where the inner product is understood as the inner product on $U$.

\begin{definition}[$U$-Valued Gaussian Random Variable]\label{def:GaussianRV}
	Let $X$ be a $U$-valued random variable. Then we say that $X$ is Gaussian if and only if $\langle X,g\rangle$ is a real Gaussian random variable, for every $g \in U$. We say that a measure $\mu \in \mcP(U)$ is Gaussian if $\mu(A)= \mbP[X\in A]$ for a Gaussian random variable $X$.
\end{definition}

It follows from the Riesz representation theorem, that given a Gaussian measure $\mu \in \mcP(U)$, there exists $m\in U$ and $Q :U\rightarrow U$, such that for all $h,\,g \in U$,
\begin{equation}\label{eq:GaussMeanCov}
	\begin{aligned}
		\langle m,h\rangle  &= \int_U \langle x,h\rangle\dd \mu(x) = \mbE[\langle X,h\rangle],\\ \langle Qh,g\rangle &= \int_U \langle x,h\rangle \langle x,g\rangle \dd \mu(x)= \mbE[\langle X,h\rangle\langle X,g\rangle].
	\end{aligned}
\end{equation}
From now on we only deal with centred Gaussian measures, that is measures for which $m=0$. By linearity one may always translate a centred measure to a non-centred measure. It is immediate that $Q$ is linear, symmetric and non-negative. We will shortly demonstrate that in fact $Q$ defines a trace class operator from $U$ to itself. \\ \par
At this point we remark that much of the theory presented below for Gaussian measures on Hilbert spaces extends quite naturally, to the more general setting of separable, reflexive Banach spaces, satisfying either a $2$-martingale or uniform martingale difference (UMD) property, see \cite{neerven_veraar_weis_07} for an overview on these topics. A benefit of the theory of stochastic integration in UMD Banach spaces is that it leads to maximal regularity of the constructed integrals. On the other hand, the theory presented here can be implemented to develop a theory for Banach space valued integrals at the price of a small loss in regularity, see \cite[Sec. 5]{hairer_09}. For an approach to SPDEs valued in $L^\infty$ based spaces, such as $\mcC^\alpha$, $W^{k,\infty}$ we refer to the pathwise approach briefly presented in Chapter \ref{ch:Extras} here.\\ \par 
For a measure, $\mu$, on a Hilbert space $U$, we define the characteristic function $U \ni h\mapsto \hat{\mu}(h)\in \mbC$, by the expression,
\begin{equation*}
	\hat{\mu}(h):= \int_U e^{i\langle g,h\rangle }\dd\mu(g).
\end{equation*}
It follows from Definition \ref{def:GaussianRV} and \eqref{eq:GaussMeanCov} that for $\mu$ a centred, Gaussian measure on $U$, one has
\begin{equation}\label{eq:InfiniteDimMuHat}
	\hat{\mu}(h) = e^{- \frac{1}{2}\langle Qh,h\rangle},\quad \text{ for all } h\in U.
\end{equation}
To see this, use the fact that by definition, there exists a $U$ valued random variable, $X$, such that $\langle X,h\rangle$ is a one dimensional Gaussian and the expression \eqref{eq:FiniteDimGaussMuHat} for the Fourier transform of finite dimensional Gaussian measures. In fact, Theorem \ref{th:GuassianCharacteristic} below, shows that any measure with characteristic function of the form \eqref{eq:InfiniteDimMuHat} is a Gaussian measure. \\ \par
We say that two measures, $\mu,\,\nu\in \mcP(U)$ are equal, if $\mu(A)=\nu(A)$ for all $A \in \mcB(U)$. A useful extension of Item \ref{it:Modification} of Corollary \ref{cor:EValuedRVs} shows that the characteristic function uniquely determines the measure. We note that the following sequence of results hold directly on separable Banach spaces, not only Hilbert spaces provided one appropriately defines $\hat{\mu}$ in terms of elements of $E^\ast$.
\begin{proposition}\label{prop:FourierUnique}
	Let $\mu,\,\nu\in \mcP(U)$ be two measures on a Hilbert space $U$. Then $\mu=\nu$ if and only if $\hat{\mu}(g)=\hat{\nu}(g)$ for all $g\in U$.
\end{proposition}
\begin{proof}
	By linearity, it suffices to show that if $\hat{\mu}\equiv 0$ then $\mu=0$. For $\tilde{\mu}\in \mcP(\mbR)$ and $\varphi \in C^\infty_c(\mbR)$ one has
	\begin{equation*}
		\int_\mbR \varphi(x)\dd \tilde{\mu}(x) = \frac{1}{2\pi} \int_{\mbR} \int_{\mbR} \hat{\varphi}(\xi) e^{-i\xi x} \dd \xi \,\tilde{\mu}(\dd x) = \frac{1}{2\pi}\int_{\mbR} \hat{\varphi}(\xi) \hat{\tilde{\mu}}(-\xi)\dd \xi.
	\end{equation*}
	Therefore, approximating any continuous and bounded function by smooth functions we see that the Fourier transform uniquely determines one dimensional probability measures. The infinite dimensional case then follows from Item \ref{it:Modification} of Corollary \ref{cor:EValuedRVs}.
\end{proof}
So far we have not in fact shown that measures satisfying Definition \ref{def:GaussianRV} in fact exist. In order to do so we establish some a priori properties of Gaussian measures. Recall that given an angle $\theta \in [0,2\pi)$ a vector in $U^2$ is rotated under the map $R_\theta(h,g):= (h \cos \theta- g \sin \theta, h\sin \theta +g\cos \theta)$. By extension we define the rotation of $\mu\in \mcP(U)$  by $\theta$, as the image of $\mu\otimes \mu$ under the map,
\begin{equation*}
	\begin{aligned}
		\mfR_\theta :\mcP(U)\otimes \mcP(U)&\rightarrow \mcP(U)\otimes \mcP(U)\\
		\mu \otimes \nu &\mapsto R_\theta\# (\mu\otimes \nu)
	\end{aligned}
\end{equation*}
\begin{lemma}\label{lem:GaussRotInvar}
	Let $\mu$ be a Gaussian measure on $U$. Then for any $\theta \in [0,2\pi)$, one has $\mfR_\theta(\mu\times \mu)= \mu\otimes \mu$.
\end{lemma}
\begin{proof}
	It follows from Proposition \ref{prop:FourierUnique} that is enough to check that the identity,\\ $\left(\widehat{\mu\otimes \mu}\right)\circ R^{-1}_\theta= \widehat{\mu\otimes \mu}:U^2\rightarrow \mbC^2$, holds for the characteristic functions. This can be checked using the explicit expression \eqref{eq:InfiniteDimMuHat}.
\end{proof}
A consequence of this rotation invariance is a bound on exponential moments of Gaussian measures known as Fernique's theorem. Note that the theorem holds for all measures that are invariant by rotations of $\pi/4$ degrees, not just Gaussian measures.
\begin{theorem}[Fernique's Theorem]\label{th:Fernique}
	Let $\mu \in \mcP(U)$ be a probability measure such that $\mfR_{\pi/4}(\mu\otimes \mu) = \mu \otimes \mu$. Then there exists an $\alpha>0$ such that
	\begin{equation*}
		\int_U e^{\alpha\|h\|^2_U} \dd \mu(h) <\infty.
	\end{equation*}
\end{theorem}
\begin{proof}
	Let $\tau,\,\sigma>0$. Then by $\pi/4$-rotation invariance we have that
	\begin{align*}
		\mu(\|h\|_{U}\leq \tau)\mu(\|g\|_U>\sigma) &= \mu\left(\left\| \frac{h-g}{\sqrt{2}}\right\|_U\leq \tau\right) \mu\left(\frac{\|h+g\|_U}{\sqrt{2}}>\sigma\right) \\
		&= \int_{\left\|\frac{h-g}{\sqrt{2}}\right\|_U \leq \tau }\int_{\left\|\frac{h+g}{\sqrt{2}}\right\|_U> \sigma} \dd\mu(h)\dd \mu(g)
	\end{align*}
	By the triangle inequality we see that
	\begin{equation*}
		\min\{ \|h\|_U,\,\|g\|_U \} \geq \frac{1}{2}\left(\|h+g\|_{U}-\|h-g\|_{U}\right),
	\end{equation*}
	and so it follows that together, the conditions, $\|h-g\|_{U}\leq \sqrt{2}\tau$ and $\|h+g\|_{U}>\sqrt{2}\sigma$, that $\min\{ \|h\|_{U},\,\|g\|_{U}\} \geq \frac{\sigma-\tau}{\sqrt{2}}$. So using the above equality, we have that
	\begin{equation*}
		\mu(\|h\|_U\leq \tau)\mu(\|g\|_U>\sigma) \leq \int_{\|h\|_U >\frac{\sigma-\tau}{\sqrt{2}}} \int_{\|g\|_U>\frac{\sigma-\tau}{\sqrt{2}}} \dd\mu(h)\dd \mu(g) = \mu\left(\|h\|_{U}>\frac{\tau-\sigma}{\sqrt{2}}\right)^2.
	\end{equation*}
	Since $\|h\|_U <\infty$ $\mu$-a.s., there exists some $\sigma>0$ such that $\mu(\|h\|_U\leq \sigma)\geq \frac{3}{4}$. So now we set $t_0 = \sigma$ to be the minimal value such that this holds and for $n>0$ we define, $t_n = \frac{t_{n+1}-\sigma}{\sqrt{2}}$. So it follows that
	\begin{equation*}
		\mu(\|h\|_U>t_{n+1}) \leq \frac{\mu\left( \|h\|_U> \frac{t_{n+1}-\sigma}{\sqrt{2}} \right)}{\mu(\|h\|_U \leq \sigma)}\leq \frac{4}{3}\mu(\|h\|_{U}>t_n)^2.
	\end{equation*}
	Therefore, defining the shorthand, $k_n := \frac{4}{3}\mu(\|h\|_{U}>t_n)$, we have shown the recursion $k_{n+1}\leq k_n^2$ with $k_0 \leq \frac{1}{3}$. Repeatedly applying this inequality gives the bound $\mu(\|h\|_U>t_n)\leq 3^{-2n}$. On the other hand, one can check explicitly that $t_n = \frac{\sqrt{2}^{n+1}-1}{\sqrt{2}-1}\sigma \leq 2^{n/2}(2+\sqrt{2})\sigma$ so that in particular $t_{n+1} \leq 2^{n/5}5\sigma$. Therefore in combination we have
	\begin{equation*}
		\mu(\|h\|_{U}>t_n)\leq 3^{-\frac{t^2_{n+1}}{25\sigma^2}}.
	\end{equation*}
	So it follows that there exists a universal constant $\tilde{\alpha}>0$ such that $\mu(\|h\|_U>t)\leq e^{-\frac{2\tilde{\alpha} t^2}{\sigma^2}}$, holds for every $t>\sigma$. So integrating by parts we have that
	\begin{align*}
		\int_U \exp\left(\frac{\tilde{\alpha}\|h\|^2_U}{\sigma^2} \right)\dd \mu(h) &\leq e^{\tilde{\alpha}} + \frac{2\tilde{\alpha} }{\sigma^2}\int_{\sigma}^{\infty} t e^{ \frac{\tilde{\alpha} t^2}{\sigma^2}}\mu(\|h\|_U>t)\dd t \notag \\
		&\leq e^{\tilde{\alpha}} + 2\tilde{\alpha} \int_1^\infty t e^{-\tilde{\alpha} t^2}\dd t \\
		&<\infty\notag ,
	\end{align*}
	which is the stated result with $\alpha =\tilde{\alpha}/\sigma$.
\end{proof}
It follows from Fernique's theorem that Gaussian measures have finite moments of all orders. In fact one can show an even stronger result, namely that the second moment of a Gaussian measure controls all higher moments. We postpone this proof to Chapter \ref{ch:Extras} since we do not make use of it here. We employ Theorem \ref{th:Fernique} here to give the following characterisation of Gaussian measures on a Hilbert space.
\begin{theorem}\label{th:GuassianCharacteristic}
	Let $U$ be a Hilbert space. A measure $\mu \in \mcP(U)$ is Gaussian if and only if there exist $m \in U$ and $Q \in L_1(U)$ such that
	\begin{equation}\label{eq:InfiniteDimGaussChar}
		\hat{\mu}(h) = e^{i\langle m,h\rangle - \frac{1}{2}\langle Qh,h\rangle}, \quad \text{ for all }h\in U.
	\end{equation}
	Furthermore one has the identity,
	\begin{equation}\label{eq:VarIsTrace}
		\int_U\|h\|^2_U\dd \mu(h) = \Tr Q.
	\end{equation}
	Finally, if $\mu= \mcL(X)$ then we say that $X\sim \mcN(m,Q)$.
\end{theorem}
\begin{proof}
	As above, without loss of generality we may prove the result only for centred Gaussian measures, so we set $m=0$. Furthermore, we have already seen in \eqref{eq:InfiniteDimMuHat} that given a Gaussian measure there exists a linear, non-negative, symmetric operator $Q:U\rightarrow U$ such that \eqref{eq:InfiniteDimGaussChar} holds. Therefore, in one direction it suffices to show that this operator $Q$ is in fact trace class and in turn that the expression \eqref{eq:VarIsTrace} is finite.\\ \par 
	It follows from Theorem \ref{th:Fernique} that given a Gaussian measure, a priori one has  the bound
	\begin{equation*}
		\int_U \|h\|^2_U\dd\mu(h)\leq \frac{1}{\alpha}\int_U e^{\alpha\|h\|^2_U}\dd\mu(h)<\infty. 
	\end{equation*}
	Now let $(e_n)_{n\geq 1}\subset U$ be a basis, so that applying Lebesgue's dominated convergence in the first step, we have,
	\begin{equation*}
		\int_U \|h\|^2_U \dd\mu(h) = \sum_{k=1}^\infty \int_U \langle h,e_k\rangle^2\dd\mu(h) = \sum_{k=1}^\infty\langle Qe_k,e_k\rangle = \Tr Q.
	\end{equation*}
	That is, if a measure, $\mu\in \mcP(U)$, satisfies \eqref{eq:GaussMeanCov}, then $Q \in L_1(U)$.\\ \par 
	To prove the converse, let $Q\in L_1(U)$ so that by \eqref{eq:TraceClassRep} there exists a basis $(e_n)_{n\geq 1}\subset U$ such that $Qe_n = \lambda_n e_n$ for a summable sequence $(\lambda_n)_{n\geq 1} \subset \mbR_{\geq 0}$. Then, let $(\beta_n)_{n\geq 1}$ be a family of i.i.d, real valued standard normal random variables, which exist by Kolmogorov's extensions theorem. By definition, $\sum_{n\geq 1} \lambda_n \mbE[\beta_n^2] = \Tr Q$ and so the series $\sum_{n\geq 1}\sqrt{\lambda_n}\beta_n e_n$ is finite in mean square and so there exists a subsequence of partial sums converging $\mbP$-a.s. in $U$. It follows that $X := \sum_{n\geq 1}^\infty \sqrt{\lambda_n}\beta_n e_n$ defines a $U$ valued Gaussian random variable in the sense of Definition \ref{def:GaussianRV}.\\ \par 
	In both cases the form of the characteristic function, \eqref{eq:InfiniteDimGaussChar}, can be computed explicitly.
\end{proof}
In the case of Banach space valued Gaussian random variables, one can instead show that the covariance operators is a compact operator. See \cite[Ex. 3.16]{hairer_09}.
\subsection{Generalised Gaussian Random Variables}\label{subsec:GenGaussinRV}
One may wonder if we can allow for more general covariance operators. An example of particular interest is the white-noise, otherwise known as a cylindrical Gaussian random variable. Let $U$ be as above, $(e_k)_{k\geq 1}$ be any orthonormal basis and $(\beta_k)_{k\geq 1}$ be i.i.d, real, standard normal, random variables. Then consider the random variable defined by the formal sum,
\begin{equation}\label{eq:WhiteNoiseSum}
	W:= \sum_{k=1} \beta_k e_k.
\end{equation}
This corresponds to formally setting $Q=\text{Id}_U$, the identity map on $U$. Since the identity map is not trace-class, $W$, cannot be a $U$ valued Gaussian random variable. However, if we consider $W$ as a random variable taking values in a suitable space larger than $U$, it turns out to be perfectly well defined Gaussian random. For example, consider the Hilbert--Schmidt extension, $U_{\text{ex}}$, of $U$, defined to be the completion of $U$ under the norm,
\begin{equation*}
	\|x\|^2_{U_{\text{ex}}} := \sum_{k=1}^\infty\frac{1}{k^2}\langle x,e_k\rangle^2.
\end{equation*}
The inclusion map $\iota:U\rightarrow U_{\text{ex}}$ is Hilbert--Schmidt and since $U\subset U_{\text{ex}}$, for any $h\in U$, the action $\langle W,h\rangle_{U_{\text{ex}}}$ is well defined. Interpreting all inner products in the only admissible sense, we then have, for $h,\,g\in E$,
\begin{equation*}
	\mbE[\langle W,h\rangle ] =0,\quad \mbE[\langle W,h\rangle\langle W,g\rangle]=\langle h,g\rangle.
\end{equation*}
Thus we may treat $W$ as defining a Gaussian random variable on $U$, which takes values almost surely in $U_{\text{ex}}$. This leads us to the following definition.
\begin{definition}[Generalised Gaussian Random Variable]\label{def:GenGaussDef}
	Given a Hilbert space, $U$, we say that a linear map $h \mapsto X_h$ defines a generalised Gaussian random variable, if $X_h$ is a real Gaussian random variable for every $h \in U$ and if $h_n\rightarrow h \in U$,
	\begin{equation*}
		\mbE\left[	|X_{h}-X_{h_n}|^2\right] \rightarrow 0.
	\end{equation*}
\end{definition}
Consider the example of the white noise, $W$, above. The map $h\mapsto \langle W,h\rangle =: W_h$ is linear and defines a family of real, Gaussian random variables. Furthermore, for $(h_n)_{n\geq 1}\subset U$ converging to $h\in U$ we have,
\begin{equation*}
	\mbE\left[|W_{h}-W_{h_n}|^2\right] = \|h_n-h\|^2_U \rightarrow 0.
\end{equation*}
It follows from the definition of a generalised Gaussian random variable, that there exists a bi-linear form, $K:E\times E\rightarrow \mbR$ and a symmetric, non-negative operator $Q:E\rightarrow E$ such that,
\begin{equation*}
	\mbE\left[X_h X_g\right] = K(h,g) = \langle Qh,g\rangle,\quad \text{ for all }h,\,g \in U.
\end{equation*}
As before, we call $Q$ the covariance operator of $X$. In the case of the white noise, $W$, the covariance operator is the identity on $U$. In fact, all generalised Gaussian random variables can be written formally as
\begin{equation}\label{eq:GenGaussianSum}
	X = \sum_{k=1} \beta_k Q^{1/2}e_k,
\end{equation} 
where $(\beta_k)_{k\geq 1}$ are a family of  i.i.d, real, standard normals, $(e_k)_{k\geq 1}$ is an orthonormal basis of $U$ and $Q\in L(U)$ is a symmetric and non-negative operator. Note that $Q$ need not be trace class, however, since it is symmetric and non-negative it admits an invertible, square root. From this we define the reproducing kernel space of $X$.
\begin{definition}[Reproducing Kernel Space]\label{def:RepKernelSpace}
	Let $X$ be a generalised Gaussian random variable on a Hilbert space $U$, with covariance $Q$. Then, we define the reproducing kernel of $X$ to be the space $U_0 := Q^{1/2}(U)$. Note that $U_0$ is itself a Hilbert space, with inner product $\langle h,g\rangle_{U_0}= \langle Q^{-1/2}h,Q^{-1/2}g\rangle_U$.
\end{definition}
This definition allows us to give proper meaning to the generalised Gaussian random variables as defined in \eqref{eq:GenGaussianSum}.
\begin{proposition}\label{prop:HSExtension}
	Let $U,\,U_1$ be Hilbert spaces and $Q \in L(U)$ be such that $Q^{1/2}(U) =: U_0 \subseteq U_1$ with Hilbert--Schmidt embedding $\iota Q^{1/2} : U\rightarrow U_1$. Then the sum on the right hand side of \eqref{eq:GenGaussianSum} defines a generalised Gaussian random variable on $U$ in the sense of Definition \ref{def:GenGaussDef}. Furthermore, the definition of $X$ as a generalised Gaussian random variable on $U$ is independent of the choice of $U_1$ and $\iota$. Finally, defining $Q_1 := \iota Q\iota^\ast $, one has that $\iota Q^{1/2}:U\rightarrow Q^{1/2}_1(U_1)$ defines an isomorphism 
\end{proposition}
\begin{proof}
	Almost sure convergence of the sum $\eqref{eq:GenGaussianSum}$ in $U_1$ follows from the bound,
	\begin{align*}
		\sum_{k\geq 1} \mbE[\beta^2_k] \langle Q^1/2 e_k,Q^{1/2}e_k\rangle_U &=  \sum_{k\geq 1}\langle \iota Q^{1/2}e_k,\iota Q^{1/2}e_k\rangle_{U_1}\\
		& = \sum_{k\geq 1} \langle \iota Q\iota^\ast e_k,e_k\rangle\\
		& = \Tr \iota Q \iota^\ast <\infty.
	\end{align*}
	Where the last bracket is understood as the duality bracket between $U$ and $U_1$. Recall that since $\iota Q^{1/2}$ is Hilbert--Schmidt, $\iota Q\iota^\ast$ is trace class by definition. It therefore follows that $X = \sum_{k\geq 1} \beta_k Q^{1/2}e_k$ defines a $U_1$ valued Gaussian random variable, on $U$ in the sense of Definition \ref{def:GaussianRV}. Furthermore, it is readily checked that for $X$ defined in this way, the map,
	\begin{equation*}
		U\ni h\mapsto X_h := \sum_{k=1} \beta_k \langle Q^{1/2}e_k,e_k\rangle \langle h,e_k\rangle,
	\end{equation*}
defines a generalised Gaussian random variable, with zero mean, in the sense of Definition \ref{def:GenGaussDef}. Furthermore, since the choice of $\iota,\,U_1$ does not play a role in the action $h\mapsto X_h$ they do not affect the definition, provided the embedding $\iota Q^{1/2}:U\rightarrow U_1$ is Hilbert--Schmidt.\\ \par 
	To see that $\iota Q^{1/2}$ defines an isomorphism between $U$ and $Q_1^{1/2}(U_1)$, we first note that since $Q_1 = (\iota Q)(\iota Q^{1/2})^\ast$, by Corollary \ref{cor:SquareRootOfSquare}, we have $Q^{1/2}_1(U_1)= \iota (U_0)$. Furthermore, from the same result we have that $\|Q^{-1/2}_1 h_1\|_{U_1} = \|\iota^{-1}u_1\|_{U_0}$ for all $h_1 \in \iota(U_0)$. However, since the mapping $\iota :U_0\rightarrow U_1$ is surjective, we may replace $h_1$ by $\iota h_0$ for any $h_0 \in U_0$ and the final claim is shown.
\end{proof}
In the case of the white noise, $W$, defined above, the Hilbert--Schmidt extension of $U$ provides an example of $U_1$. However, due to the final assertion of Proposition \ref{prop:HSExtension}, the definition of the white noise is independent of the choice of $U_1$, provided the embedding $\iota:U_0\rightarrow U_1$ is Hilbert--Schmidt.
\begin{example}
	Let $U = L^2([0,T])$ and $W$ be a white noise defined on $U$. Then, define the real process, $t\mapsto W_t := \langle W,\indic_{[0,t]}\rangle$. For any $t,\,s \in [0,T]$ we have, 
	\begin{equation*}
		\mbE[W_t]=0,\quad \mbE[W_tW_s]= t\wedge s.
	\end{equation*}
	The mapping $[0,T] \ni t\mapsto W_t \in \mbR$ is by definition $\mbP$-a.s. continuous. Finally, if we consider the filtration $(\mcF_t)_{t\in [0,T]}$ generated by $(W_t)_{t\in [0,T]}$, then by linearity, for any $t>s$ we have
	\begin{equation*}
		\mbE[W_t|\mcF_s] = \mbE[W_{s}+ W_{t-s}|\mcF_s] = W_s.
	\end{equation*}
	So by Levy's characterisation and Kolmogorov's continuity criterion, $t\mapsto W_t$ defines a real Brownian motion.
\end{example}
\section{Stochastic Analysis in Infinite Dimensions}
The main purpose of this chapter is to build up a theory of stochastic analysis in infinite dimensional Hilbert spaces. In Chapter \ref{ch:SPDE} we will use this theory to understand a certain class of SPDE as infinite dimensional stochastic evolution equations. We first develop some theory of abstract martingales in Banach spaces, after which we turn to the specific case of Gaussian processes in Hilbert spaces and conclude by developing a theory of stochastic integration for these processes.
\begin{definition}[Stochastic Processes]\label{def:StochasticProcess}
	Given a Banach space $E$, equipped with its Borel $\sigma$-algebra and $(\Omega,\mcF,\mbP)$ an abstract probability space, we say that a measurable map $X:[0,\infty)\times \Omega\rightarrow E$ is an $E$ valued stochastic process.
\end{definition}
We typically suppress the dependence of $X$ on $\omega\in \Omega$ and use the notations $(X_t)_{t\in [0,\infty)}$ and $[0,\infty)\ni t\mapsto X_t$ to denote these processes.
\begin{definition}[Modifications of Stochastic Processes]\label{def:Modifications}
	Let $X,\,Y :[0,\infty)\times \Omega\rightarrow E$ be two stochastic processes. Then we say that $Y$ is a modification of $X$ if for all $t\in [0,\infty)$, one has
	\begin{equation*}
		\mbP[X_t=Y_t]=1.
	\end{equation*}
\end{definition}
In what follows we often deal with stochastic processes restricted to an interval $[0,T]$ for $T\in (0,\infty)$. In this case these definition only change up to modifying the interval of definition.
\subsection{Banach Space Valued Martingales}\label{subsec:BValuedMartingales}
As in finite dimensional stochastic analysis, the class of Martingale processes will play a central role. We introduce the concept of a filtered probability space.
We say that a family $(\mcF_t)_{t\geq 0}$, of $\sigma$-algebras is non-decreasing if $\mcF_s \subseteq \mcF_t$ for all $0\leq s <t<\infty$. We also introduce the notations,
\begin{equation*}
	\mcF_{t+} = \bigcap_{s>t} \mcF_s,\quad \mcF_{t-}= \bigcap_{s<t}\mcF_s.
\end{equation*}
\begin{definition}[Filtered Probability space]\label{def:FilteredSpace}
	We say that a tuple, $(\Omega,\mcF,(\mcF_t)_{t\geq 0},\mbP)$, consisting of the the usual triple; state space $\Omega$, $\sigma$-algebra, $\mcF$, and probability measure, $\mbP$, and in addition a non-decreasing family of sub-$\sigma$-algebras, $(\mcF_t)_{t\geq 0}$ is a filtered probability space. We say that a filtration is normal if $\mcF_0$ contains all $\mbP$-null sets of $\mcF$ and if
	$$\mcF_t =\mcF_{t+} \quad \text{ for all }t\geq 0.$$
\end{definition}
Given a filtered probability space, we define the notion of an $(\mcF_t)_{t\geq 0}$-martingale.
\begin{definition}[Martingales]\label{def:EValuedMart}
	Let $(M_t)_{t\geq 0}$ be a stochastic process, taking values in $E$ and $(\mcF_t)_{t\geq 0}$ be a filtration on $(\Omega,\mcF,\mbP)$. Then we say that $(M_t)_{t\geq 0}$ is an $(\mcF_{t})_{t\geq 0}$-martingale if,
	\begin{enumerate}
		\item $\mbE[\|M_t\|_E]<\infty$ for all $t\geq 0$,
		\item $M_t$ is $\mcF_t$ measurable for all $t\geq 0$,
		\item $\mbE[M_t \,|\,\mcF_s] = M_s$ $\mbP$-a.s. for all $0\leq s\leq t<\infty$.
	\end{enumerate}
\end{definition}
We say that a stochastic process is a sub-martingale if the final property is replaced with the inequality $\mbE[M_t|\mcF_s]\geq M_s$ and a super-martingale if it is replaced by the inequality $\mbE[M_t|\mcF_s] \leq M_s$.\\ \par 
Given a filtered probability space $(\Omega,\mcF,(\mcF_t)_{t\geq 0}, \mbP)$, we recall the notion of a stopping time.
\begin{definition}[Stopping Times]\label{def:StoppingTime}
	We say that a random variable $\Omega\ni \omega \mapsto \tau(\omega)\in \mbR_+$ is an $(\mcF_t)_{t\geq 0}$ stopping time, if for any $t\geq 0$,
	\begin{equation*}
		\{ \omega \in \Omega\,:\, \tau(\omega) \leq t\} \in \mcF_t.
	\end{equation*}
\end{definition}
\begin{remark}
	If $(\mcF_t)_{t\geq 0}$ is normal, in the sense of Definition \ref{def:FilteredSpace}, then any deterministic time $T\geq 0$ is an $(\mcF_t)_{t\geq 0}$ stopping time.
\end{remark}
\begin{remark}
	We will mostly be concerned with so called hitting times, that is given an $E$ valued process, $(X_t)_{t\geq0}$, for $R\geq 0$ and $\varphi \in E^\ast$, we define
	\begin{equation*}
		\tau_R := \inf \{ t\geq 0\,:  |\varphi(X_t)| > R \}.
	\end{equation*}
	If $t\mapsto X_t$ is almost surely continuous and $(\mcF_t)_{t\geq 0}$ is normal then any random time of this form is an $(\mcF_t)_{t\geq 0}$ stopping time.
\end{remark}
The following example shows that in some cases the requirements of being a true martingale in the sense of Definition \ref{def:EValuedMart} are too stringent.
\begin{example}
	Let $(\Omega,\mcF,(\mcF_t)_{t\geq 0},\mbP)$ be a filtered probability space carrying two random variables, $\xi$ which is $\mcF_1$ measurable, takes values in $\mbR_+$ and such that $\mbE[\xi]=\infty$ and $\eta$ which is $\mcF_2$ measurable, independent of $\mcF_{2-}$ and takes values $\pm 1$ each with probability $1/2$. We define the process $[0,\infty)\ni t\mapsto X_t := \mathds{1}_{t\geq 2}\eta \xi$. It follows immediately that $\mbE[|X_2|] = \mbE[\xi] =\infty$ so that $X$ cannot be a martingale. However, we can define the sequence of stopping times, for $R\geq 0$,
	\begin{equation*}
		\tau_R := \begin{cases}
			R & \text{if }\xi\leq R,\\
			1 & \text{if }\xi>R.
		\end{cases}
	\end{equation*}
	Since $\xi(\omega)<\infty$ $\mbP$-a.s., $\tau_R \nearrow \infty$ $\mbP$-a.s. and for every $n\geq 0$, $X^{\tau_R}_{\,\cdot\,}:= X_{\tau_R \wedge \,\cdot\,}$ is a bounded martingale.
\end{example}
This leads us to the definition local martingales, which will be recurring objects in our analysis. In particular the stochastic integrals we define in Subsections \ref{subsec:QStochInt} \& \ref{subsec:GenStochInt} will only be local martingales for the integrands we typically consider.
\begin{definition}[Local Martingale]\label{def:LocalMart}
	We say that an $E$ valued stochastic process $(M_t)_{t\geq 0}$ is a local martingale, if there exists a sequence of $(\mcF_t)_{t\geq 0}$ stopping times $(\tau_n)_{n\geq 0}$ such that $\tau_n\rightarrow \infty$, $\mbP$-a.s. and for any $n\geq 0$, the stopped process $(M^{\tau_n}_t)_{t\geq 0}:=(M_{\tau_n\wedge t})_{t\geq 0}$ is a martingale in the sense of Definition \ref{def:EValuedMart}. Local sub/super-martingales are defined analogously. The sequence $(\tau_n)_{n\geq 0}$ is called a localising sequence for the local martingale $M$.
\end{definition}
As in the previous section we make a connection to real martingales.
\begin{proposition}\label{prop:RealMartingales}
	Let $(M_t)_{t\geq 0}$ be an $E$ valued stochastic process such that $\mbE[\|M_t\|_E]<\infty$ for all $t\geq 0$ and $(\mcF_t)_{t\geq 0}$ be a filtration as in Definition \ref{def:EValuedMart}. Then $(M_t)_{t\geq 0}$ is an $E$ valued martingale if and only if $(\langle M_t,h\rangle)_{t\geq 0}$ is a real, $(\mcF_t)_{t\geq 0}$-martingale for all $h\in E$. Furthermore, if $(M_t)_{t\geq 0}$ is an $(\mcF_t)_{t\geq 0}$-sub-martingale, then $(\|M_t\|^p_E)_{t\geq 0}$ is a real $(\mcF_t)_{t\geq 0}$-sub-martingale for any $p\in [1,\infty)$.
\end{proposition}
\begin{proof}
	To prove the equivalence, it is clear that if $(M_t)_{t\geq 0}$ is an $E$ valued $(\mcF_t)_{t\geq0}$-martingale then $(\langle M_t,h\rangle)_{t\geq 0}$ is a real $(\mcF_t)_{t\geq 0}$-martingale for all $h \in E$. On the other hand, if $\langle M_t,h\rangle$ is $\mcF_t$ measurable for all $t\geq 0$ and $h\in E$ then the element $M_t\in E$ must be $\mcF_t$ measurable for every $t\geq 0$. To show the martingale property it suffices to recall the construction of the conditional expectation in Hilbert spaces from Theorem \ref{th:CondExpect}.\\ \par 
	To show that the norm is a real martingale, recall that by Hahn--Banach there exists a sequence of elements $(h_k)_{k\geq 1}$ such that $\|M_t\|_E =\sup_{k\geq 1}\langle M_t,h_k\rangle$. Therefore, from the first equivalence, $\mbP$-a.s., we have that
	\begin{align*}
		\mbE[\|M_t\|_E |\mcF_s]\geq \sup_{k\geq 1} \mbE[\langle M_t,h_k\rangle|\mcF_s] = \sup_{k\geq 1} \langle M_s,h_k\rangle = \|M_t\|_E.
	\end{align*}
	This proves the claim for $p=1$, to show the same for $p>1$ it suffices to apply Jensen in the first inequality.
\end{proof}
We have the following version of Doob's maximal inequality.
\begin{theorem}\label{th:Doob}
	Let $(M_t)_{t\in [0,T]}$ be a right continuous, $E$ valued $(\mcF_t)_{t\in [0,T]}$-martingale. Then for any $p>1$,
	\begin{equation*}
		\mbE\left[\sup_{t \in [0,T]}\|M_t\|^p_E\right]^{\frac{1}{p}} \leq \frac{p}{p-1}\mbE[\|M_T\|^p_{E}]^{\frac{1}{p}}.
	\end{equation*}
\end{theorem}
\begin{proof}
	The result follows from Proposition \ref{prop:RealMartingales} and the same result for real valued, non-negative sub-martingales. See \cite[Thm. 4.5.6]{cohen_elliott_15}
\end{proof}
From now on we will focus in particular on the square integrable, continuous, martingales, i.e $E$ valued, continuous, $(\mcF_{t})_{t\in[0,T]}$-martingales such that
\begin{equation}\label{eq:M2Norm}
	\|M\|_{\mcM^2_T}:= \sup_{t \in [0,T]} \mbE\left[ \|M_t\|^2_E \right] = \mbE\left[\|M_T\|^2_E \right] <\infty.
\end{equation}
Identifying $\mcM^2_T(E)$ as a space as a subspace of the Banach space $L^2(\Omega;C([0,T];E))$ it follows that $\mcM^2_T(E)$ is itself a Banach space if it can be shown to be closed. However, by construction of the conditional expectation, it is easy to show that even convergence in $L^1(\Omega;C([0,T];E))$ preserves the martingale property and so $\mcM^2_T(E)$ is a Banach space, with norm given by \eqref{eq:M2Norm}. Note that finiteness of \eqref{eq:M2Norm} is not enough to ensure that a local martingale, $M$, is a true martingale, so that we may also speak of continuous, square integrable, local martingales and we write $\mcM^2_{T;\text{loc}}(E)$ for this space.\\ \par
Our ultimate goal in this chapter is to build the stochastic integral with respect to a given element of $\mcM^2_T(E)$, concretely a $Q$-Wiener process. In order to do so we proceed in analogy with the construction of the Lebesgue integral. We first identify a suitable space of simple processes, for which the integral is naturally defined. We then show that this integral is an isometry from the simple processes, equipped with a suitable norm, to the space $\mcM_T^2$. We then define the stochastic integral as the limit under approximation by simple processes. From now on we keep as implicit a filtered probability space $(\Omega,\mcF,(\mcF_t)_{t\in [0,T]};\mbP)$ and assume that all martingales are with respect to $(\mcF_t)_{t\in [0,T]}$. We also assume without further comment that all martingales are zero at $t=0$.
\subsection{Hilbert Space Valued Wiener Processes}
In the context of stochastic evolution equations we will wish to consider noise terms given by evolving processes of random variables. That is a Hilbert space valued stochastic process. Given a spatial domain, $\Gamma\subset \mbR^d$, we typically think of $U$ in the previous section, being a Hilbert space of functions $f:\Gamma \rightarrow \mbR^m$ and the noise as being a sequence of random elements in $U$ (or some larger space) indexed continuously by $t\geq 0$. We begin by defining the class of $Q$-Wiener processes.
\begin{definition}\label{def:QWienerProcess}
	Given a Hilbert space $U$ and a trace class operator $Q\in L_1(U)$, we say that $(W_t)_{t\geq 0}$ is a standard $Q$-Wiener process if,
	\begin{enumerate}
		\item $W_0 =0 \in U$,
		\item the map $\mbR_+\ni t\mapsto W_t\in U$ is $\mbP$-a.s. continuous,
		\item for any $0\leq s<t<\infty$, one has
		\begin{equation*}
			W_t-W_s \sim \mcN(0,(t-s)Q)\quad\text{and}\quad W_t-W_s\perp W_s.
		\end{equation*}
	\end{enumerate} 
\end{definition}
The qualification \textit{standard} here refers to the condition $W_0=0$, which entails by the martingale property (see below) that $\mbE[\langle W_t,h\rangle ]=0$ for all $t\geq 0$. Setting $W_0=m$ for a deterministic element $m\in E$ gives a non-centred $Q$-Wiener  process. One can also consider random initial data, provided the independent increments assumption still holds. This can be useful for example when studying invariant behaviour of stochastic systems, as one may wish to start the system already from its invariant distribution. From now on we only treat standard $Q$-Wiener processes and remove the qualifier from all subsequent statements.\\ \par  
Given Theorem \ref{th:GuassianCharacteristic} we can give a characterisation of the $Q$-Wiener processes.
\begin{theorem}
	Let $Q \in L_1(U)$ and $(e_k)_{k\geq 1},\,(\lambda_k)_{k\geq 1}$ be an orthonormal basis of $U$ generated by $Q$ and associated eigenvalues. Then $(W_t)_{t\geq 0}$ is a $Q$-Wiener process if and only if,
	\begin{equation}\label{eq:QWienerDef}
		W_t = \sum_{k\geq 1} \sqrt{\lambda_k}\beta^k_t e_k,
	\end{equation}
\end{theorem}
where $(\beta^k)_{k\geq 1}$ are independent, real, Brownian motions.
\begin{proof}
	This result essentially follows from the characterisation of Gaussian random variables given by Theorem \ref{th:GuassianCharacteristic}. It follows directly that the random variables $(W_t)_{t\geq 0}$ defined in this way are centred, Gaussian, $U$-valued random variables. Since $Q$ is trace class, the right hand side of \eqref{eq:QWienerDef} is square summable in expectation and so by the dominated convergence theorem we may exchange summation and the expectation. Therefore, independence of the increments, $W_t-W_s$, follows from independence of the increments $\beta^k_t -\beta_s^k$. finally, almost sure continuity of the paths follows also from almost sure continuity of the maps $t\mapsto \beta^k_t$, although in Chapter \ref{ch:Extras} we show this more directly by a generalisation of the Kolmogorov continuity criterion.
\end{proof}
For the remainder of this chapter, unless otherwise mentioned $Q\in L_1(U)$. Given a filtered probability space we define the notion of a $Q$-Wiener process with respect to a filtration.
\begin{definition}
	Let $(\Omega,\mcF,(\mcF_t)_{t\geq 0},\mbP)$ be a normal, filtered probability space and $(W_t)_{t\geq 0}$ be a $Q$-Wiener process. We say that $(W_t)_{t\geq 0}$ is a $Q$-Wiener process with respect to $(\mcF_t)_{t\geq0}$ if,
	\begin{enumerate}
		\item $W_t$ is $\mcF_t$ measurable for every $t\geq 0$,
		\item $W_t-W_s$ is independent of $\mcF_s$ for all $0\leq s<t<\infty$.
	\end{enumerate}
\end{definition}
Every $Q$-Wiener process naturally induces a normal filtration, defined by setting, for all $t\geq 0$,
\begin{equation*}
	\tilde{\mcF}_t := \sigma(W_t\,:\,s\leq t), \quad \tilde{\mcF}^0_t:= \sigma(\tilde{\mcF}_t \cup\{ A\in \mcF\,:\,\mbP[A]=0\}), \quad \mcF_t := \bigcap_{s>t}\tilde{\mcF}^0_s.
\end{equation*}
From now on we will always consider filtrations generated in this way and so without further comment we will assume that any $Q$-Wiener process is adapted to the given filtration. A probabilistically strong solution to our SPDE will be a stochastic process $(u_t)_{t\geq 0}$ also adapted to the given filtration.\\ \par 
As in the finite dimensional case, the  $Q$-Wiener processes are martingales. Let us a fix a $T>0$ and adapt the definitions above accordingly. Then we have the following lemma.
\begin{lemma}\label{lem:QWienerMartingale}
	Let $(W)_{t\in[0,T]}$ be a $Q$-Wiener process and $(\mcF_t)_{t\in[0,T]}$ be its natural filtration. Then $(W_t)_{t\in [0,T]}\in \mcM^2_T(U)$.
\end{lemma}
\begin{proof}
	The fact that $(W_t)_{t\in [0,T]}\in \mcM^2_T(U)$ follows from the definition of $(\mcF_t)_{t\in [0,T]}$ and the assumption of independent increments of $(W_t)_{t\in[0,T]}$. If we equip $U$ with the orthonormal basis generated by $Q$, then we have
	\begin{align*}
		\sup_{t \in [0,T]}\mbE[\|W_t\|^2_U] =\sup_{t \in[0,T]} \mbE\left[ |\langle W_t,W_t\rangle|\right] &= \sup_{t \in[0,T]} \mbE\left[ \sum_{k=1}^\infty|\langle W_t,e_k\rangle|^2\right]  \\
		&= T \sum_{k=1}^\infty \langle Qe_k,e_k\rangle\\
		&=  \Tr Q\, T <\infty.
	\end{align*}
\end{proof}
\subsection{Generalised Wiener Processes}
As in subsection \ref{subsec:GenGaussinRV}, it is natural to ask if we can extend the definition of a Wiener process beyond that of Definition \ref{def:QWienerProcess}. This leads to the definition of generalised Wiener processes, which includes the white noise (or cylindrical Wiener) process.
\begin{definition}[Generalised Wiener Process]\label{def:GenWienerProcess}
	Given a Hilbert space $U$, we say that a family of linear maps $ [0,T]\times U\ni (t,h) \mapsto W_{h;t}\in \mbR$ defines a generalised Wiener process if for each $h \in U$, $t\mapsto W_{h;t}$ defines a real Wiener process and if for any sequence $(h_n)_{n\geq 1}\subset U$ converging to $h\in U$ we have,
	\begin{equation*}
		\mbE\left[|W_{h;t} -W_{h_n;t}|^2 \right] \rightarrow 0, \quad \text{ for all }t \geq 0.
	\end{equation*}
\end{definition}
We employ the same ideas as in Subsection \ref{subsec:GenGaussinRV} to understand $(W_{h;t})_{t\geq 0}$ as a Wiener process taking values in a space slightly larger than $U$. From the definition of a generalised Wiener process it follows that there exists a bi-linear form $K:U\times U\rightarrow \mbR$ and a symmetric, non-negative operator $Q\in L(U)$ such that for all $s,\,t \geq 0$ and $h,g\in U$,
\begin{equation*}
	\mbE[W_{h;t}W_{g;s}] = t\wedge s K(h,g) = t\wedge s \langle Qh,g\rangle.
\end{equation*}
We write $W_{t}$ as a formal sum,
\begin{equation*}
	W_t := \sum_{k\geq 1} \beta^k_t Qe_k,
\end{equation*}
where $(\beta^k)_{k\geq 1}$ is a family of i.i.d, standard, real, Brownian motions and $(e_k)_{k\geq 1}$ is an orthonormal basis of $U$. As in Subsection \ref{subsec:GenGaussinRV} we define the reproducing kernel of $W$, as the space $U_0:= Q^{1/2}(U)$ and we recall that for any $t\geq 0$, $h\mapsto W_{t;h}$ defines a generalised Gaussian random variable on $U$ and it takes values in any space $U_1 \hookleftarrow U_0$ such that the embedding $\iota:U_0\rightarrow U_1$ is Hilbert--Schmidt. Thus, we may consider $(W_{h;t})_{t\in [0,T],h\in U}$ as a $U_1$ valued $Q$-Wiener process on $U$ and its definition is independent of the choice of $U_1$. By definition, the generalised Wiener process is an element of $\mcM^2_T(U_1)$.
\begin{example}[Space-Time White Noise]\label{ex:WhiteNoiseProcess}
	A common example of a generalised Wiener process is the space-time white noise, sometimes called a cylindrical Wiener process. It is the natural infinite dimensional analogue of the finite dimensional white noise familiar from classical stochastic analysis. Let $\Gamma \subset \mbR^d$ be a bounded, domain and $U:= L^2(\Gamma)$. Then the space-time white noise on $[0,T]\times \Gamma$ is the generalised Wiener process $[0,T]\times L^2(\Gamma) \ni(t,h)\mapsto W_{h;t}$ such that, for all $h,\,g \in U$ and $s,\,t \in [0,T]$,
	\begin{equation*}
		\mbE[W_{t;h}]=0,\quad \mbE[W_{t;h}W_{s;g}] = t\wedge s \langle h,g\rangle_{L^2(\Gamma)}.
	\end{equation*}
	One can check that for given a family of i.i.d, standard, real Brownian motions, $(\beta^k)_{k\geq 1}$ and a orthonormal basis of $L^2(\Gamma)$, $(e_k)_{k\geq 1}$, the formal sum
	\begin{equation*}
		W_{t}:= \sum_{k\geq 1} \beta^k_te_k,
	\end{equation*}
	defines a generalised Gaussian process in the above sense, known as the cylindrical Gaussian process, and that the sum converges in any Hilbert space $U_1$ such that the embedding $\iota : L^2(\Gamma)\rightarrow U_1$ is Hilbert--Schmidt. The formal derivative, $\dd W_t = \sum_{k\geq 1} \dd \beta^k_te_k$ is referred to as the space-time white-noise.\\ \par
	An alternative approach to defining the space-time white noise is to return to Subsection \ref{subsec:GenGaussinRV} and consider the Hilbert space $U:= L_0^2([0,T]\times \Gamma)$ of square integrable space-time functions, zero at zero. Then it is an easy exercise to check that the white noise on $L_0^2([0,T]\times \Gamma)$ coincides with the space-time white noise defined above.
\end{example}
\subsection{Stochastic Integration Against $Q$-Wiener Processes}\label{subsec:QStochInt}
We conclude this chapter by giving meaning to integrals of the form 
\begin{equation}\label{eq:ItoIntegral}
	[0,T]\ni t\mapsto \int_0^t \varphi_s \dd W_s,
\end{equation}
where $(\varphi_t)_{t\in [0,T]}$ is an adapted, suitably regular, operator valued process. We first treat the case when $(W_t)_{t\in [0,T]}$ is a $Q$-Wiener process from a Hilbert space to itself and treat this case in the most detail. We follow this by defining the stochastic integral for $(W_t)_{t\in [0,T]}$ a generalised Wiener process, of which the space-time white noise will be our canonical example.\\ \par
We fix a Hilbert space, $U$ and assume that $W$ is an $U$ valued $Q$-Wiener process. We recall the definition of its reproducing kernel, the Hilbert space $U_0 := Q^{1/2}(U) \subseteq U$ equipped with the inner product $\langle Q^{-1/2}f,Q^{-1/2}g\rangle_U$. We also fix a second Hilbert space, $H$; we will subsequently assume $(\varphi_t)_{t\in [0,T]}$ takes values in a subspace of $L(U_0;H)$, so that the stochastic integral, \eqref{eq:ItoIntegral}, takes values in $H$.\\ \par
As in the finite dimensional case we first define the stochastic integral for a suitable class of simple processes, before extending it to a larger space of integrands.
\begin{definition}[Simple Processes in $L(U;H)$]\label{def:SimpleProcesses}
	We say that a process $[0,T]\ni t\mapsto \varphi_t\in L(U;H)$ is simple if, there exist a sequence of deterministic times, $0=t_0<t_1<\cdots<t_{n-1}=T$, and a set of $L(U;H)$ valued, $(\mcF_{t_m})^{n-1}_{m=0}$-measurable, random variables, $(\varphi_{m})_{m=1}^{n-1}$, taking only finitely many values, such that,
	\begin{equation*}
		\varphi_t = \varphi_0 \mathds{1}_{\{0\}}(t) + \sum_{m=0}^{n-1} \varphi_{m} \mathds{1}_{(t_m,t_{m+1}]}(t).
	\end{equation*}
	We write $\mcE_T(U;H)$ for the set of $L(U;H)$ valued simple processes on $[0,T]$.
\end{definition}
We define the space of Hilbert--Schmidt operators $L^0_2(U_0;H)$ equipped with the norm,
\begin{equation*}
	\|\varphi \|^2_{L^0_2} :=  \Tr\left[(Q^{1/2}\varphi)(Q^{1/2}\varphi)^\ast \right].
\end{equation*}
Note that since $L^0_2$ is the space of Hilbert--Schmidt operators defined on the image of a Hilbert--Schmidt operator $Q^{1/2}:U\rightarrow U$, it follows that that $L(U;H)\subset L^0_2(U,H)$ and so example the simple process $\mcE_T(U;H)$ are measurable, $L^0_2(U_0,H)$, valued processes.\\ \par 
This allows us to define the set of processes that are square integrable with respect to a $Q$-Wiener process. For a measurable process $[0,T]\ni t\mapsto \varphi_t \in L^0_2(U_0,H)$, we set,
\begin{equation*}
	\|\varphi\|_{\msH_T^2} := \mbE\left[\int_0^T \|\varphi_s\|^2_{L^0_2}\dd s \right]^{\frac{1}{2}} = \mbE\left[ \int_0^T \Tr\left[(Q^{1/2}\varphi_s)(Q^{1/2}\varphi_s)^\ast \right]\dd s\right]^{\frac{1}{2}}.
\end{equation*}
If $\varphi \in \mcE_T$ then, since we specified that the random variables $(\varphi_m)_{m=1}^{n-1}$ all take only finitely many values in $L(U;H)$, it follows directly that $\|\varphi\|_{\msH^2_T}<\infty$ for simple processes.\\ \par
We are now able to define the stochastic integral for simple process. We will then show that the map sending $\varphi \in \mcE_T(U;H)$ to its stochastic integral, \eqref{eq:ItoIntegral}, is an isometry between $\mcE_T(U;H)$ and $L^2(\Omega;H)$. Given $(W_t)_{t\in [0,T]}$ we define the linear map,
\begin{equation}\label{eq:SimpleStochIntegral}
	\begin{aligned}
		\mcE_T(U;H)&\longrightarrow  \mcM^2_T(H)\\
		\varphi &\mapsto \varphi \bigcdot W_{\,\cdot\,} := \int_0^{\,\cdot\,} \varphi _s \dd W_s := \sum_{k=1}^{n-1} \varphi_k (W_{t_{k+1}\wedge \,\cdot\,}-W_{t_k\wedge \,\cdot\,}).
	\end{aligned}
\end{equation}
We show in the next proposition that this map does indeed define an isometry from $(\mcE_T,\|\,\cdot\,\|_{\msH^2_T})$ to the space of square integrable martingales, $(\mcM_T^2(H),\,\|\,\cdot\,\|_{\mcM^2_T})$, where the norm is defined in \eqref{eq:M2Norm}.
\begin{proposition}\label{prop:SimpleStochInt}
	Let $\varphi \in \mcE_T(U;H)$. Then, $(\varphi \bigcdot W_t)_{t\in [0,T]}$, defined in \eqref{eq:SimpleStochIntegral}, is a continuous, square integrable, $H$-valued martingale, adapted to the natural filtration of $(W_t)_{t\in[0,T]}$. Furthermore, one has that
	\begin{equation}\label{eq:SimpleItoIsometry}
		\|\varphi\bigcdot W\|_{\mcM^2_T}=\sup_{t\in [0,T]}\mbE[\|\varphi\bigcdot W_t\|_H^2 ]= \|\varphi\|_{\msH^2_T}.
	\end{equation}
\end{proposition}
\begin{proof}
	The fact that $t\mapsto \varphi \bigcdot W_{t}$ defines a continuous, $H$-valued martingale is clear from the definitions. Once can check the martingale property in detail using optional stopping, see proof of \cite[Prop. 2.3.2]{prevot_rockner_07}. We check square integrabillity by proving \eqref{eq:SimpleItoIsometry} directly. Note that it suffices to show \eqref{eq:SimpleItoIsometry} for $t= t_k \in [0,T]$. Let $\zeta_k := W_{t_{k+1}}-W_{t_k}$ so that we have,
	\begin{align*}
		\mbE[\|(\varphi \cdot W)_{t}\|^2_H] &= \mbE\left[\left\|\sum_{k=0}^{n-1} \varphi_{k} \zeta_k\right\|_E^2\right]  = \mbE\left[ \sum_{k=0}^{n-1} \|\varphi_k \zeta_k\|_H^2\right] + 2\mbE\left[\sum_{\substack{k_2 =1\\k_1<k_2}}^{n-1} \langle \varphi_{k_1}\zeta_{k_1},\varphi_{k_2}\zeta_{k_2}\rangle\right].
	\end{align*}
	Considering the first term, taking each summand individually and using the definition of a $Q$-Wiener process, and $(e_k)_{k\geq 1}$ the basis of $U_0$ associated to the Hilbert--Schmidt operator $\varphi_k\in L_2^0(U_0,H)$, and defining $\varphi^\ast_k \in L_2(H;U_0)$ to be its adjoint, we have
	\begin{align*}
		\mbE[\|\varphi_k\zeta_k\|_E^2] = \sum_{m=1}^\infty \mbE[|\langle\varphi_k\zeta_k,e_m\rangle|^2] &=  \sum_{m=1}^\infty \mbE[|\langle\zeta_k,\varphi_k^\ast e_m\rangle|^2]\\
		& =  (t_{k+1}-t_k)\sum_{m=1}^\infty |\langle Q \varphi_k^\ast e_m, \varphi_k^\ast e_m\rangle|\\
		&=  (t_{k+1}-t_k) \sum_{m=1}^\infty \|Q^{1/2} \varphi_k^\ast e_m\|^2_{E}\\
		&= (t_{k+1}-t_k)  \|\varphi_{k}\|_{L^0_2}.
	\end{align*}
	Similarly, by the independence of increments, one has
	\begin{equation*}
		\mbE[\langle \varphi_{k_1}\zeta_{k_1},\varphi_{k_2}\zeta_{k_2}\rangle] =0, \quad k_1 \neq k_2.
	\end{equation*}
	Putting these two results together and summing over $k$ gives equality \eqref{eq:SimpleItoIsometry}.
\end{proof}
We now wish to extend the definition of $\varphi\bigcdot W$ to the set of all $\mbP$-a.s. $L^0_2$-predictable processes. Before doing so however, we introduce the predictable $\sigma$-algebra and the notion of a predictable process. This is necessary since we will want our stochastic integration to preserve the martingale property and the notion of a predictable process ensures it does not look into the future. Note that for simple processes this was not a problem, since almost by definition their integrals were martingales.
\begin{definition}\label{def:PredictableSigAlgebra}[Predictable $\sigma$-algebra]
	Let $(\mcF_t)_{t\in [0,T]}$ be a given filtration. Then we define, $\msP_T$, to be the $\sigma$-algebra of subsets of $[0,T]\times \Omega$ given by,
	\begin{equation*}
		\msP_T := \sigma \left(\{ (s,t]\times F\,:\, 0\leq s<t\leq T,\, F\in \mcF_s\}\cup \{ \{0\}\times F\,:\, F\in \mcF_0\}\right).
	\end{equation*}
	Given a Hilbert space, $\mcX$, we say that a process $\varphi :[0,T]\times \Omega\rightarrow \mcX$ is predictable if it is $\msP_T$-measurable.
\end{definition}
Let us define the measure space $(\Omega_T,\mcP_T,\mbP_T)$ where $\Omega_T =[0,T]\times \Omega$, $\mcP_T$ is as above and $\mbP_T = \lambda_T \otimes \mbP$ with $\lambda_T$ the Lebesgue measure on $[0,T]$. Note that the norm $\|\,\cdot\,\|_{\msH^2_T}$ defined above corresponds to the Bochner integral,
\begin{equation*}
	\|\varphi\|^2_{\msH^2_T} = \int_{\Omega_T} \|\varphi_t(\omega)\|^2_{L^0_2} \dd \mbP_T(t,\omega).
\end{equation*}
Recall that since $L^0_2$ is a separable Banach space, this Bochner integral is well defined, see Section \ref{sec:LebBochInt}.\\ \par
With $W$ fixed, let us define the space of square integrable stochastic integrands - the fact that they are suitable stochastic integrands will be proved in Theorem \ref{th:StochInt}.
\begin{definition}\label{def:StochIntegrands}[Stochastic Integrands]
	We say that $[0,T] \times \Omega \ni (t,\omega ) \mapsto \varphi_t(\omega)\in L^0_2(U_0,H)$ is stochastically integrable with respect to $W$ if it is $\msP_T$ measurable with respect to the natural filtration of $W$ and is such that $\|\varphi\|_{\msH^2_T}<\infty$. We denote this space by $\msH^2_T(W)$ and observe that it forms a Hilbert space with inner product $\langle \varphi,\tilde{\varphi}\rangle_{\msH^2_T} =  \int_0^T \langle \varphi_t,\tilde{\varphi}_t\rangle_{L^0_2}\,\dd t$
\end{definition}
The following proposition allows us to approximate predictable elements of $\msH^2_T(W)$ by simple processes in $\msH^2_T(W) \cap \mcE_T$.
\begin{proposition} \label{prop:SimplePredictApprox} 
	The following statements both hold:
	\begin{enumerate}[label=\roman*)]
		\item If a mapping $\varphi:[0,T]\times \Omega\rightarrow L(U;H)$ is $L(U;H)$-predictable, then it is also $L^0_2$-predictable. \label{it:LPredictIsL_2^0Predict}
		\item If $\varphi \in \msH^2_T(W)$ is $L^0_2$-predictable then there exists a sequence of simple processes $(\varphi^n)_{n\geq 1}$ such that $\displaystyle \lim_{n\rightarrow \infty}\|\varphi-\varphi^n\|_{\msH^2_T} =0$. \label{it:SimpleApprox}
	\end{enumerate}
\end{proposition}
\begin{proof}
	To prove \ref{it:LPredictIsL_2^0Predict}, first recall from Remark \ref{rem:HilbSchmidtBasis} that given bases $(e_k)_{k\geq 1},\,(f_k)_{k\geq 1}$ of $U$ and $H$ respectively, the family $(e_k\otimes f_k)_{k\geq 1}$ is a basis of $L_2(U;H)$. So it follows that $L(U;H)$ is a dense, linear subspace of $L_2(U;H)$. Furthermore, since $Q^{1/2}(U) = U_0 \subset U$ has continuous image in $U$, it follows that $L(U;H)\subset L_2(U_0;H)= L^0_2$ with continuous embedding. By an adaptation of Theorem \ref{th:OpExpectationCommute} it follows that any $L(U;H)$ valued process is also an $L^0_2$ valued process, with continuous image. Since predictability is preserved by continuous transformations the first claim follows.\\ \par
	As in the proof of \ref{it:LPredictIsL_2^0Predict} we use the fact that $L(U;H)$ is densely embedded in $L^0_2$ and Proposition \ref{prop:SimpleApprox} that for $\mbP_T$-a.e $(t,\omega)\in \Omega_T$ there exists a sequence of simple $L(U;H)$ valued random variables, $(\varphi_n)_{n\geq 1}\subset \mcE_T$ such that
	\begin{equation*}
		\|\varphi_t(\omega) - \varphi_{n;t}(\omega)\|_{L^2_0} \rightarrow 0, \quad \mbP_T\,{\text{-a.s.}}.
	\end{equation*}
It follows from the assumptions, that therefore $\|\varphi-\varphi_n\|_{\msH^2_T}\rightarrow 0$. Note, however, that we have not yet treated predictability. So to conclude it suffices to show that any predictable set $A \in \msP_T$ can be approximated arbitrarily by predictable events in $\mbP_T$. However, this follows from Dynkin's lemma, since the predictable sets form a $\pi$-system which generates $\mcP_T$. See the last paragraph in the proof of \cite[Prop. 4.22]{daprato_zabczyk_14} for details.
\end{proof}
Therefore, for fixed $(W_{t})_{t\in [0,T]}$, we have defined an isometry, $\msH^2_T(W) \ni \varphi \mapsto (\varphi \bigcdot W) \in \mcM^2_T(U)$, on a linear, dense subspace. This allows us to extend the map, first to all of $\msH^2_T(W)$ and then to all $L^0_2$-predictable processes, $\varphi$ such that
\begin{equation}\label{eq:FullIntegrandCondition}
	\mbP\left[ \int_0^T \|\varphi_s\|^2_{L^0_2}\dd s <\infty\right] =1.
\end{equation}
We denote this, the full space of integrands, by $\msH_{T}(W)$.
\begin{theorem}\label{th:StochInt}
	Given $\varphi \in \msH^2_T(W)$, the stochastic integral,
	\begin{equation}\label{eq:QWienerStockInt}
		[0,T]\ni t\mapsto \int_0^t \varphi_s\dd W_s =: \varphi \bigcdot W_t,
	\end{equation}
	defines a continuous, square integrable, $H$-valued martingale. Furthermore, if instead $\varphi \in \msH_T(W)$ then \eqref{eq:QWienerStockInt} defines a continuous, local martingale.
\end{theorem}
\begin{proof}
	Combining Proposition \ref{prop:SimplePredictApprox} and the isometry property, \eqref{eq:SimpleItoIsometry}, for the It\^o integral defined for simple processes, allows us to extend the integral to all of $\msH^2_T(W)$ by density. Since the martingale property and continuity of paths are preserved under convergence in $L^2(\Omega)$, it follows that the integral takes values in $\mcM^2_T(H)$.\\ \par 
	To extend the integral to all of $\msH_T(W)$, we first claim that for $\tau$ an $(\mcF_t)_{t\in [0,T]}$ stopping time and $\varphi \in \msH^2_T(W)$, it holds that,
	\begin{equation}\label{eq:StoppedIntegral}
		\int_0^T \mathds{1}_{[0,\tau]}(s)\varphi_s \dd W_s = \varphi \bigcdot W_{t\wedge \tau}, \quad \text{ for all }t\in [0,T], \,\, \mbP\text{-a.s.} 
	\end{equation}
	In other words, stopping the integrand at $\tau$ amounts to stopping the integral. Identity \eqref{eq:StoppedIntegral} can be shown by approximating $\tau$ by simple stopping times (stopping times taking only finitely many values) and $\varphi$ by simple processes and then passing to the limits. For details see \cite[Lem. 4.24]{daprato_zabczyk_14}. So then for $\varphi \in \msH_T(W)$, i.e such that \eqref{eq:FullIntegrandCondition} holds, let us define,
	\begin{equation}\label{eq:LocalisingSequence}
		\tau_n := \inf \left\{ t \in [0,T]\,:\, \int_0^t \|\varphi_s\|^2_{L^0_2}\dd s \geq n\right\}, \quad n\geq 0,
	\end{equation}
	with the convention that $\tau_n = T$ if the set is empty. Therefore, by \eqref{eq:StoppedIntegral}, the stochastic integrals $\mathds{1}_{[0,\tau_n]}\varphi \bigcdot W = \varphi \bigcdot W_{\,\cdot\, \wedge \tau_n}$ are well defined. Furthermore, due to \eqref{eq:FullIntegrandCondition}, the sequence $(\tau_n)_{n\geq 0}$ is $\mbP$-a.s. increasing so for any $t\in [0,T]$ and $\mbP$-a.a. $\omega \in \Omega$ there exists an $n(\omega)\geq 0$ such that $t\leq \tau_n(\omega)$. Therefore, the integral $t\mapsto \varphi  \bigcdot W_t = \mathds{1}_{[0,\tau_n]}(t) \varphi \bigcdot W_t$, with $t\leq \tau_n$, is $\mbP$-a.s. well defined for all $\varphi\in \msH_T(W)$ and the definition is independent of the chosen sequence $(\tau_n)_{n\geq 1}$ satisfying \eqref{eq:LocalisingSequence}. Since $\msH_T(W)$ contains elements which are not stochastically square integrable on $[0,T]$, the resulting stochastic integral is only a local martingale, with the localisation given by a sequence of the form \eqref{eq:LocalisingSequence}.
\end{proof}
%
%
\subsection{Stochastic Integration Against Generalised Wiener Processes}\label{subsec:GenStochInt}
We recall that given a, non-negative, symmetric, bounded, linear operator $Q \in L(U)$, there exists a generalised Wiener process $(W_t)_{t\in [0,T]}$ on $U$, with covariance $Q$. Furthermore, there exists a pair of Hilbert spaces $U_0:= Q^{1/2}(U),\, U_1$ such that $U_0 \hookrightarrow U_1$ and the embedding $\iota U_0\rightarrow U_1$ is Hilbert--Schmidt. As a result, the generalised Wiener process, $[0,T]\times U\ni(t,h)\mapsto W_{h;t}$, can be seen as  taking values in $U_1$, although the particular choice of $U_1$ is not important.\\ \par 
As such, we may almost immediately define the stochastic integral with respect to a generalised Wiener process for all predictable, processes, $(\varphi_t)_{t\in [0,T]} \in L_2(Q^{1/2}(U_1);H)$ such that,
\begin{equation*}
	\mbP\left[\int_0^T \|\varphi_s\|^2_{L_2(Q^{1/2}(U_1);H)}\dd s <\infty \right] =1.
\end{equation*}
It suffices to replace $U_0$ with $Q^{1/2}(U_1)$ in the previous section. However, this is unsatisfactory, since we know that the definition of the generalised Wiener process is independent of the choice $U_1$ and so it would be natural for the integral to retain this property. It turns out that we can retain the same space of integrands as in Subsection \ref{subsec:QStochInt}. Note that $U_0:= Q^{1/2}(U)$ is still well defined in this case, only now it does not hold that the embedding $U_0 \hookrightarrow U$ is Hilbert--Schmidt.
\begin{proposition}\label{prop:Q1Isom}
	Let $(e_k)_{k\geq 1}$ be an orthonormal basis of $U$, $(\beta^k)_{k\geq 1}$ be a family of i.i.d, standard, real Brownian motions, $Q \in L(U)$ be non-negative and symmetric and $\iota:Q^{1/2}(U)=:U_0\rightarrow U_1$ define a Hilbert--Schmidt embedding as above. Then, defining the non-negative, symmetric and trace class operator, $Q_1:= \iota Q \iota^\ast \in L_1(U_1)$, it follows that,
	\begin{equation}\label{eq:EmbeddedGenWiener}
		[0,T]\ni t\mapsto W_t = \sum_{k=1}^\infty \beta^k_t \iota Q^{1/2}e_k,
	\end{equation}
	defines $Q_1$-Wiener process in $\mcM^2_T(U_1)$. Moreover, one has that $Q_1^{1/2}(U_1) = \iota(U_0) \subset U_1$ and for all $h_0 \in U_0$,
	\begin{equation*}
		\|h_0\|_{U_0} = \|Q_1^{-1/2}\iota h_0\|_{U_1} = \|\iota h_0\|_{Q^{1/2}_1U_1},
	\end{equation*}
	where $Q^{-1/2}_1$ is the pseudo-inverse of $Q_1^{1/2}$, see Definition \ref{def:PseudoInverse}. In other words, $\iota :U_0\rightarrow Q^{1/2}_1U_1$ is an isometry.
\end{proposition}
\begin{proof}
	By assumption, $\iota Q^{1/2}$, defines a trace class operator from $U$ to $U_1$ and so it follows that the right hand side of \eqref{eq:EmbeddedGenWiener} defines a $U_1$ valued Wiener process. To see that $W$ defines a generalised Wiener process on $U$ in the sense of Definition \ref{def:GenWienerProcess} we may apply the same arguments as in the proof of Proposition \ref{prop:HSExtension}. The independence of the definition on $U_1,\,\iota$ also follows from the same arguments and we recall, using, Corollary \ref{cor:SquareRootOfSquare}, that $\iota:U_0 \rightarrow Q^{1/2}_1U_1$ is an isometry.
\end{proof}
\begin{theorem}\label{th:GenStochInt}
	Let $(W_t)_{t\in [0,T]}$ be a generalised Wiener process on $U$, $[0,T] \ni t \mapsto \varphi_t \in L_2(U_0;H) =:L^0_2$ be a stochastic process, such that,
	\begin{equation*}
		\mbP \left[ \int_0^T \|\varphi_t\|^2_{L^0_2}\dd t <\infty\right]=1.
	\end{equation*}
	Then, letting $U_1$, $\iota:U_0\rightarrow U_1$ be as above, the stochastic integral defined by setting
	\begin{equation}\label{eq:GenItoIntegral}
		[0,T]\ni t\mapsto \varphi \bigcdot W_t := \int_0^t (\varphi_s \circ \iota^{-1})\dd W_s,
	\end{equation}
	is a continuous, $H$-valued, square integrable, local martingale.
\end{theorem}
\begin{proof}
	By Proposition \ref{prop:Q1Isom}, for $Q_1 = \iota \iota^\ast\in N(U_1)$, we have $Q^{1/2}_1(U_1)= \iota (U_0)$ and by the polarisation identity, for any $\varphi_0,\psi_0 \in U_0$,
	\begin{equation*}
		\langle \iota \varphi_0,\iota\psi_0\rangle_{Q^{1/2}_1U_1} = \langle Q^{-1/2}_1\iota \varphi_0,Q^{-1/2}_1\iota \psi_0\rangle_{U_1}= \langle \varphi_0,\psi_0\rangle_{U_0}.
	\end{equation*}
	In particular, given $(e_k)_{k\geq 1}$ a basis of $U_0$, it follows that $(\iota e_k)_{k\geq 1}$ is a basis of $Q^{1/2}_1(U_1)$. Hence we have the equivalence,
	\begin{equation*}
		\varphi \in L_2(U_0;H) \quad \iff \quad \varphi \circ \iota^{-1} \in L_2(Q_1^{1/2}(U_1);H),
	\end{equation*}
	and so the stochastic integral in \eqref{eq:GenItoIntegral} is well defined as a continuous, $H$-valued, square integrable, local martingale, by the results of Subsection \ref{subsec:QStochInt}. The independence on the choice of $U_1$ and $\iota$ is a consequence of the same result for the generalised Wiener process.
\end{proof}
For another approach to the construction of stochastic integrals against generalised Wiener processes, we refer to \cite[Subsec. 4.2.1]{daprato_zabczyk_14}. For more details on this approach, see \cite{prevot_rockner_07}.
	
		\chapter{Variational Approach to PDE}\label{ch:PDE}
In this chapter we present an introduction to the variational approach to partial differential evolution equations. The method relies on viewing time dependent PDE in spirit as ordinary differential equations in infinite dimensional spaces. We use this is our starting point for viewing SPDE as infinite dimensional stochastic differential equations in the next chapter. A useful outcome of this approach to PDE and SPDE is a natural approach to numerical analysis of such problems through the finite element method. This approach is also sometimes known as the Galerkin method.
\section{Unbounded Operators Between Banach Spaces}
We begin by recalling the notion of unbounded operators on Banach spaces. We refer to \cite[Sec. 2.6]{brezis_11} for more details.
\begin{definition}[Unbounded Linear Operators]\label{def:UnboundedOp}
	Given Banach  spaces $E,\,F$, an unbounded, linear operator $A$ from $E$ to $F$ is a pair $(A,D(A))$, where $D(A)\subseteq E$ is a linear subspace and $A:D(A)\rightarrow F$ is a linear map. We write write $\msL(E;F)$ for the set of unbounded linear operators from $X$ to $Y$, leaving the dependence on defining a domain inside $X$ implicit in the notation.
\end{definition}
We say that $A$ is densely defined if $D(A)$ is a dense subset of $E$. If $A$ is densely defined from $E$ to $F$ and continuous on its domain then there exists a unique, continuous extension of $A$ defined on all of $E\rightarrow F$.\\ \par 
We define the graph of $A$, $G_A := \{ (x,y) \in E\oplus F\,:\, y = Ax\}$ and say that $A$ is a closed operator if $G_A \subset E\oplus F$ is a closed set. Explicitly, $A$ is closed if and only if for any sequence $\{(x_n,Ax_n)\}_{n\geq 0}\subset E \oplus F$ such that $x_n \rightarrow x$ and $Ax_n \rightarrow y$, we have $x \in D(A)$ and $y=Tx$.
\begin{example}\label{ex:CtsFuncDerivative}
	Let $C([0,1];\mbR)$ be the set of continuous, real valued functions on the unit interval, equipped with the norm $\|f\|_{\infty}:= \sup_{x \in [0,1]}|f(x)|$. Then, for $f \in C([0,1];\mbR)$, define $Af:= \frac{d}{dx}f $ and
	$$D\left(\frac{d}{dx}\right):= C^1([0,1];\mbR)\subset C([0,1];\mbR),$$
	the set of continuously differentiable maps $f:[0,1]\rightarrow \mbR$ equipped with the norm $\|f\|_{C^1} := \|f\|_{\infty}+ \|\frac{d}{dx}f\|_{\infty}$. It follows that $(\frac{d}{dx},C^1([0,1];\mbR))$ is an unbounded, linear operator from $C([0,1];\mbR)$ to itself in the sense of Definition \ref{def:UnboundedOp}.\\ \par 
	By Stone--Weierstrass, $\frac{d}{dx}$ is densely defined, however, it is also unbounded on its domain and so cannot be continuously extended to all of $C([0,1];\mbR)$. To see this, define the sequence $f_n (x):= \sin (nx)$, we have $\|f_n\|_{\infty} \equiv 1$ but $\|\frac{d}{dx}f_n \|_{\infty} = n \rightarrow \infty$. On the other hand, it is easy to see that $\frac{d}{dx}$ is a closed operator, since if $\|f-f_n\|_{C^1}\rightarrow 0$, then by definition $\|\frac{d}{dx}(f-f_n)\|_{\infty} \rightarrow 0$.
\end{example}
\begin{example}[Dirichlet Laplacian]\label{ex:DirLaplace}
	Let $\Gamma\subset \mbR^d$ be a smooth, bounded domain and define the second order differential operator, for mappings $f :\Gamma\rightarrow \mbR$,
	\begin{equation*}
		\Delta f := \sum_{i=1}^d \partial^{2}_{ii} f.
	\end{equation*}
	We consider $\Delta$ as an unbounded operator on the set,
	\begin{equation*}
		\mcH^1_0(\Gamma):= \left\{f :\Gamma \rightarrow \mbR\,:\,  \|f\|_{L^2(\Gamma)}+ \sup_{i=\{1,\ldots,d\}}\|\partial_i  f\|_{L^2(\Gamma)} <\infty,\, f|_{\partial \Gamma}=0\right\},
	\end{equation*}
	with domain of definition, $D(\Delta)$,
	$$C^2_c(\Gamma):=\left\{ f :\Gamma\rightarrow \mbR\,:\, \sup_{i,j=\{1,\,\ldots,d\}} \|\partial_{ij} f\|_{\infty}<\infty,\, \supp(f) \text{ is compact in }\Gamma \right\}.$$
	The lack of boundedness in Example \ref{ex:CtsFuncDerivative} arose because we chose too small a target space for the derivative map. In order that $\Delta$ remain bounded on its domain in this case, we define $\mcH^{-1}(\Gamma):= (\mcH(\Gamma))^\ast$, which can be explicitly represented as,
	\begin{equation}\label{eq:HMinusOneDef}
		\mcH^{-1}(\Gamma):= \left\{ h = f+ \sum_{i=1}^d \partial_i g,\, f,\,g \in L^2(\Gamma;\mbR) \right\}.
	\end{equation}
	Then we consider the unbounded operator, $\Delta:C^2_c(\Gamma)\subset \mcH^1_0(\Gamma)\rightarrow \mcH^{-1}(\Gamma)$. It follows from the density of $C^2_c(\Gamma)$ in $\mcH^1_0(\Gamma)$ and the fact that $\Delta$ is continuous as a map from $C^2_c(\Gamma)$ to $\mcH^{-1}(\Gamma)$ that there exists a unique extension of $\Delta:\mcH^1_0(\Gamma)\rightarrow \mcH^{-1}(\Gamma)$. 
\end{example}
Note that by setting $g=0$ in \eqref{eq:HMinusOneDef}, we see that $L^2(\Gamma)\subset \mcH^{-1}(\Gamma)$. Furthermore, it is clear that $\mcH^1_0(\Gamma)\subset L^2(\Gamma)$ and so we have the ordering,
\begin{equation*}
	\mcH^1_0(\Gamma)\subset L^2(\Gamma)\subset \mcH^{-1}(\Gamma)=(\mcH^1_0(\Gamma))^\ast.
\end{equation*}
This is an example of a Gelfand triple, and $\Delta:\mcH^{1}_0(\Gamma)\rightarrow 
\mcH^{-1}(\Gamma)$ defines a continuous, coercive operator. We explore this setting in more general abstraction in the next section.
\section{Linear and Coercive PDE}\label{sec:LinearPDE}
Let $V$ be a separable, reflexive, Banach space with $V\subset H$ a separable Hilbert space. It is readily checked that for Banach spaces $V,\,H$, one has,
\begin{equation*}
	V\subset H\quad \iff \quad H^\ast \subset V^\ast.
\end{equation*}
However, for $H$ a Hilbert space, since we also have $H \cong H^\ast$, it follows that
\begin{equation*}
	V \subset H \cong H^\ast \subset V^\ast.
\end{equation*}
The triple, $(V,H,V^\ast)$ is called a \textit{Gelfand triple}. In what follows we write $\|\,\cdot\,\|_{V}$ for the canonical norm on $V$ and $|\,\cdot\,|_{H}$ for the canonical norm on $H$. We use $\langle\,\cdot\,,\,\cdot\,\rangle$ to denote both the duality pairing between $V,\,V^\ast$ and the inner product on $H$ and interpret this notation accordingly wherever it occurs. Since the two expressions agree only when both arguments lie in $H$ this should not cause too much confusion. We recall the following important theorem, regarding weak compactness in normed vector spaces.
\begin{theorem}[Banach--Alaoglu]\label{th:BanachAlaoglu}
	Let $E$ be a normed vector space. Then the closed unit ball in $E^\ast$ is compact with respect to the weak$-\ast$ topology. Furthermore, $E$ is reflexive if and only if the closed unit ball in $E$ is weakly compact.
\end{theorem}
\begin{proof}
	See \cite[Thm. 3.16 \& Thm. 3.17]{brezis_11}.
\end{proof}
\begin{remark}
	It follows from Theorem \ref{th:BanachAlaoglu}, that the closed unit ball of any Hilbert space is weakly compact.
\end{remark}
Let $A:D(A)\subset V \rightarrow V^\ast$ be an unbounded linear operator, $T>0$ and $B\in L^2([0,T];V^\ast)$ be a square integrable map. Then we consider the deterministic, linear PDE
\begin{equation}\label{eq:AbstractLinearPDE}
	\begin{cases}
		\partial_t u - A u = B_t, \\
		u\tzero =u_0.
	\end{cases}
\end{equation}
One might be concerned that \eqref{eq:AbstractLinearPDE} does not involve a specified spatial domain. However, this is intentional and we leave any specification of spatial domain and any boundary conditions to the definition of the spaces $V$ and $H$. It may be that in particular cases of interest more care must be taken, for example when considering domains with non-standard geometry or more involved boundary conditions.\\ \par
In order to give a proper definition of weak solutions to \eqref{eq:AbstractLinearPDE} on $[0,T]$ we need a preliminary definition of weak derivatives in time.
\begin{definition}[Weak Derivative in Time]\label{def:WeakTimeDerivative}
	Let $E$ be a real Banach space and $u \in L^1([0,T];E)$. Then we say that $v \in L^1([0,T];E)$, is the weak derivative in time of $u$ and write $v= \frac{d}{dt}u$, if for any $\varphi \in C^\infty_c((0,T);\mbR)$,
	\begin{equation*}
		\int_0^T \frac{d}{dt}\varphi_t u_t\dd t = -\int_0^T \varphi_t v_t\dd t\in E,
	\end{equation*}
	where both integrals are understood as Bochner integrals.
\end{definition}
We are now in a position to give a meaningful definition of weak solutions to \eqref{eq:AbstractLinearPDE}.
\begin{definition}[PDE Weak Solution]\label{def:PDEWeakSol}
	We say that $u \in L^2([0,T];V)$ is a weak solution to \eqref{eq:AbstractLinearPDE} if $\frac{d}{dt}u \in L^2([0,T];V^\ast)$ and for any $v \in D(A)\subset V$ and $t\in (0,T]$, one has the identity,
	\begin{equation}\label{eq:PDEWeakForm}
		\left\langle	\frac{\dd}{\dd t} u_t,v\right\rangle =\langle A u_t, v\rangle+  \langle B_t,v\rangle. 
	\end{equation}
	In other words, a weak solution is such that the identity, $\frac{d}{dt}u = Au_t+B_t$, holds in $V^\ast$.
\end{definition}
\begin{remark}
	Note that the duality pairings on the right hand side are all well defined since by assumption $A$ takes values in $V^\ast$ on $V$ and $B \in L^2([0,T];V^\ast)$. \end{remark}
\begin{remark}
	One could alternatively phrase \eqref{eq:PDEWeakForm} in integral form and avoid introducing Definition \ref{def:WeakTimeDerivative} at this stage. This is the perspective that we take in Chapter \ref{ch:SPDE}, where the inclusion of an It\^o integral term in the PDE makes the integral form more informative.
\end{remark}
Our starting point is to show that under the assumption of \textit{coercivity} on $A$, there exists a unique, weak solution to \eqref{eq:AbstractLinearPDE}.
\begin{assumption}[Coercivity I]\label{ass:CoercivityI}
	There exist $\lambda,\,\alpha>0$ such that for any $u \in V\subset H$,
	\begin{equation}\label{eq:Coercivity}\tag{H1}
		2 \langle Au,u\rangle + \alpha \|u\|_V^2 \leq \lambda |u|^2_H 
	\end{equation}
\end{assumption}
We will require the following lemma.
\begin{lemma}[Lions--Magenes]\label{lem:LionsMagenes}
	Let $(V,H,V^\ast)$ be a Gelfand triple as above and $u \in L^2([0,T];V)$ be such that $\frac{d}{dt}u \in L^2([0,T];V^\ast)$. Then $u \in C([0,T];H)$, the map $[0,T]\ni t\mapsto |u_t|_{H}\in \mbR$ is absolutely continuous and 
	\begin{equation*}
		\frac{\dd }{\dd t}|u_t|^2_H = 2 \left\langle \frac{\dd}{\dd t}u_t, u_t\right\rangle, \quad \text{for a.a. } t \in [0,T].
	\end{equation*}
\end{lemma}
\begin{proof}[Sketch of Proof]
	The proof follows by first convolving $u$ with a standard mollifier in $t$, to give $u^\varepsilon:= u\ast_t\eta_\varepsilon \in C^\infty([0,T];V)$ and $|u^\varepsilon|_H\in C^\infty([0,T];\mbR)$. Since the weak derivative agrees with the strong derivative when the latter is well defined, the results of Lemma \ref{lem:LionsMagenes} obtain directly for $u^\varepsilon$. Then using the ordering $V\subset H\subset V^\ast$ and assumptions on $u$, the results also hold in  the limit. For a detailed exposition in the case $H=L^2$, $V=H^{1}_0$ and $V^\ast = H^{-1}$ see \cite[Thm. 3; Sec. 5.9.2]{evans_10} 
\end{proof}
\begin{theorem}\label{th:LinearPDE}
	Let $T>0$, $u_0\in H$ and $B \in L^2([0,T];V^\ast)$ and $A\in \msL(V;V^\ast)$ satisfy \eqref{eq:Coercivity} and be continuously defined from $V\rightarrow V^\ast$. Then there exists a unique weak solution $u \in L^2([0,T];V)\cap  C([0,T];H)$ to \eqref{eq:AbstractLinearPDE}.
\end{theorem}
\begin{proof}
	We prove uniqueness of solutions first, using the coercivity assumption.
	\paragraph{Uniqueness:} Let $u,\,w $ be two weak solutions to \eqref{eq:AbstractLinearPDE} in the sense of Definition \ref{def:PDEWeakSol}. Then by linearity, for any $t\in (0,T]$ and $v\in D(A)$, we have
	\begin{equation*}
		\left\langle \frac{\dd}{\dd t}(u_t-w_t),v\right\rangle =  \langle A(u_t-w_t),  v\rangle 
	\end{equation*}
	Applying Lemma \ref{lem:LionsMagenes} and the coercivity assumption, \eqref{eq:Coercivity}, we see that
	\begin{equation*}
		|u_t-w_t|_H^2= 2\int_0^t \langle A(u_s-w_s),u_s-w_s\rangle \dd s \leq \lambda \int_0^t |u_s-w_s|^2_H \dd s .
	\end{equation*}
	Then applying Gr\"onwall shows that $|u_t-w_t|_H = 0$ for all $t \in [0,T]$, from which it follows that $u=w \in L^2([0,T];V)$.
	\paragraph{Existence:} We obtain existence by a generalisation of the argument described in Subsection \ref{subsec:WorkedExample}. Let $\{e_k\}_{k\geq 1}\subset V$ be an orthonormal basis of $H$. For $n\geq 1$ we define,
	\begin{equation*}
		V_n := \text{span} \{ e_1,\ldots,e_n\}.
	\end{equation*}
	For any $n\geq 1$,  we consider the finite dimensional system of linear, real, ODEs,
	\begin{equation}\label{eq:GalerkinODE}
		\begin{aligned}
			\frac{\dd }{\dd t}\langle u_{n;t},e_k\rangle& = \langle Au_{n;t}, e_k\rangle + \langle B_t,e_k\rangle,\\
			\langle u_{n;0},e_k\rangle &= \langle u_0,e_k\rangle,
		\end{aligned}\quad \text{for }k=1,\ldots,n.
	\end{equation}
	For any $n\geq1$, there exists a unique, classical solution to the system of ODEs, \eqref{eq:GalerkinODE}. Furthermore, setting
	$$[0,T]\ni t\mapsto u_{n;t}:= \sum_{k=1}^n \langle u_{n;t},e_k\rangle e_k,$$
	defines a weak solution to \eqref{eq:AbstractLinearPDE} with initial data $u_{n;0}$. We now derive a strong enough \textit{a priori} bound on the family $(u_n)_{n\geq 1}$ to allow us pass to the limit. Since $u_{n;0}\rightarrow u_0$ and the right hand side of \eqref{eq:AbstractLinearPDE} is linear this limit will define a weak solution to \eqref{eq:AbstractLinearPDE}.\\ \par 
	Note that by definition we have
	\begin{equation*}
		\frac{\dd}{\dd t} \langle u_{n;t},e_n\rangle = \left\langle \frac{\dd}{\dd t} u_{n;t},e_n\right\rangle,
	\end{equation*}
	where the first derivative is understood as a usual finite dimensional derivative and the second in the sense of Definition \ref{def:WeakTimeDerivative}. Therefore, applying Lemma \ref{lem:LionsMagenes}, we observe the identity,
	\begin{equation*}
		|u_{n;t}|_H^2 = \sum_{k=1}^n\langle u_0,e_k\rangle^2 + 2 \int_0^t \langle Au_{n;s}+B_s,u_{n;s}\rangle \dd s.
	\end{equation*}
	So by the coercivity assumption, \eqref{eq:Coercivity}, the bound $\langle B,v\rangle \leq \|B\|_{V^\ast}\|v\|_{V}$ and Young's product inequality we see that for some $\lambda,\,\alpha>0$,
	\begin{equation*}
		|u_{n;t}|_H^2 + \alpha \int_0^t \|u_{n;s}\|^2_V \dd s 
		\, \leq \,|u_{n;0}|_H^2 + \int_0^T \|B_s\|^2_{V^\ast}\dd s + (\lambda +1)\int_0^t |u_{n;s}|^2 \dd s.
	\end{equation*}
	Applying Gr\"onwall, and taking suprema on both sides, gives the bound
	\begin{equation*}
		\sup_{t\in [0,T]}|u_{n;t}|_H^2 + \alpha \int_0^T\|u_{n;s}\|^2_V \dd s \,\leq \, \left(|u_{n;0}|_H^2+ \int_0^T\|B_s\|^2_{V^\ast}\dd s\right)e^{(\lambda +1)T}.
	\end{equation*}
	We note that $|u_{n;0}|_H \leq |u_0|^2_H$ and so the right hand side is bounded uniformly in $n \geq 1$. Thus we also have,
	\begin{equation}\label{eq:LinPDEEnergy}
		\begin{aligned}
			\sup_{n\geq 1} \left(\sup_{t \in [0,T]} |u_{n;t}|^2_H + \alpha \int_0^T \|u_{n;s}\|^2_{V} \dd s \right) \leq \left(|u_0|^2_H+\int_0^T \|B_s\|^2_{V^\ast}\dd s\right)e^{(\lambda+1)T} <\infty.
		\end{aligned}
	\end{equation}
	This shows that the set $(u_{n})_{n=1}^\infty$ lies in a closed, bounded subset of $C([0,T];H)\cap L^2([0,T];V)$. Since $V$ was assumed to be reflexive, from Theorem \ref{th:BanachAlaoglu}, there exists a weakly convergent subsequence (which we do not relabel) such that $u_n \rightharpoonup u \in L^2([0,T];V)$. Since $A:V\rightarrow V^\ast$ is strongly continuous, it is also continuous with respect to the weak topology on $V$ and so we have $A u_n \rightharpoonup Au \in L^2([0,T];V^\ast)$. Thus, we may pass to the limit in \eqref{eq:GalerkinODE} to see that, for any $v \in D(A)\subset V$,
	\begin{equation*}
		\left\langle \frac{\dd}{\dd t} u_{t},v\right \rangle = \langle A u_t,v\rangle + \langle B_t,v\rangle.
	\end{equation*}
	Furthermore, by definition $\frac{\dd}{\dd t} u_t = A u_t + B_t \in V^\ast$ means that $\frac{\dd}{\dd t} u \in L^2([0,T];V^\ast)$ and so we may apply the first part of the proof to establish that $u$ defined as the weak limit of (a subsequence of) the $u_n$ is the unique weak solution to \eqref{eq:AbstractLinearPDE}. Furthermore, it follows from Lemma \ref{lem:LionsMagenes} that $u \in C([0,T];H)$.
\end{proof}
\begin{example}
	To a linear, self-adjoint, coercive operator $A\in \msL(V;V^\ast)$ as above, one can associate a natural scale of $L^2$ based function spaces that resemble the Sobolev spaces. For $k \in \mbN$, one defines $H_A^k:=D(A^{k/2}):=\{ u\in L^2\,:\, A^{k/2}u \in L^2\}$. Here $A^{1/2}:V\rightarrow L^2$ is the operator such that for any $u,v \in V$, $\langle Au,v\rangle = \langle A^{1/2}u,A^{1/2}v\rangle$. In this scale, a natural Gelfand triple is given by $(H_A^{1},L^2,H^{-1}_A)$ where $H^{-1}_A:= (H^1_A)^\ast$. Note that when $A=\Delta$ this scale exactly defines the Hilbertian Sobolev spaces $H^{k}$ and their topological duals. In the following section where we consider non-linear $A(u)$, it will usually be the case that $A(u)=A_1 u + A_2(u)$ with $A_1$ satisfying the above assumptions and $A_2$ non-linear and of lower order. In that case one defines the natural scale using $A_1$.
\end{example}
We now turn to the question of non-linear equations, that is we no longer assume $A:V\rightarrow V^\ast$ to be a linear operator. As remarked above, we typically assume $A = A_1 u + A_2 u$ where $A_1$ is linear and coercive and $A_2$ is non-linear and often of lower order. The main issue in the above argument is that even for strongly continuous non-linear operators $A:V\rightarrow V^\ast$, it is not in general the case that they are also weakly continuous. Since the subsequence we extract is only weakly convergent we would not be able to conclude that $Au_n\rightarrow Au$. To remedy this we are required to make additional assumptions. We present two general approaches; additional assumptions on $A$ in the form of monotonicity and regularity, or additional assumptions on the spaces in question, $(V,H,V^\ast)$.
\section{Monotonicity and Weak Continuity}
Throughout we do not impose any linearity on $A$ but keep a triple $V\subset H\subset V^\ast$ as in Section \ref{sec:LinearPDE}.
\begin{assumption}[Coercivity II]\label{ass:CoerciveNLPDE}
	There exist $\lambda,\alpha>0$ such that for any $u,\,v \in V$,
	\begin{equation}\label{eq:CoerciveNLPDE}\tag{H$1$}
		2\langle A(u),u\rangle + \alpha\|u\|^2_{V} \leq \lambda|u|^2_H.
	\end{equation}
\end{assumption}

\begin{assumption}[Monotonicity]\label{def:MonotoneOperator}
	For any $\lambda>0$ and $u,w \in V$,
	\begin{equation}\label{eq:MonotoneOperator}\tag{H2}
		\langle A(u) -A (w),u-w\rangle \leq \frac{\lambda}{2}|u-w|^2_H.
	\end{equation}
\end{assumption}
\begin{remark}
	Note that if $u,w$ solve \eqref{eq:AbstractLinearPDE}, then they also solve,
	\begin{equation*}
		\partial_t u^\lambda - A^\lambda (u^\lambda)= B_t,\quad u^\lambda\tzero =u_0,
	\end{equation*}
	where $A^\lambda(u):= e^{- t\lambda/2}A(e^{t\lambda/2}u) - \frac{\lambda}{2}u$ and $u_t^\lambda := e^{t\lambda/2}u_t$. Then it is easily seen that if $A$ satisfies \eqref{eq:CoerciveNLPDE} and \eqref{eq:MonotoneOperator}, there exist $\alpha>0$ and $\nu\geq 0$, possibly new, such that for all $u,w\in V$
	\begin{equation}\label{eq:CoercivePrime}\tag{H$1'$}
		2\langle A^\lambda(u),u\rangle + \alpha \|u\|^2_{V}\leq  0,
	\end{equation}
	\begin{equation}\label{eq:MonotonePrime}\tag{H$2^\prime$}
		\langle A^\lambda(u) - A^\lambda(w),u-w\rangle \leq 0.
	\end{equation}
	Since this transformation of the equation is \textit{a posteriori} valid for the solutions we construct, we in general assume \eqref{eq:CoercivePrime}, \eqref{eq:MonotonePrime} instead of \eqref{eq:CoerciveNLPDE}, \eqref{eq:MonotoneOperator}.
\end{remark}
\begin{assumption}[Linear Growth]\label{def:LinearGrowth}
	There exists a $c>0$ such that for any $u\in V$,
	\begin{equation}\label{eq:LinearGrowth}\tag{H3}
		\|A (u)\|_{V^\ast} \leq c(1+\|u\|_V).
	\end{equation}
\end{assumption}
\begin{remark}
	If we replace the term $\|u\|^2_V$ with $\|u\|^p_{V}$ for $p>2$, in $\eqref{eq:CoerciveNLPDE}$, i.e we assume instead,
	\begin{equation}\label{eq:CoerciveM}\tag{H$1_{p}$}
		2\langle A(u),u\rangle + \alpha \|u\|^p_{V} \leq \lambda |u|^2_H+\nu,
	\end{equation}
	for some $\alpha,\,\lambda>0$, then we may also replace $\eqref{eq:LinearGrowth}$ with the $p$-growth assumption,
	\begin{equation}\label{eq:MGrowth}\tag{H$3_p$}
		\|A(u)\|_{V^\ast} \leq c(1+ \|u\|^{p-1}_V), 
	\end{equation} 
	for some $c>0$ and all $u \in V$. In this case the spaces $L^2([0,T];V)$ below should be replaced by $L^p([0,T];V)$.
\end{remark}
\begin{assumption}[Weak Continuity I]\label{ass:WeakContinuityPDE}
	For any $u,v,w \in V$, the mapping,
	\begin{equation}\label{eq:WeakContinuityI}\tag{H4}
		\mbR \ni \theta \mapsto \langle A(u+\theta w),v\rangle \in \mbR,
	\end{equation}
	is continuous.
\end{assumption}
\begin{remark}
	Note that Definition \eqref{eq:WeakContinuityI} is not the usual definition of weak continuity, i.e. when $u_n\rightharpoonup u$ in $V$, one has $A(u_n)\rightharpoonup A(u)\in V^\ast$ . Instead, we ask that the image of a strongly convergent family $u+\theta w\rightarrow u$ be weakly convergent, under $A$. Note that this implies the usual definition of weak continuity. We could rephrase Definition \ref{ass:WeakContinuityPDE} as requiring that the mapping $A:V_{\text{strong}}\rightarrow V^\ast_{\text{weak}}$ be continuous and refer to Assumption \ref{ass:WeakContinuityPDE} as strong-weak continuity.
\end{remark}
We now consider the non-linear PDE,
\begin{equation}\label{eq:NLPDE}
	\begin{cases}
		\partial_t u-A(u)= B_t,&\\
		u\tzero =u_0,
	\end{cases}
\end{equation}
and under the above conditions give the following well-posedness result.
\begin{theorem}\label{th:MonotonePDE}
	Let $T>0$, $u_0\in H$, $B\in L^2([0,T];V^\ast)$ and $A:V\rightarrow V^\ast$ satisfy \eqref{eq:CoercivePrime}, \eqref{eq:MonotonePrime}, \eqref{eq:LinearGrowth} and \eqref{eq:WeakContinuityI}. Then there exists a unique, weak solution $u \in L^2([0,T];V) \cap C([0,T];H)$ to \eqref{eq:AbstractLinearPDE}. 
\end{theorem}
\begin{proof}
	We proceed in a similar manner as the proof of Theorem \ref{th:LinearPDE}, establishing uniqueness first followed by existence.
	\paragraph{Uniqueness:} Now that $A$ is not necessarily linear, the assumption of coercivity is not enough. However, we can make use of monotonicity. Letting $u,w\in L^2([0,T];V)\cap C([0,T];H)$ be two solutions to \eqref{eq:AbstractLinearPDE}, for any $v \in V$, we have that 
	\begin{equation*}
		\left\langle \frac{\dd}{\dd t} (u_t-w_t),v\right\rangle = \langle A(u_t)-A(w_t),v\rangle,
	\end{equation*}
	and therefore by applying Lemma \ref{lem:LionsMagenes} and \eqref{eq:MonotonePrime}, we have
	\begin{equation*}
		|u_t-w_t|^2_H = 2 \int_0^t \langle A(u_s)-A(w_s),u_s-w_s\rangle \dd s \leq 0.
	\end{equation*}
	Thus $u_t=w_t$ in $H$ for all $t\in [0,T]$. Since $H\subset V^\ast$, this implies $u_t=w_t\in V^\ast$ for all $t\in [0,T]$ as well and so $\frac{d}{dt} u \equiv \frac{d}{dt}w \in V^\ast$. Then, let $\mcX \subset V \subset V^\ast$ be a linear, dense subset of both spaces. By the fundamental theorem of calculus, which holds also for weak derivatives in time, and using that $u,w\in L^2([0,T];V)$, for any $v\in \mcX$ and $t\in [0,T]$, we have
	\begin{equation*}
		\langle u_t-w_t,v\rangle = \int_0^t \left \langle \frac{d}{dt} (u_s-w_s),v\right\rangle \dd s =0.
	\end{equation*}
	So equality in $L^2([0,T];V)$ follows by density.
	\paragraph{Existence:} The first steps of the existence proof of Theorem \ref{th:LinearPDE} carry over in much the same way. We first obtain a sequence $(u_n)_{n\geq 1}\subset L^2([0,T];V)\cap C([0,T];H)$, which for any $n\geq 1$ solves the system of ODEs,
	\begin{equation}\label{eq:GalerkinNLODE}
		\begin{aligned}
			\frac{\dd }{\dd t}\langle u_{n;t},e_k\rangle& = \langle A(u_{n;t}), e_k\rangle + \langle B_t,e_k\rangle,\\
			\langle u_{n;0},e_k\rangle &= \langle u_0,e_k\rangle,
		\end{aligned}\quad \text{for }k=1,\ldots,n.
	\end{equation}
	Establishing well-posedness of the ODE system in this case is a little more involved but mostly technical. The coefficients are continuous and so existence can be shown for example by a Peano type argument and uniqueness follows using the monotonicity assumption, by a similar argument as applied to the PDE above. As in Section \ref{sec:LinearPDE}, we define,
	\begin{equation*}
		[0,T]\ni t\mapsto u_{n;t} := \sum_{k=1}^n \langle u_{n;t},e_k\rangle e_k,
	\end{equation*}
	which is a solution to \eqref{eq:NLPDE} with initial data $u_{n;0}$.\\ \par 
	In order to establish the a priori bound, we now apply \eqref{eq:CoercivePrime} to the identity,
	\begin{equation*}
		|u_{n;t}|_H^2 = \sum_{k=1}^n\langle u_0,e_k\rangle^2 + 2 \int_0^t \langle A(u_{n;s})+B_s,u_{n;s}\rangle \dd s,
	\end{equation*}
	in order to obtain, for some $\alpha>0$,
	\begin{equation*}
		|u_{n;t}|_H^2 + \alpha \int_0^t \|u_{n;s}\|_V \dd s 
		\, \leq \,|u_{n;0}|_H^2 + \int_0^T \|B_s\|^2_{V^\ast}\dd s +  \int_0^t |u_{n;s}|^2 \dd s.
	\end{equation*}
	Applying Gr\"onwall, and taking suprema over $t\in [0,T]$ and $n\geq 1$ on both sides, as before, gives the bound
	\begin{equation}\label{eq:NLinPDEEnergy}
		\begin{aligned}
			\sup_{n\geq 1} \left(\sup_{t \in [0,T]} |u_{n;t}|^2_H + \alpha \int_0^T \|u_{n;s}\|^2_{V} \dd s \right)\leq \left(|u_0|^2_H+\int_0^T \|B_s\|^2_{V^\ast}\dd s\right)e^{T}<\infty.
		\end{aligned}
	\end{equation}
	It follows from the a priori bound, \eqref{eq:NLinPDEEnergy}, and the Banach--Alaoglu theorem, Thm. \ref{th:BanachAlaoglu}, that there exists a weakly convergent subsequence, $u_n\rightharpoonup u \in V$, which we do not relabel. From \eqref{eq:LinearGrowth} we also have that $(A(u_n))_{n\geq 1}$ is a bounded sequence in $V^\ast$ and so in addition up to extracting a further subsequence, it follows, by the same reasoning, that there exists some $\zeta \in L^2([0,T];V^\ast)$ such that,
	\begin{align*}
		&u_n \rightharpoonup u \quad \text{in}\quad L^2([0,T];V),\\
		&u_n \overset{*}{\rightharpoonup} u \quad \text{in}\quad L^\infty([0,T];H),\\
		&A(u_n) \rightharpoonup \zeta \quad \text{in}\quad L^2([0,T];V^\ast).
	\end{align*} 
	So, letting $n\rightarrow \infty$ in \eqref{eq:GalerkinNLODE}, for any $v \in V$ and $t\in [0,T]$, we have that
	\begin{equation}\label{eq:NLWeakLimitPDE}
		\langle u_t,v\rangle = \langle u_0,v\rangle + \int_0^t \langle \zeta_s ,v\rangle \dd s + \int_0^t \langle B_s,v\rangle\dd s.
	\end{equation}
	Thus it suffices to conclude that $\zeta = A(u)$. We proceed in several steps. First we claim the inequality,
	\begin{equation}\label{eq:PDELimInf}
		\int_0^T \langle \zeta_t ,u_t\rangle \dd t  \leq \liminf_{n\rightarrow \infty}  \int_0^T \langle A(u_{n;t}),u_{n;t}\rangle \dd t.
	\end{equation}
	To see this, first note that from \eqref{eq:GalerkinNLODE} it follows that
	\begin{equation*}
		|u_{n;T}|^2_H -|u_{n;0}|_H^2 = 2\int_0^T \langle A(u_{n;t}),u_{n;t}\rangle \dd t + \int_0^T \langle B_t, u_{n;t}\rangle \dd t,
	\end{equation*}
	and similarly from \eqref{eq:NLWeakLimitPDE},
	\begin{equation*}
		|u_T|^2_H-|u_0|^2_H = \int_0^T \langle \zeta_t,u_t\rangle \dd t + \int_0^T \langle B_t,u_t\rangle \dd t.
	\end{equation*}
	Hence, it also follows that $u_{n;T} \rightharpoonup u_T$ in $H$ and so by convexity of the map $H\ni \rho \mapsto |\rho|^2_H$  we have the inequality,
	\begin{equation*}
		|u_T|^2_H - |u_0|_H^2 \leq \liminf_{n\rightarrow \infty} \left(|u_{n;T}|^2 - |u_{n;0}|^2_H\right).
	\end{equation*}
	So \eqref{eq:PDELimInf} follows from the fact that $\liminf_{n\rightarrow \infty} \langle B_t,u_{n;t}\rangle =0$ for all $t\in [0,T]$.\\ \par 
	With \eqref{eq:PDELimInf} in hand, we observe that from \eqref{eq:MonotonePrime}, for all $v \in L^2([0,T];V)$ and $n\geq 1$,
	\begin{equation*}\label{eq:PDEUnMononone}
		\int_0^T \langle A(u_{n;t})-A(v_t),u_{n;t}-v_t\rangle \dd t \leq 0
	\end{equation*}
	By weak convergence and applying \eqref{eq:PDELimInf} in the first inequality below, we have
	\begin{equation}\label{eq:PDEZetaMonotone}
		\hspace{-2em}	\begin{aligned}
			\int_0^T \langle \zeta_t -A(v_t), u_t-v_t\rangle \dd t 	&= \int_0^T  \langle \zeta_t,u_t\rangle  \dd t\\
			&\quad - \lim_{n\rightarrow \infty} \int_0^T \langle A(u_{n;t}),v_t \rangle +\langle A(v_t),u_{n;t}\rangle\dd t\\
			& \quad + \int_0^T \langle A(v_t),v_t\rangle  \dd t\\
			& \leq \liminf_{n\rightarrow \infty} \int_0^T \langle A(u_{n;t}),u_{n;t}\rangle\dd t\\
			&\quad -\lim_{n\rightarrow \infty}\int_0^T \langle A(u_{n;t}),v_t \rangle +\langle A(v_t),u_{n;t}\rangle \dd t \\
			&\quad + \int_0^T \langle A(v_t),v_t\rangle  \dd t \\
			&= \liminf_{n\rightarrow \infty} \int_0^T  \langle A(u_{n;t})-A(v_t),u_{n;t}-v_t\rangle \dd t \\
			&\leq 0.
		\end{aligned}
	\end{equation}
	So now choosing $v = u-\theta w$ with $\theta >0$, after dividing both sides by $\theta$, we have that
	\begin{align*}
		\int_0^T \langle \zeta_t -A(u_t-\theta w_t),  w_t\rangle \dd t \leq 0.
	\end{align*}
	So from \eqref{eq:WeakContinuityI} and taking the limit $\theta \rightarrow 0$, we obtain that
	\begin{align*}
		\int_0^T \langle \zeta_t -A(u_t),  w_t\rangle \dd t \leq 0, \quad \text{ for all }w \in L^2([0,T];V). 
	\end{align*}
	Hence it must be that $\zeta = A(u)$ in $L^2([0,T];V^\ast)$ and this concludes the proof.
\end{proof}
\begin{example} (See \cite[Ex.~4.1.9]{liu_rockner_15_spde_introduction})\label{ex:p_laplace_reac_diffuse}
	Let us fix a smooth open bounded domain $\Gamma \subset \mbR^d$, $m \geq 2,\, \kappa\geq 0$ and consider the $p$-Laplace equation with reaction term and Dirichlet boundary data
	\begin{equation*}\label{eq:p_laplace_reac_diffuse}
		\begin{cases}
			\partial_t u_t - \nabla \cdot (\nabla u_t |\nabla u_t|^{m-2}) + \kappa  |u_t|^{m-2}u_t= 0,& \text{in }\Gamma,\\
			u_t|_{\partial \Gamma}=0, &\text{for all }t\geq 0,\\
			u\tzero	=u_0, &  \text{in }\Gamma.
		\end{cases}	
	\end{equation*}
	We pick our functional setting by choosing
	\begin{equation*}
		W^{1,m}_0(\Gamma) = V,
	\end{equation*}
	which is the space of maps $f:\Gamma \to \mbR$ such that 
	\begin{equation*}
	\|f\|_{W^{1,m}(\Gamma)} \coloneqq 	\|f\|_{L^m(\Gamma)} + \|\nabla f\|_{L^m(\Gamma)} <+\infty\quad \text{and}\quad f|_{\partial \Gamma} = 0.
	\end{equation*} 
	Then we set $H= L^2(\Gamma)$ and $V^* = (W^{1,m}_0(\Gamma))^*$ and define the map
	\begin{equation*}
		\begin{aligned}
			A: 	W^{1,m}_0(\Gamma) &\to (W^{1,m}_0(\Gamma))^*\\
			u &\mapsto \nabla \cdot (\nabla u |\nabla u|^{m-2}) - \kappa |u|^{m-2}u \eqqcolon \Delta_m u - \kappa |u|^{m-2}u.
		\end{aligned}
	\end{equation*}
	Checking \eqref{eq:MGrowth} will verify that this is indeed a well-defined operator.  To check \eqref{eq:CoerciveM} we write
	\begin{align*}
		\langle A(u),u\rangle = &\langle \nabla \cdot (\nabla u |\nabla u |^{m-2}) , u \rangle  -\kappa  \langle |u|^{m-2}u,u\rangle \\
		=& - \langle \nabla u |\nabla u |^{m-2} ,\nabla  u \rangle - \kappa \langle |u|^{m-2},|u|^2\rangle \\
		=&- \|\nabla u \|^p_{L^p} - \kappa \|u\|^p_{L^p}.
	\end{align*}
	Now, since we are on a bounded domain we have the Poincar\`e inequality (see discussion after Example~\ref{ex:stoch_rd_field}), that is there exists a constant $C\coloneqq C(m,d,|\Gamma|) >0$ such that
	\begin{equation*}
		\|u\|_{L^p(\Gamma)} \leq C \, \|\nabla u\|_{L^p(\Gamma)}, 
	\end{equation*}
so that 
\begin{equation*}
	\|u\|_{W^{1,p}} \leq (1+C) \|\nabla u\|_{L^p(\Gamma)},
\end{equation*}
and hence there exists a constant $C\coloneqq (m,d,\kappa,|\Gamma|)>0$ such that
\begin{equation*}
		\langle A(u),u\rangle  \lesssim  - C \|u\|^p_{W^{1,m}},
\end{equation*}
so that \eqref{eq:CoerciveM} is satisfied with $p=m$, $\alpha = C$ and $\lambda = \nu = 0$.\\

To check \eqref{eq:MonotoneOperator} we begin by taking $u,\, w\in W^{1,m}_0(\Gamma)$ and integrate by parts to write
\begin{align*}
	\langle A(u) - A(w), u-w\rangle = & -  \langle |\nabla u|^{m-2}\nabla u - |\nabla w|^{m-2}\nabla w , \nabla u - \nabla w\rangle - \langle |u|^{m-2}u- |w|^{m-2}w,u-w\rangle \\
	=&  - \RN{1} + \RN{2}.
\end{align*}
To treat $-\RN{1}$ we collect terms and the bound $a\cdot b \leq |a||b|$ to obtain
\begin{align*}
	\RN{1} =  &\|\nabla u\|^m_{L^m} + \|\nabla w\|^m_{L^m}  -\langle |\nabla u|^{m-2},\nabla u\cdot \nabla w\rangle - \langle |\nabla w|^{m-2}, \nabla u \cdot \nabla w\rangle\\
	\geq & \, \|\nabla u\|^m_{L^m} + \|\nabla w\|^m_{L^m} - \langle |\nabla u|^{m-1},|\nabla w|\rangle - \langle |\nabla w|^{m-1}, |\nabla w|\rangle\\
	=&\, \langle |\nabla u|^{m-1} - |\nabla w|^{m-1}, |\nabla u|- |\nabla w|\rangle \\
	\geq &\, 0.
\end{align*}
Hence $-\RN{1} \leq 0$. To treat $\RN{2}$ we recall that one has $|a|^{m}-|b|^{m} \geq |a-b|^{m-2}(a-b)^2=|a-b|^{m}$ for all $a,\,b\in \mbR$, so that
\begin{equation*}
	\RN{2} \leq -\kappa \|u-w\|^{m}_{L^{m}} <0.
\end{equation*}
Hence $A$ satisfies \eqref{eq:MonotoneOperator} with $\lambda =0$.\\

For the growth bound \eqref{eq:MGrowth}, we write $w\in W^{1,m}_0(\Gamma)$ 
\begin{equation*}
	\|A(u)\|_{(W^{1,m}_0(\Gamma))^*} = \sup_{w \in W^{1,m}_0(\Gamma)} \frac{\langle A(u),w\rangle}{\|w\|_{W^{1,m}(\Gamma)}} 
\end{equation*}
For any given $w\in W^{1,p}_0(\Gamma)$,
\begin{equation*}
\langle A(u),w\rangle = 	\langle \nabla \cdot (\nabla u |\nabla u|^{m-2}) - \kappa |u|^{m-2}u,w \rangle = -\langle  \nabla u |\nabla u|^{m-2}, \nabla w\rangle - \kappa \langle |u|^{m-2}u,w\rangle.
\end{equation*}
Note that since we take the supremum over all $w\in W^{1,m}_0(\Gamma)$ we have no information on the sign of the brackets above, so that the best we can do is to bound them by H\"older's inequality to find
\begin{equation*}
	|\langle  \nabla u |\nabla u|^{m-2}, \nabla w\rangle| \leq \|\nabla u\|^{m-1}_{L^m} \|\nabla w\|_{L^m},
\end{equation*}
and
\begin{equation*}
	|\langle |u|^{m-2}u,w\rangle| \leq \|u\|^{m-1}_{L^m}\|w\|_{L^m}.
\end{equation*}
So that putting both together
\begin{equation*}
		\|A(u)\|_{(W^{1,m}_0(\Gamma))^*}  \lesssim_{m,\kappa} \|u\|_{W^{1,p}(\Gamma)}. 
\end{equation*}
To check \eqref{eq:WeakContinuityI} we let $u,\,v,\, w \in W^{1,p}_0(\Gamma)$ and show that for $\theta \in (0,1]$,
\begin{equation*}
	\lim_{\theta \to 0} \langle A(u+\theta w),v\rangle  = \langle A(u),v\rangle.
\end{equation*}
To this end we write
\begin{align*}
\langle A(u+\theta w),v\rangle   = & - \langle \nabla (u+\theta w)|\nabla (u+\theta w)|^{m-2},\nabla v\rangle  - \kappa \langle |u+\theta w|^{m-2}(u+\theta w),v\rangle \\
= &- \langle |\nabla (u+\theta w)|^{m-2},\nabla (u+\theta w) \cdot \nabla v\rangle - \kappa \langle |u+\theta w|^{m-2}(u+\theta w),v\rangle.
\end{align*}
It is clear that Lebesgue almost everywhere we have the limits,
\begin{equation*}
	\lim_{\theta \to 0} |\nabla (u+\theta w)|^{m-2} = |\nabla u|^{m-2},\quad  \lim_{\theta \to 0} \nabla (u+\theta w) = \nabla u\quad \text{and}\quad \lim_{\theta \to 0} |u+\theta w|^{m-2}(u+\theta w) = |u|^{m-2}u.
\end{equation*}
Hence, we only need to find a dominating function so that we can apply Lebesgue's dominated convergence theorem  (Theorem~\ref{th:DCTBochner} with $V = \mbR^d$). To this end, since we restricted to $\theta \in (0,1]$, for all $x\in \Gamma$ we have the estimate
\begin{equation*}
	\sup_{\theta \in (0,1]}\left|\nabla (u(x)+\theta w(x))|^{m-2}\nabla (u(x)+\theta w(x)) \cdot \nabla v(x)\right| \lesssim_m  \left(|\nabla u(x)|^{p-1}+|\nabla w(x)|^{m-1}\right) |\nabla v(x)|,
\end{equation*}
while
\begin{equation*}
	\sup_{\theta \in (0,1]} \left| |u(x)+\theta w(x)|^{m-2}(u(x)+\theta w(x))v(x)\right |\, \lesssim_m\, \left(|u(x)|^{m-1}+ |w(x)|^{m-1} \right) |v(x)|.
\end{equation*}
We can control the right hand side of both inequalities above in $L^1(\Gamma)$ by H\"older's inequality and so the proof is complete. 

\qed
\end{example}
\begin{remark}\label{rem:l_p_weak_strong_continuous}
	The checking of \eqref{eq:WeakContinuityI} in Example~\ref{ex:p_laplace_reac_diffuse} demonstrates a general principle of weak-strong uniqueness of polynomials in $W^{k,p}(\Gamma)$ spaces.
\end{remark}
%
\begin{remark}\label{rem:massive_heat_equation_and_defocussing_rd}
	If one sets $m=2$ in \eqref{eq:p_laplace_reac_diffuse} then the resulting equation is linear and is known as the massive heat equation. It can be solved directly by Fourier analysis or the arguments of Section \ref{sec:LinearPDE}. Alternatively, for $m>2$ and $\theta <0$ the equation is known as \emph{defocussing} and solutions may only exist locally in time, depending on the size of the initial data in $W^{1,m}_0(\Gamma)$.
\end{remark}
\section{A Compactness Method}
In the absence of the above structural assumption, \eqref{eq:MonotoneOperator}, one can instead make progress using more functional analytic techniques. We dedicate less space to this approach since we will not present its stochastic counterpart, instead referring to \cite[Sec. 2.3.3]{pardoux_21_spde_introduction} for further details.\\ \par 
Note that in the proceeding arguments it was not challenging to establish uniform boundedness of $\left(\frac{\dd}{\dd t}u_{n}\right)_{n \geq 1} \subset L^2([0,T];V^\ast)$ and $(u_{n})_{n\geq 1} \subset L^2([0,T];V)$. The difficulty was in establishing convergence of the non-linear term $A(u_n)$ to $A(u)$. In the case that the injection $V\hookrightarrow H$ is compact then we have access to the following Lemma.
\begin{lemma}[Aubin--Lions]\label{lem:AubinLions}
	Let $V$ and $H$ be such that $V\hookrightarrow H$ is compact. Then, given $(u_n)_{n\geq 1}$ uniformly bounded in $L^2([0,T];V)$ and such that $(\frac{d}{dt} u_n)_{n\geq 1}$ is uniformly bounded in $L^2([0,T];V^\ast)$, there exists a non-relabelled subsequence, $(u_n)_{n\geq 1}$, which converges strongly in $L^2([0,T];H)$.
\end{lemma}
We define the following set of assumptions, which will replace \eqref{eq:MonotoneOperator} and \eqref{eq:WeakContinuityI}.
\begin{assumption}[Compactness]\label{def:Compactness}
	The injection $V \hookrightarrow H$ is compact, i.e.
	\begin{equation}\label{eq:Compactness}\tag{H5}
		\text{Any bounded set $\mcK \subset V$ is a compact in $H$.}
	\end{equation}
\end{assumption} 
\begin{assumption}[Weak Continuity II]\label{def:WeakContII} $A$ is weakly continuous in the sense that the map 
	\begin{equation}\label{eq:WeakContinuityII}\tag{H6}
		\begin{aligned}
			A:V_{\text{weak}}\cap H_{\text{strong}} &\longrightarrow V^\ast_{\text{weak}}\\
			u &\mapsto A(u),
		\end{aligned}
	\end{equation}
	is continuous.
\end{assumption}
Under these assumptions we have the following existence result.
\begin{theorem}\label{th:CompactPDE}
	Let $T>0$, $u_0 \in H$, $B \in L^2([0,T];V^\ast)$ and $A: V \rightarrow V^\ast$ be an operator, not necessarily linear, satisfying \eqref{eq:CoercivePrime}, \eqref{eq:LinearGrowth}, \eqref{eq:Compactness} and \eqref{eq:WeakContinuityII}. Then there exists a weak solution $u \in L^2([0,T];H)$ to \eqref{eq:AbstractLinearPDE}.
\end{theorem}
\begin{proof}
	Defining, solving and obtaining a uniform bound on the discrete system $(u_n)_{n\geq 1}$ follows exactly as in Theorems \ref{th:LinearPDE} and \ref{th:MonotonePDE}. Thus, as in Theorem \ref{th:MonotonePDE} we have, up to relabelling, a sequence, $(u_n)_{n\geq 1}$ and a $\zeta \in L^2([0,T];V^\ast)$ such that,
	\begin{align*}
		u_n &\rightharpoonup u \quad \text{in}\quad L^2([0,T];V),\\
		u_n &\overset{\ast}{\rightharpoonup} \quad \text{in}\quad  L^\infty([0,T];H),\\
		A(u_n) &\rightharpoonup \zeta \quad \text{in}\quad L^2([0,T];V^\ast).
	\end{align*} 
	Recall that the linear growth assumption \eqref{eq:LinearGrowth} allows us to pass the boundedness of $(u_n)_{n\geq 1}\subset L^2([0,T];V)$ to boundedness of $(A(u_n))_{n\geq 1}\subset L^2([0,T];V^\ast)$, from which we extracted the convergent subsequence above and also gives a uniform bound on $(\frac{d}{dt} u_n)_{n\geq 1}\subset L^2([0,T];V^\ast)$. Therefore, since we assumed \eqref{eq:Compactness}, we may apply Lemma \ref{lem:AubinLions} to extract a further, not relabelled subsequence, $u_n \rightarrow u \in L^2([0,T];H)$. So now $(u_{n})_{n\geq 1}$ is convergent in $V_{\text{weak}}\cap H_{\text{strong}}$. Thus, by \eqref{eq:WeakContinuityII} it follows that $\zeta = A(u) \in L^2([0,T];V^\ast)$ and so we conclude the existence of a weak solution \eqref{eq:AbstractLinearPDE}.
\end{proof}
\begin{remark}
	Note that in this case one needs some additional ingredient to establish uniqueness of the weak solution. This could be by imposing additional conditions on $A$ or by ad hoc methods given a particular equation.
\end{remark}
\begin{example}
	Consider the following non-linear parabolic PDE which is a toy model for the Navier--Stokes equations,
	\begin{equation*}
		\begin{cases}
			\partial_t u - \Delta u = u\cdot \nabla u,& \text{ on }\mbR_+\times \mbT^d,\\
			u\tzero = u_0 &\text{ on }\mbT^d.
		\end{cases}
	\end{equation*}
	Here $u:\mbR_+\times \mbT^d\rightarrow \mbR^d$ is an evolving vector field which in the true Navier--Stokes problem represents the velocity of a fluid. We may identify the non-linear operator $A(u):= \Delta u + u\cdot \nabla u $, for which we may check \eqref{eq:CoerciveM}, \eqref{eq:MGrowth}, \eqref{eq:Compactness} and \eqref{eq:WeakContinuityII} with the triple $(H^1(\mbT^d),L^2(\mbT^d),H^{-1}(\mbT^d))$ and $m=2$. To see \eqref{eq:WeakContinuityII} observe that if $u_n\rightarrow u $ in $V_{\text{weak}}\cap H_{\text{strong}}$, then the non-linear term $(u_n\cdot \nabla u_n)_{n\geq 1} \subset V^\ast$ is a product of a strongly convergent sequence with a weakly convergent sequence, which is always weakly convergent.\\ \par 
	To cover the true Navier--Stokes problem one would need to include the divergence free condition and work in the space of weakly divergence free vector fields. This programme was originally carried out by Leray, see \cite{leray_34}, and now constitutes an important result in the wider programme to establish existence and uniqueness of strong solutions to the Navier--Stokes equations. A review of Leray's work in English can be found in \cite{farwig_17} and the preprint \cite{ozanksi_pooley_17}. A fuller survey of results concerning the Navier--Stokes equations can be found in \cite{seregin_14} (available as \cite{seregin_14_Free}).
\end{example}
In the next chapter we turn to the application of these methods to stochastic PDEs. We focus only on extending the monotone operator approach of Theorem \ref{th:MonotonePDE}, however a detailed treatment of the compactness approach can be found in \cite[Sec. 2.3.3]{pardoux_21_spde_introduction} and the lecture notes by M. Metivier and M. Viot, \cite{metivier_viot_88}.

	\chapter{Variational Approach to Monotone SPDE}\label{ch:SPDE}\hfill \\

We fix a $T>0$ a filtered probability space $(\Omega,\mcF,(\mcF_t)_{t\in [0,T]};\mbP)$ and a Gelfand triple $(V,H,V^\ast)$ as in Chapter \ref{ch:PDE}. Our aim is to apply the approach of Chapter \ref{ch:PDE} to stochastic PDE of the form,
\begin{equation}\label{eq:AbstractSPDE}
	\begin{cases}
		\dd u_t = A(u_t)\dd t + B(u_t) \dd W_t ,&\\
		u\tzero = u_0,
	\end{cases}
\end{equation}
where $A$ is a suitable operator, $W$ denotes an infinite dimensional, Wiener process and $u\mapsto B(u)$ is a map taking values in the space of processes that are stochastically integrable against $W$. These notions will be made more precise below. As in Chapter \ref{ch:PDE} we leave the specification of a spatial domain implicit in \eqref{eq:AbstractSPDE}. For simplicity we consider $A,\,B$ independent of $(t,\omega)\in \mbR_+\times \Omega$. Extending to include these cases, under natural assumptions, is relatively straightforward while the extra degrees of freedom also allow for more bespoke applications and methods.  For a treatment in this more general setting see \cite{prevot_rockner_07}.\\ \par
We briefly recall some necessary material from Chapter \ref{ch:Background}, in particular Subsection \ref{subsec:GenStochInt}. Given $U,H$, Hilbert spaces and a generalized Wiener process $(W_t)_{t\in [0,T]}$ with covariance $Q$, defined on $U$, we say that a predictable process $[0,T]\times \Omega \ni (t,\omega)\mapsto \varphi_t(\omega) \in L^0_2:= L_2(Q^{1/2}U;H)$ is stochastically integrable if,
\begin{equation*}
	\mbP\left(\int_0^T \|\varphi_s\|^2_{L_2^0} \dd s <\infty \right) =1.
\end{equation*} 
In this case we say $\varphi \in \msH_T(W)$ and recall from Theorem \ref{th:GenStochInt} that its stochastic integral against $W$, $(\varphi \bigcdot W_t)_{t\in [0,T]}$, defines a square integrable, local martingale. Furthermore, we can give a concrete representation of the integral. Let $U_1$ be a Hilbert space containing $U_0$ such that the embedding $\iota :U_0\rightarrow U_1$ is Hilbert--Schmidt. Then, for any $t\in [0,T]$ we have the following identity, which is independent of the choices $\iota,\,U_1$, satisfying the necessary requirements,
\begin{equation}\label{eq:GenStochIntBis}
	\varphi \bigcdot W_t = \int_0^t (\varphi_s \circ \iota^{-1})\dd W_s .
\end{equation}
In what follows we will assume $B:V\rightarrow \msH_T(W)$, in the sense that $B(u_{\,\cdot\,}) \in \msH_T(W)$ for any suitable process, $[0,T]\ni t \mapsto u_t \in V$. In order to discuss stochastic processes of the form described by \eqref{eq:AbstractSPDE} we introduce the notion of a semi-martingale, that is a stochastic process whose evolution is described by a finite variation process, $t\mapsto A(u_t)$, in the case of \eqref{eq:AbstractSPDE} and a local martingale term, $t\mapsto B(u)\bigcdot W_t $.
\section{Semi-Martingales and It\^o's Formula}\label{sec:QVandIto}
In Chapter \ref{ch:Background} we introduced the space of square integrable, continuous, Banach space valued, martingales, $\mcM^2_T(E)$. In this chapter we specify to the case of Hilbert space valued martingale and we introduce some more properties of these objects, defining their quadratic variation processes and studying the larger space of semi-martingales through It\^o's formula in both finite and infinite dimensions.\\ \par
We introduce some notation. Given a $T>0$ and $n\geq 1$, we write $\pi_n(T)$ for any set of the form $\pi_n(T)= \{ 0=t^n_0<t^n_1\cdots t^n_n =T\}$. We define $|\pi_n(T)|:= \sup_{i = 1,\ldots,n} |t^n_{i}-t^n_{i-1}|$. We say that a family of partitions $(\pi_n(T))_{n\geq 1}$ is decreasing if $\lim_{n\rightarrow \infty }|\pi_n(T)|=0$. Using this definition we recall what it means for a process to be of finite variation.
\begin{definition}[Finite Variation Process]\label{def:FiniteVariation}
	Let $T>0$ and $\psi:[0,T]\rightarrow H$ be a measurable map. We say that $\psi$ is of finite variation if,
	\begin{equation}\label{eq:FiniteVariation}
		\sup_{\substack{n\in \mbN\\ \pi_n(T)}} \left(\sum_{[s,t]\in \pi_n(T)} \|\psi_{t} - \psi_{s}\|_{H} \right)<\infty,
	\end{equation}
	where the supremum ranges over all partitions for a given $n\in \mbN$. We say that a stochastic process $[0,T]\times \Omega\ni (t,\omega)\mapsto \psi_t(\omega)\in H $ is of finite variation if \eqref{eq:FiniteVariation} holds $\mbP$-a.s.
\end{definition}
Given a Hilbert space $H$, we recall the definition of $L_1(H)\subset \mcL(H;H)$, as the set of trace class operators from $H$ to itself, which becomes a Banach space when equipped with the trace norm, \eqref{eq:TraceNorm}. We say that an $L_1(H)$ valued process, $[0,T]\ni t \mapsto V_t \in L_1(H)$ is non-decreasing if for every $s< t\in [0,T]$, $V_t-V_s$ is a non-negative operator. If $M\in H$, then we denote by, $M\otimes M$, the non-negative element of $L_1(H)$ such that for all $h,\,g\in H$,
\begin{equation}
	\langle (M\otimes M) h,g\rangle = \langle M,h\rangle\langle M,g\rangle.
\end{equation}
\begin{definition}[Quadratic Variation]\label{def:QV}
	Let $M\in \mcM^2_T(H)$. An $L_1(H)$ valued, non-decreasing process, $(V_t)_{t\in [0,T]}$, such that $V_0=0$, is said to be the quadratic variation of $M$, if the $L_1(H)$ valued process,
	\begin{equation}\label{eq:QVProperty}
		[0,T]\ni t\mapsto	M_t\otimes M_t - V_t\in L_1(H),
	\end{equation}
	is an $(\mcF_t)_{t\in [0,T]}$-martingale. Furthermore, given two martingales, $M,\,N\in \mcM^2_T(H)$, we say that an $L_1(H)$ valued process $(V_t)_{t\in [0,T]}$ is the quadratic co-variation of $M$ and $N$ if
	\begin{equation}\label{eq:CoQVProperty}
		[0,T]\ni t\mapsto M_t\otimes N_t- V_t,
	\end{equation}
	is an $(\mcF_t)_{t\in[0,T]}$ martingale.
\end{definition}
We demonstrate that to every $M\in \mcM^2_T(H)$ there exists a unique such process for which we write $[0,T]\ni t\mapsto [M]_t\in L^1(H)$. Uniqueness here means that for any other process, $(V_t)_{t\in [0,T]}$, satisfying \eqref{eq:QVProperty}, $[M]_t=V_t$ $\mbP$-a.s., for all $t\in [0,T]$. Given two martingales, $M,\,N\in \mcM^2_T(H)$ we write $[M,\,N]$ for their quadratic co-variation. If $M,\,N$ are local martingales then the same quantities are analogously defined.
\begin{proposition}\label{prop:QVExistUnique}
	Let $\mcM^2_{T;\text{loc}}(H)$. Then there exists exactly one, continuous, non-decreasing, zero at zero process $([M]_t)_{t\in [0,T]}\subset L_1(H)$ such that setting $V= [M]$ in \eqref{eq:QVProperty} gives a continuous local martingale. Furthermore, for any $t\in [0,T]$ and decreasing sequence of partitions $(\pi_n(t))_{n\geq 1}$, it holds that
	\begin{equation}\label{eq:QVDefine1}
		[M]_t = \lim_{n\rightarrow \infty } \sum_{i=1}^n (M_{t_i}-M_{t_{i-1}})\otimes (M_{t_i}-M_{t_{i-1}}),\quad t\in [0,T],
	\end{equation} 
	where the limit holds in $L_1(H)$ in probability. For $M,\,N\in \mcM^2_{T;\text{loc}}(H)$ there exists a unique, zero at zero, finite variation process, $([M,N]_{t})_{t\in [0,T]}$ such that setting $V=[M,N]$ in  \eqref{eq:CoQVProperty} gives a continuous, local martingale.
\end{proposition}
\begin{proof}
	To prove existence, consider $V,\,\tilde{V}$, two processes satisfying \eqref{eq:QVProperty} for a martingale $M\in \mcM^2_T(H)$. Therefore, for any $t\in [0,T]$ we have,
	\begin{equation*}
		\mbE[V_t-\tilde{V}_t] = \mbE[(M_t\otimes M_t-V_t) - (M_t \otimes M_t -\tilde{V}_t)] = 0.
	\end{equation*}
	We reduce the proof of existence to considering finite dimensional martingales. Let $(e_k)_{k\geq 1}$ be a basis of $H$. Then for any $k\geq 1$, $M_{k;t}:= \langle M_t,e_k\rangle$, defines a real valued martingale. It follows from \cite[Thm. 9.1.6 \& Def. 11.1.1]{cohen_elliott_15} that for any $k,\,l \geq 1$, there exists a unique, real valued process $([M_k,M_l]_t)_{t\in [0,T]}$ such that $t\mapsto M_{k;t} M_{l;t} - [M_k,M_l]_t$ is an $(\mcF_t)_{t\in [0,T]}$ martingale. Furthermore, we have the explicit formula,
	\begin{equation}\label{eq:FiniteDimQVDefine}
		[M_k,M_l]_t = \lim_{n\rightarrow  \infty} \sum_{i=1}^n (M_{k;t_i}-M_{k;t_{i-1}})(M_{l;t_i}-M_{l;t_{i-1}}), \quad t\in [0,T].
	\end{equation}
	Where the sum converges in probability along any decreasing sequence of partitions to a real valued process. By polarisation we have that
	\begin{equation*}
		[M_k,M_l] = \frac{1}{4} \left([M_k+M_l] -  [M_k-M_l]\right).
	\end{equation*}
We now show that we can define,
	\begin{equation}\label{eq:QVDefine2}
		[M]_t := \sum_{k,l \geq 1} [M_k,M_l]_t e_k\otimes e_l , \quad t\in [0,T],
	\end{equation}
	where $[M_k,M_l]_t\in \mbR$ for each $k,\,l\geq 1$ and $t\in [0,T]$, so that if the sum is finite, the right hand side defines an element of $L_1(H)$. We have,
	\begin{align*}
		\mbE \left[\,\left|\sum_{k,l \geq 1} [M_k,M_l]_t e_k\otimes e_l \right| \,\right] = \mbE\left[\sum_{k,\,l\geq 1} |[M_k,M_l]| \right] \leq \frac{1}{4}\mbE\left[\,\sum_{k,\,l\geq 1} [M_k+M_l]+[M_k-M_l]\right],
	\end{align*}
	where in the last inequality we used the polarisation identity and that $[M_k+M_l]$ and $[M_k-M_l]$ are both non-negative. By definition, $\mbE[[M_k+M_l]] =\mbE[(M_k+M_l)^2]$ and $\mbE[[M_k-M_l]]=\mbE[(M_k-M_l)^2]$. So by linearity of the expectation (properly employed after truncating the sums and then taking limits) and applying Young's product inequality we have
	\begin{equation*}
		\mbE\left[ \,\left|\sum_{k,l \geq 1} [M_k,M_l]_t e_k\otimes e_l \right|\, \right] \lesssim \mbE\left[\sum_{k\geq 1} M_k^2 \right]= \mbE\left[\sum_{k\geq 1} |\langle M,e_k\rangle|^2\right] <\infty.
	\end{equation*}
	Thus the right hand side of \eqref{eq:QVDefine1} is $\mbP$-a.s. a well defined element of $L_1(H)$. Almost sure continuity of the map $[0,T]\ni t\mapsto [M]_t \in L_1(H)$ is easily established using continuity of $[0,T]\ni t \mapsto M_t\in H$. Furthermore, combining \eqref{eq:FiniteDimQVDefine} and \eqref{eq:QVDefine2} shows that $[M]$ as defined satisfies \eqref{eq:QVProperty} and agrees with \eqref{eq:QVDefine1}.\\ \par
	In order to obtain the results for the quadratic co-variation $[M,N]$ between two elements $M,N \in \mcM^2_T(H)$ it suffices to apply the polarisation identity again, $[M,N] = \frac{1}{4}\left([M+N]-[M-N]\right)$ and the above argument to the quadratic variations $[M+N], \, [M-N]$. Uniqueness of this process holds by an adaptation of the proof of uniqueness for $[M]$. The fact that $[0,T]\ni t \mapsto [M,N]_t \in L_1(H)$ is $\mbP$-a.s. of finite variation follows from the same fact for finite dimensional quadratic co-variations.
\end{proof}
\begin{remark}
	Note that the quadratic variation is not the $2$-variation in the sense of Definition \ref{def:FiniteVariation}, which would be defined as the quantity,
	\begin{equation*}
		\sup_{\substack{n\in \mbN \\ \pi_n(T)}} \left( \sum_{[s,t]\in \pi_n(T)} \|\psi_t-\psi_s\|^2_{H} \right).
	\end{equation*}
	In contrast the quadratic variation is only the limit along any given decreasing sequence of partitions.
\end{remark}
The following proposition collects some useful facts regarding the quadratic variation that we give without proof.
\begin{proposition}\label{prop:QVProperties} The following statements all hold,
	\begin{enumerate}[label=\arabic*)]
		\item \label{it:ZeroQVZeroMart} If $M\in \mcM^2_T(H)$ is such that $\mbP$-a.s. $[M]\equiv 0$ on $[0,T]$, then $M\equiv M_0$ $\mbP$-a.s. on $[0,T]$.
		\item \label{it:IndependentCoQV} If $M,\,N \in \mcM^2_T(H)$ are independent martingales, in the sense that $M_t \perp N_t$ for all $t\in [0,T]$, then $[M,N]\equiv 0$ $\mbP$-a.s. on $[0,T]$.
		\item \label{it:FiniteVarQV} Any continuous, finite variation process $[0,T]\ni t\mapsto \psi_t \in H$, $\mbP$-a.s. has zero quadratic variation and the quadratic co-variation, $[\psi,M] \equiv 0$, $\mbP$-a.s. for any $M\in \mcM_T^2(H)$.
		\item \label{it:LocalQV} If $(M_t)_{t\in [0,T]}$ is a square integrable, local martingale, then for any $(\mcF_t)_{t\in [0,T]}$-stopping time, $\Omega\ni \omega\mapsto \tau(\omega)\in [0,T]$, such that $[0,T] \ni t\mapsto M^\tau_t := M_{\tau\wedge t}$, is a true martingale, $\mbP$-a.s. one has that $[M^\tau]_t = [M]_{\tau\wedge t}$ for all $t\in [0,T]$.
		\item \label{it:NormQV} If $M\in \mcM_T^2(H)$, then since $\|M\|_H\in \mcM_T^2(\mbR)$ there exists a unique, non-decreasing, continuous, zero at zero process, for which we write $[0,T]\ni t \mapsto [[M]]_t$ such that $\|M\|^2_H-[[M]]$ is again a square integrable martingale. Furthermore, $[[M]] = \Tr [M]$. 
	\end{enumerate}
\end{proposition}
Given an interval $[0,T]\subset \mbR_+$ and a $p\in (0,\infty)$, we say that an $H$ valued, martingale, $(M_t)_{t\in [0,T]}\subset H$, is $p$-integrable, if $\mbE\left[\sup_{t\in [0,T]}\|M_t\|^p_H \right]<\infty$. We state a version of the BDG inequality suited to the particular setting of Hilbert space valued martingales.
\begin{theorem}[Burkholder--Davis--Gundy]\label{th:BDG}
	Let $p\in [1,\infty)$ and $(M_t)_{t\in [0,T]}$ be a $p$-integrable, $H$ valued, local martingale. Then, there exist constants $c(p),C(p)>0$ such that for any $(\mcF_t)_{t\geq 0}$-stopping time $\tau$,
	\begin{equation*}
		c \mbE\left[(\Tr [M^\tau])^{\frac{p}{2}}\right] \leq \mbE\left [\sup_{t \in [0,\tau]} \|M_t\|_H^p\right]   \leq C \mbE \left[(\Tr [M^\tau])^{\frac{p}{2}}\right]
	\end{equation*}
	If $(M_t)_{t\in [0,T]}$ is a continuous, local martingale then the result holds for $p \in (0,\infty)$.
\end{theorem}
\begin{proof}
	For a proof in the case of general finite dimensional martingales, see \cite[Thm. 11.5.5]{cohen_elliott_15}. In our case the particular form of the quantities inside the left and right expectations follows from Item  \ref{it:NormQV} of Proposition \ref{prop:QVProperties}.
\end{proof}
By way of an example we exhibit the quadratic variation of stochastic integrals, which from now on will be our central examples of continuous, local martingales. We state this result only for stochastic integrals against generalized Wiener processes. Note that stochastic integrals against finite dimensional Wiener processes and $Q$-Wiener processes are both particular examples of the generalised cases.
\begin{proposition}
	Let $U,\,H$ be two Hilbert spaces, $(W_t)_{t\in [0,T]}$ be a generalised Wiener process defined on $U$ with covariance $Q$ and $\varphi \in \msH_T(W)$ be a stochastically integrable process. Then the stochastic integral $[0,T]\ni t\mapsto \varphi \bigcdot W_t\in H$ defined by \eqref{eq:GenStochIntBis} is a square integrable, continuous, local martingale with quadratic variation,
	\begin{equation}\label{eq:GenItoQV}
		[\varphi \bigcdot W]_t = \int_0^t (\iota^{-1}Q^{1/2}\varphi_s)(\iota^{-1}Q^{1/2}\varphi_s)^\ast\dd s.
	\end{equation}	
\end{proposition}
\begin{proof}
	Let $(\varphi^n)_{n\geq 1}$ be a simple approximation to the $L^0_2$ valued process $\varphi$ and consider the approximate stochastic integral,
	\begin{equation*}
		\varphi^n \bigcdot W_t =\int_0^t (\varphi^n_s\circ \iota^{-1})\dd W_s := \sum_{i=1}^{k_n} (\varphi^n_i\circ \iota^{-1}) (W_{t_{i}}-W_{t_{i-1}}),
	\end{equation*}
	where $(\varphi^n_i)_{n\geq 1, i \in \{1,\ldots,k_n\}}$ are $L^0_2$ valued random variables and where $0=t^n_0<\cdots t^n_{k_n}=t$. Using the independent increment property and the explicit distribution of $W_t-W_s$, we can compute the quadratic variation directly, to give, for any $n\geq 1$,
	\begin{equation*}
		[\varphi^n \bigcdot W]_t = \sum_{i=1}^{k_n} (\iota^{-1}Q^{1/2}\varphi^n_i)(\iota^{-1}Q^{1/2}\varphi_i^n)^\ast (t_i -t_{i-1}).
	\end{equation*}
	Taking limits on both sides gives the identity \eqref{eq:GenItoQV}, where the right hand side is understood as a Riemann--Stieltjes integral.
\end{proof}
In the remainder of the chapter we will work with a more general class of processes known as semi-martingales.
\begin{definition}
	We say that a stochastic process $(X_t)_{t\in [0,T]}$ is a semi-martingale if there exists a finite variation process $(\psi_t)_{t\in [0,T]}$ and a martingale $(M_t)_{t\in [0,T]}$ such that,
	\begin{equation*}
		X_t = \psi_t+ M_t.
	\end{equation*}
\end{definition}
It follows from Proposition \ref{prop:QVProperties} that this representation is essentially unique and furthermore that $[X]=[M]$ and $[[X]]=[[M]]$. In the theory of finite dimensional martingales, the well known result of It\^o's formula, shows both that the space of semi-martingales is closed under composition with twice differentiable functions as well as providing a stochastic version of the chain rule in It\^o calculus. We state this result below, in both the finite and infinite dimensional settings in the particular case of semi-martingales whose martingale part is given by a particular stochastic integrals.
\begin{theorem}[It\^o's Formula in Finite Dimensions]\label{th:FiniteDimIto}
	Let $T>0$, $(\Omega,\mcF,(\mcF_t)_{t\in [0,T]},\mbP)$ be a filtered probability space, carrying an $\mbR^d$ valued Brownian motion, $(W_t)_{t\geq0}$, $\psi \in L^2(\Omega\times[0,T]; \mbR^d)$, $\varphi \in L^2(\Omega\times[0,T];\mbR^{d\times d})$ and $F\in C^2(\mbR^d;\mbR)$. Then, for a process $(X_t)_{t\geq 0}$ satisfying,
	\begin{equation*}
		X_t = X_0 + \int_0^t \psi_s \dd s + \int_0^t \varphi_s \dd W_s,
	\end{equation*}
	one has,
	\begin{equation}\label{eq:FiniteDimIto}
		\begin{aligned}
			F(X_t) &= F(X_0) + \int_0^t \psi_s \cdot \nabla F(X_s)\dd s  + \int_0^t  \varphi_s  \nabla F(X_s)\cdot \dd W_s \\
			&\quad + \frac{1}{2} \int_0^t \Tr \left(\varphi^\perp_s\nabla^2 F(X_s) \varphi_s^\perp \right)\dd s.
		\end{aligned}
	\end{equation}
\end{theorem}
\begin{proof}[Sketch of Proof]
	We only give a sketched proof, details can be found as \cite[Thm. 4.2.1]{oksendal_13}. It suffices to establish the same result for $d=1$ and by Stone--Weierstrass for $F(x)$ a polynomial. First, we let $(X_t)_{t\in [0,T]},\,(Y_t)_{t\in [0,T]}$ be two semi-martingales and show that
	\begin{equation}\label{eq:ItoIBP}
		X_tY_t = X_0Y_0 + \int_0^t X_s \dd Y_s + \int_0^t Y_s\dd X_s + \int_0^t \dd [ X,Y]_s.
	\end{equation}
	Note that this is the usual integration by parts formula for finite variation process with the addition of the integral against $\dd [X,Y]_s$ - which is an artefact of the It\^o integral. To show \eqref{eq:ItoIBP}, let $t\in [0,T]$ be fixed and $\pi_n(t) = \{0=t_0,t_1,\ldots,t_N=t \}$ be a partition of $[0,t]$. Then, $\mbP$-a.s. we have,
	\begin{align*}
		X_tY_t -X_0Y_0 &= \sum_{i=1}^N X_{t_{i+1}}Y_{t_{i+1}}-X_{t_i}Y_{t_i}\\
		&= \sum_{i=1}^N X_{t_{i}}(Y_{t_{i+1}}-Y_{t_i}) + Y_{t_{i}}(X_{t_{i+1}}-X_{t_i}) + (X_{t_{i+1}}-X_{t_i})(Y_{t_{i+1}}-Y_{t_i})\\
		&\overset{N\rightarrow \infty}{\longrightarrow}\int_0^t X_s\dd Y_s + \int_0^t Y_s \dd X_s + \int_0^t \dd [X,Y]_s.
	\end{align*}
	If $X,\,Y$ were finite variation processes the last term can only have zero as a limit and we recover the usual integration by parts formula. With \eqref{eq:ItoIBP} in hand, let us assume that
	\begin{equation*}
		X_t^n = X^n_0 + n \int_0^t \psi_s X_t^{n-1}\dd s + n \int_0^t \varphi_s X_s^{n-1}\dd W_s + \frac{n(n-1)}{2} \int_0^t \varphi_s^2 X_s^{n-2}\dd s.
	\end{equation*}
	It follows therefore, using that $[X,X^n]_t =[W,W]_t = t$, that,
	\begin{align*}
		X_t^{n+1} &= X_0^{n+1} + (n+1) \int_0^t \psi_s X_s^n \dd s + (n+1) \int_0^t \varphi_s X_s^{n}\dd W_s \\
		&\quad + \frac{(n+1)n}{2} \int_0^t \varphi_s^2 X_s^{n-1}\dd s.
	\end{align*}
	Since \eqref{eq:FiniteDimIto} holds by definition for $F(x)=x$ we conclude by induction that \eqref{eq:FiniteDimIto} holds for any polynomial and so by Stone--Weierstrass for any $C^2$ function. The extension to higher dimensions follows by arguing component-wise and the identity $[W^i,W^j]_t = \delta_{ij} t$.
\end{proof}
One can see from the proof  of Theorem \ref{th:FiniteDimIto} that if the Brownian motion were a finite variation process, then the stochastic integral, as defined, would not contribute the third integral term. Alternatively, if one considers stochastic processes defined by Stratonovich integrals then it is well known that the usual chain rule also applies in this case. The same is true in infinite dimensions, see \cite{daprato_zabczyk_14}. Although Stratonovich integrals preserve the usual chain rule, they are neither predictable processes nor local martingales, informally speaking they look microscopically into the future.
\begin{remark}
	One can easily include a time inhomogeneous function $F:[0,T]\times \mbR^d\rightarrow \mbR$ in Theorem \ref{th:FiniteDimIto} provided $t\mapsto F(t,x)$ is continuously differentiable for all $x\in \mbR^d$. In this case one must include the term, $\int_0^t \partial_t F(s,X_s)\dd s$, on the right hand side of \eqref{eq:FiniteDimIto}. We cover this in the infinite dimensional setting below without proof.
\end{remark}
The It\^o formula generalises to the infinite dimensional setting, however,we refer to \cite[Thm. 4.32]{daprato_zabczyk_14} for the proof. Given a Hilbert space $H$ and a functional $F \in \mcL(H;\mbR)$, we write $D_x^1F\in \mcL(H;\mcL(H;\mbR)),\,D^2_{xx}F \in \mcL(H\otimes H;\mcL(H;\mbR))$ for the first and second Fr\'echet derivatives respectively. For $F:[0,T]\rightarrow \mcL(H;\mbR)$ we write $\partial_t F$ for the functional obtained by taking the time derivative of $F(t,x)$ in the first variable for every $x\in H$. We write $C^{1,2}([0,T]\times H;\mbR)$ for the space of time inhomogeneous functionals, once continuously differentiable in time and and twice continuously, Fr\'echet differentiable in $H$.
\begin{theorem}[It\^o's Formula in Infinite Dimensions]\label{th:InfiniteDimIto}
	Let $T>0$, $U,\,H$ be Hilbert spaces, $(\Omega,\mcF,(\mcF_t)_{t\in [0,T]}, \mbP)$ be a filtered probability space, carrying a generalized Wiener process $(W_t)_{t\in [0,T]}$ defined on $U$, with covariance $Q$ and assume $\psi \in L^2(\Omega\times[0,T];H)$, a $L_2(U_0;H)$ valued, predictable process, $\varphi\in \msH_T(W)$ and $X_0 \in H$. Then given $F \in C^{1,2}([0,T]\times H;\mbR)$ and a stochastic process, $(X_t)_{t\in [0,T]}$, satisfying,
	\begin{equation*}
		X_t = X_0 + \int_0^t \psi_s\dd s + \int_0^t \varphi_s\dd W_s, 
	\end{equation*}
	it holds that,
	\begin{equation}\label{eq:InfiniteDimIto}
		\begin{aligned}
			F(t,X_t)&= F(0,X_0) + \int_0^t \left(\partial_t F(s,X_s) + \langle D^1_x F(s,X_s),\psi_s\rangle\right) \dd s \\
			&\quad + \int_0^t \langle D^1_xF(s,X_s),\varphi_s\dd W_s\rangle\\
			&\quad + \frac{1}{2}\int_0^t \Tr \left[D^2_{xx} F(s,X_s)(\varphi_sQ^{1/2})(\varphi_sQ^{1/2})^\ast\right]\dd s .
		\end{aligned}
	\end{equation}
\end{theorem}
\begin{proof}
	We refer to the proof of \cite[Thm. 4.32]{daprato_zabczyk_14} for the general case. Within the proof of Theorem \ref{th:SPDEWellPosed} we establish identity \eqref{eq:InfiniteDimIto} in the particular case of $[0,T]\ni t\mapsto u_t\in H$ is a solution to \eqref{eq:AbstractSPDE} and $F(t,u_t)=\frac{1}{2}|u_t|_H$.
\end{proof}

\section{Well-Posedness of Monotone and Coercive SPDE}
Throughout we assume that we are given a Gelfand triple $(V,H,V^\ast)$, as in Chapter \ref{ch:PDE}, a fourth Hilbert space $U$ and a white noise process $(W_t)_{t\in [0,T]}$, defined on $U$. Recall that a white-noise process is a generalized Wiener process with covariation equal to the identity, $I:U\rightarrow U$, see Example \ref{ex:WhiteNoiseProcess}. Furthermore, we assume we are given operators $A: V\rightarrow V^\ast$, $B:V\rightarrow L^0_2:= L_2(U,H)$ and such that $(B(u_t))_{t\in [0,T]}\in \msH_T(W)$ for all square integrable processes, $[0,T]\ni t\mapsto u_t\in V$. Our aim is to establish well-posedness of \eqref{eq:AbstractSPDE}, which we interpret as an integral equation,
\begin{equation}\label{eq:SPDEIntegralForm}
	u_t = u_0 + \int_0^t A(u_s)\dd s + B(u)\bigcdot W_t,
\end{equation}
where the right hand side defines an element of $V^\ast$ for all square integrable processes $[0,T]\ni t \mapsto u_t\in V$. Note that in the above setting we allow for $B(u)= \tilde{B}\circ Q^{1/2}(u)$ for $Q \in L_1(U;H)$ a trace-class covariance operator and $\tilde{B}\in L_2(Q^{1/2}U;H)$. Thus we cover the case of $W$ a $Q$-Wiener process as well.\\ \par
We retain the conventions from Chapters \ref{ch:Background} and \ref{ch:PDE}, writing $|\,\cdot\,|_H$ for the norm on $H$ and for example, $\|\,\cdot\,\|_{V},\, \|\,\cdot\,\|_{L_2}$ for the norms on $V$ and $L_2(U;H)$ respectively. Unless otherwise stated, the inner product $\langle\,\cdot\,,\,\cdot\,\rangle$ will denote the duality pairing between $V,\,V^\ast$ or the inner product on $H$, whichever makes sense.\\ \par 
We list the following additional assumptions, which in some cases are either direct, or analogous, repetitions of those of Chapter \ref{ch:PDE} but are kept here for completeness. 
\begin{assumption}[Coercivity]\label{ass:CoercivitySPDE}
	There exist $\alpha,\,\lambda,\nu>0$ and $\nu \in L^1([0,T];\mbR)$ such that, for any $u \in V$,
	\begin{equation}\label{eq:CoerviceSPDE}\tag{H1}
		2\langle A(u),u\rangle + \|B(u)\|^2_{L_2(U;H)} + \alpha\|u\|_V ^2 \leq \lambda|u|^2_H+\nu.
	\end{equation}
\end{assumption}
\begin{remark}
	The inclusion of the extra, real parameter $\nu>0$ here, compared to Assumption \ref{ass:CoerciveNLPDE}, is to allow for possibly lower order, non-stochastic terms in the equation.
\end{remark}
\begin{assumption}[Monotonicity]\label{ass:MonotoneSPDE}
	There exists a $\lambda>0$ such that for any $u,w \in V$,
	\begin{equation}\label{eq:MontoneSPDE}\tag{H2}
		2\langle A(u)-A(w),u-w\rangle + \|B(u)-B(w)\|^2_{L_2(U;H)}\leq \lambda |u-w|_H^2.
	\end{equation}
\end{assumption}
\begin{assumption}[Linear Growth]\label{ass:LinearGrowthSPDE}
	There exists a $c>0$ such that for all $u \in V$,
	\begin{equation}\label{eq:LinearGrowthSPDE}\tag{H3}
		\|A(u) \|_{V^\ast} \leq c(1+\|u\|_{V}).
	\end{equation}
\end{assumption}
\begin{assumption}[Weak Continuity]\label{ass:WeakContinuitySPDE}
	For any $u,v,w \in V$, the mapping,
	\begin{equation}\label{eq:WeakContinuitySPDE}\tag{H4}
		\mbR \ni \theta \mapsto \langle A(u+\theta w),v\rangle \in \mbR,
	\end{equation}
	is continuous.
\end{assumption}
\begin{remark}\label{rem:BLinearBound}
	Since by definition $|u|_H\leq \|u\|_{V}$, combining \eqref{eq:CoerviceSPDE} with $\eqref{eq:LinearGrowthSPDE}$ shows that we also have $\|B(u)\|_{L_2}\leq c(1+\|u\|_{V})$, for a possibly different $c>0$.
\end{remark}
\begin{remark}
	As in the deterministic case, without loss of generality, we can perform the transformation $u^\lambda := e^{-\lambda t/2}u$, which instead solves \eqref{eq:AbstractSPDE} with $A,B$ replaced by,
	\begin{align*}
		A^\lambda(u) &:= e^{-\lambda t/2} A(e^{\lambda t/2}u) - \frac{\lambda}{2} u,\\
		B^\lambda(u)&:= e^{-\lambda t/2} B(e^{\lambda t/2}u).
	\end{align*}
	As in the previous chapter, it is easily checked that if $A,B$ satisfy \eqref{eq:CoerviceSPDE} and \eqref{eq:MontoneSPDE}, then $A^\lambda,B^\lambda$ satisfy, for a possibly new $\alpha,\nu>0$,
	\begin{equation}\label{eq:CoerciveSPDEPrime}\tag{H$1'$}
		2\langle A^\lambda(u),u \rangle + \|B^\lambda(u)\|^2_{L_2(U;H)} + \alpha \|u\|_V^2 \leq \nu,
	\end{equation}
	and
	\begin{equation}\label{eq:MonotoneSPDEPrime}\tag{H$2'$}
		2\langle A(u)-A(w),u-w\rangle + \|B(u)-B(w)\|^2_{L_2(U;H)} \leq 0,
	\end{equation}
	for all $u,\,w \in V$.
\end{remark}
\begin{remark}
	Again, as in the deterministic case, for $p>2$, provided we replace \eqref{eq:CoerviceSPDE} with,
	\begin{equation}\label{eq:CoerciveMSPDE}\tag{H$1_p$}
		2\langle A(u),u\rangle +\|B(u)\|^2_{L_2} + \alpha \|u\|^p_{V}\leq \lambda |u|^{2}_H+\nu,
	\end{equation}
	we may replace \eqref{eq:LinearGrowthSPDE} with,
	\begin{equation}\label{eq:MGrowthSPDE}\tag{H$3_p$}
		\|A(u)\|_{V^\ast}\leq c(1+ \|u\|^{p-1}_{V}).
	\end{equation}
\end{remark}
From now on we work under \eqref{eq:CoerciveSPDEPrime}, \eqref{eq:MonotoneSPDEPrime}, \eqref{eq:LinearGrowthSPDE} and \eqref{eq:WeakContinuitySPDE} only. Before stating the main theorem of this chapter we define the notion of weak solution to the SPDE \eqref{eq:AbstractSPDE}, which gives meaning to \eqref{eq:SPDEIntegralForm} as an equality holding in $V^\ast$.
\begin{definition}\label{def:SPDEWeakSol}
	Given $(V,\,H,\,V^\ast)$ and $A,\,B,\,W$ as in the start of the section, we say that an $(\mcF_t)_{t\in[0,T]}$-adapted process $[0,T]\times \Omega \ni (t,\omega)\mapsto u_t(\omega)\in V$ is a weak solution to \eqref{eq:AbstractSPDE}, if  there exists a modification, which we do not relabel, such that
	\begin{equation}\label{eq:SPDEDefIntegrabillity}
		u \in L^2(\Omega\times[0,T]; V),
	\end{equation}
	and for every $v\in V$ and $t\in [0,T]$, $\mbP$-a.s.
	\begin{equation}\label{eq:SPDEWeakForm}
		\langle u_t,v\rangle = \langle u_0,v\rangle + \int_0^t \langle A(u_s),v\rangle \dd s + \int_0^t \langle v,B(u_s)\dd W_s\rangle.
	\end{equation}
\end{definition}
\begin{remark}
	If we replace \eqref{eq:CoerviceSPDE}, \eqref{eq:LinearGrowthSPDE} with \eqref{eq:CoerciveMSPDE}, \eqref{eq:MGrowthSPDE} for some $m>2$ then the condition $u \in L^2(\Omega\times[0,T];V)$ in \eqref{eq:SPDEDefIntegrabillity} should be replaced by the condition $u \in L^m(\Omega\times [0,T] ;V)$.
\end{remark}
\vspace{0.3em}
We have the following well-posedness result which is an analogue of Theorem \ref{th:MonotonePDE}. We retain the general assumptions stated at the beginning of this section on the spaces $(V,\,H,\,V^\ast)$ and a white noise process, $(W_t)_{t\in [0,T]}$, defined on $U$.
\begin{theorem}\label{th:SPDEWellPosed}
	Let $T>0$, $u_0 \in L^2(\Omega;H)$ and $A,\,B$ satisfy Assumptions \eqref{eq:CoerciveSPDEPrime}, \eqref{eq:MonotoneSPDEPrime}, \eqref{eq:LinearGrowthSPDE}, \eqref{eq:WeakContinuitySPDE}. Then there exists a unique, weak solution $(u_t)_{t\in [0,T]}$ to \eqref{eq:AbstractSPDE}, in the sense of Definition \ref{def:SPDEWeakSol}. Furthermore,
	\begin{equation*}
		\mbE\left(\sup_{t \in [0,T]}\|u_t\|^2_{H}\right)<\infty.
	\end{equation*}
\end{theorem}
\begin{proof}
	We proceed in a similar fashion to the proof of Theorem \ref{th:MonotonePDE}, only now we have to also handle the stochastic term. We prove uniqueness first, using Theorem \ref{th:InfiniteDimIto}, followed by existence. Before tackling these steps we obtain an ad hoc, a priori It\^o formula for $t\mapsto |u_t|^2_H$ for any solution $u$ to \eqref{eq:AbstractSPDE}.
	\paragraph{Ad Hoc It\^o Formula:} Let $(u_t)_{t\in [0,T]}$ be a solution to \eqref{eq:AbstractSPDE} in the sense of Definition \ref{def:SPDEWeakSol}. Let $(e_k)_{k=1}^\infty\subset V$ be an orthonormal basis of $H$ and define $u_k := \langle u,e_k\rangle$, so that we have
	\begin{equation*}
		|u_t|_H^2 = \sum_{k=1}^\infty \langle u_t,e_k\rangle^2 = \sum_{k=1}^\infty u_{k;t}^2.
	\end{equation*}
	Taking $e_k$ as the test function in \eqref{eq:SPDEWeakForm}, we see that $u_{k;t}$ solves,
	\begin{equation*}
		u_{k;t} = \langle u_0,e_k\rangle + \int_0^t  \langle A(u_s),e_k\rangle \dd s + \int_0^t \langle e_k,B(u_s)\dd W_s\rangle.
	\end{equation*}
	Recall that given a basis $(f_k)_{k\geq 1}$ of $U$, the white noise process has the representation, 
	\begin{equation*}
		W_t := \sum_{k\geq 1} f_k \beta^k_t,
	\end{equation*}
	where $(\beta^k)_{k\geq 1}$ is a family of i.i.d, real valued Brownian motions. It follows from the representation of $(B(u_t))_{t\in [0,T]}\subset  L_2(U;H)$, \eqref{eq:HSandTraceClassRep} that there exist $[0,T]\ni t\mapsto (\lambda_{k;t})_{k\geq 1}\subset \mbR_+$ such that for all $x\in U$,
	\begin{equation*}
		B(u_t)x = \sum_{k\geq 1} \sqrt{\lambda_{k;t}} \langle x,f_k\rangle_U e_k, \quad t\in [0,T].
	\end{equation*}
	Where we may assume that we chose the basis, $(e_k)_{k\geq 1}$ of $H$, suitably. Note here that $\|B(u_t)\|_{L_2} = \sum_{k\geq 1}\lambda_{k;t} $. Since $\langle W,f_k\rangle = \beta^k$, we can formally write the integral term of \eqref{eq:SPDEWeakForm}, without a test function, as
	\begin{equation*}
		\int_0^t B(u_s) \dd W_s = \sum_{k\geq 1} \int_0^t \sqrt{\lambda_{k;s}} e_k \dd \beta^k_s.
	\end{equation*}
	This is not strictly valid, since we know that $\mbP$-a.s. $W$ does not take values in $U$. However, by considering an additional space $U_1$, containing $U$ with a Hilbert--Schmidt embedding $\iota :U\rightarrow U_1$ and instead defining the same object for $B(u)\circ \iota^{-1}$ we arrive eventually at the same expression. For concision we do not include this argument in detail and instead conclude directly, that for any $k\geq 1$, we now have
	\begin{equation*}
		u_{k;t} = \langle u_0,e_k\rangle + \int_0^t  \langle A(u_s),e_k\rangle \dd s + \int_0^t \sqrt{\lambda_{k;s}} \dd \beta^k_s
	\end{equation*}
	This is a one dimensional SDE and so we may apply Theorem \ref{th:FiniteDimIto}. Note that the map $x\mapsto \frac{1}{2}|x|^2$ has first derivative equal to $x$ and second derivative equal to $1$ so that,
	\begin{align*}
		\frac{1}{2}|u_{k;t}|^2 &= \frac{1}{2}|u_{k;0}|^2 + \int_0^t \langle A(u_s),e_k\rangle u_{k;s}\dd s + \int_0^t  \sqrt{\lambda_{k;s}}u_{k;s}\dd \beta^k_s + \frac{1}{2}\int_0^t \lambda_{k;s}\dd s.
	\end{align*}
	So now summing over $k$ we have,
	\begin{align*}
		\frac{1}{2}|u_t|^2_H &= \frac{1}{2}\sum_{k=1}^\infty |u_{k;t}|^2 \\
		&= \sum_{k\geq 1} |u_{k;0}|^2  + \sum_{k\geq 1}\int_0^t   \langle A(u_s),e_k\rangle u_{k;s}\dd s + \sum_{k\geq 1} \int_0^t \sqrt{\lambda_{k;s}} u_{k;s}\dd \beta^k_s\\
		&\quad +\sum_{k\geq 1} \int_0^t \lambda_{k;s} \dd s.
	\end{align*}
	So in more compact form,
	\begin{equation*}\label{eq:SPDEL2Ito}
		\begin{aligned}
			|u_t|^2_H &= |u_0|_H^2 + 2\int_0^t \langle A(u_s),u_s\rangle \dd s + 2\int_0^t \langle u_s,B(u_s)\dd W_s\rangle \\
			&\quad  + \int_0^t \Tr \left[B(u_s) B(u_s)^\ast \right]\dd s.
		\end{aligned}
	\end{equation*}
	\paragraph{Uniqueness:} We employ the monotonicity assumption, \eqref{eq:MonotoneSPDEPrime}. Let $u,\,w \in L^2(\Omega\times[0,T];V)$ be two solutions to \eqref{eq:AbstractSPDE} in the sense of \ref{def:SPDEWeakSol}. By assumption,
	\begin{equation*}
		u_t-w_t = \int_0^t A(u_s)-A(w_s)\dd s + \int_0^t B(u_s)-B(w_s)\dd W_s,
	\end{equation*}
	in the sense of definition \ref{def:SPDEWeakSol}. Therefore, by a similar argument as we used to prove \eqref{eq:SPDEL2Ito}, we obtain the identity
	\begin{equation}\label{eq:UWDiff1}
		\begin{aligned}
			|u_t-w_t|^2_H &= 2 \int_0^t \langle A(u_s)-A(w_s),u_s-w_s\rangle \dd s \\
			&\quad + 2 \int_0^t \langle u_s-w_s,(B(u_s)-B(w_s))\dd W_s\rangle \\
			&\quad + \int_0^t \Tr\left((B(u_s)-B(w_s))(B(u_s)-B(w_s))^*\right)\dd s.
		\end{aligned}
	\end{equation}
	By assumption the stochastic integral term is only a local martingale, so it does not directly vanish under the expectation. Let us define the sequence of stopping times, for $R\geq 0$,
	\begin{equation*}
		\tau_R := \inf \left\{ t\in (0,T]\,:\,  \int_0^t \|u_s\|^2_{V} \dd s \vee \int_0^t \|w_s\|^2_V\dd s \geq R \right\}.
	\end{equation*}
	It follows from Remark \ref{rem:BLinearBound} that the stopped process $[0,T]\ni t\mapsto \mathds{1}_{t\leq \tau_R}B(u_t) \in \msH_T^2(W)$, i.e it is a square integrable integrand and so the stopped integral,
	\begin{equation*}
		\int_0^{\tau_R\wedge t} B(u_s)\dd W_s = \int_0^t \mathds{1}_{s\leq \tau_R} B(u_s)\dd W_s,
	\end{equation*}
	is a true martingale. Therefore applying \eqref{eq:UWDiff1} at $t\wedge \tau_R$ for any $R\geq 0$, and then taking expectations on both sides we have that,
	\begin{align*}
		\mbE[|u_{t\wedge \tau_R} - w_{t\wedge \tau_R}|^2_H] &=  2\int_0^{t\wedge \tau_R}  \mbE\left[\langle A(u_s)-A(w_s),u_s-w_s\rangle+\|B(u_s)-B(w_s)\|^2_{L_2}\right]  \dd s .
	\end{align*}
	So applying \eqref{eq:MonotoneSPDEPrime} gives that
	\begin{align*}
		\mbE[|u_{t\wedge \tau_R} - w_{t\wedge \tau_R}|^2_H] \leq 0.
	\end{align*}
	Since this holds for any $t\in [0,T]$ and $\lim_{R\rightarrow \infty}\tau_R =\infty$, applying the dominated convergence theorem, Theorem \ref{th:DCTBochner}, we conclude that $u = w \in L^2(\Omega\times[0,T];H)$ and therefore also $\mbP$-a.s. in $L^2([0,T];H)$.
	\paragraph{Existence:} We argue in a similar fashion as in the proof of Theorem \ref{th:MonotonePDE} only now the system of finite dimensional ODE is replaced by a finite dimensional system of SDE. For $n\geq 1$, let us define the process $(u_{n;t})_{t\geq 0}$ to be the unique $L^2(\Omega\times[0,T];V_n)$ valued solution to the system,
	\begin{equation}\label{eq:FiniteDimSDE}
		\begin{aligned}
			\langle u_{n;t},e_k\rangle &= \langle u_{0},e_k\rangle + \int_0^t \langle A(u_{n;s}),e_k\rangle \dd s +  \int_0^t \langle e_k, B(u_{n;t})\dd W_s\rangle ,
		\end{aligned}
	\end{equation}
	for $k=1\,\ldots,n$. As with \eqref{eq:GalerkinNLODE}, well-posedness of the system \eqref{eq:FiniteDimSDE} does not follow directly from the standard Cauchy--Lipschitz theory for SDEs, \cite[Thm. 5.2.1]{oksendal_13}. However, strong existence (in a probabilistic sense) and uniqueness can be established using the assumptions \eqref{eq:CoerciveSPDEPrime}, \eqref{eq:MonotoneSPDEPrime} and essentially applying an Euler scheme. The steps are technical but follow in spirit the same arguments one would make in the deterministic case, only requiring the additional technicality of a stopping time approximation. Full details can be found in \cite[Sec. 3]{prevot_rockner_07}. Taking well posedness of the system \eqref{eq:FiniteDimSDE} as given for now, the next step is to establish a uniform bound on the family $(u_n)_{n\geq 1}$.\\ \par 
	Arguing along the same lines as we did to establish \eqref{eq:SPDEL2Ito}, but this time only summing up to $k=n$ we obtain the identity,
	\begin{equation}\label{eq:FiniteDimSDEHNorm}
		\begin{aligned}
			|u_{n;t}|^2_H &= \sum_{k=1}^n |\langle u_0,e_k\rangle|^2 + 2\int_0^t \langle A(u_{n;s}),u_{n;s}\rangle \dd s + 2  \int_0^t \langle u_{n;s},B(u_{n;s}) \dd W_s\rangle \\
			&\quad  + \sum_{k=1}^n \int_0^t \lambda_{k;s}\dd s.
		\end{aligned}
	\end{equation}
	Again, we would like to take the expectation and so remove the stochastic integral, however, as in the proof of uniqueness we first need to take a sequence of stopping times 
	$$\tau_R:= \inf \left \{t \in (0,T]\,:\, \int_0^t \|u_{n;s}\|^2_{V}\dd s \geq R\right\},$$ 
	and consider the stopped process, $u_{n;t\wedge \tau_R}$, for which we have the bound,
	\begin{equation*}
		\mbE[|u_{n;t\wedge \tau_R}|_H^2] \leq |u_0|^2_H + 2\mbE\left[\int_0^{t\wedge \tau_R} \langle A(u_{n;s}),u_{n;s}\rangle \dd s \right] + \mbE\left[\int_0^{t\wedge \tau_R} |B(u_{n;s})|^2_{L_2}\dd s  \right].
	\end{equation*}
	So applying \eqref{eq:MonotoneSPDEPrime} we obtain,
	\begin{equation*}
		\mbE\left[|u_{n;t\wedge \tau_R}|_H^2 + \alpha \int_0^{t\wedge \tau_R} \|u_{n;s}\|^2_V\dd s \right]  \leq |u_0|^2_H + \nu\mbE[t\wedge \tau_R].
	\end{equation*}
	Thus, since the right hand side is independent of $n\geq 1$ and $\tau_R \rightarrow \infty$ $\mbP$-a.s, as $R\rightarrow \infty$ we see that we have,
	\begin{equation}\label{eq:SDEAPriori1}
		\sup_{n\geq 1}\sup_{t\in [0,T]} \mbE\left[|u_{n;t}|_H^2 + \alpha \int_0^t \|u_{n;s}\|^2_V\dd s \right]  <\infty.
	\end{equation}
	This is almost the a priori bound we want, we only need to exchange the supremum over $t\in [0,T]$ with the expectation. Since the supremum is a convex function and the expectation is really an integral, exchanging their order in \eqref{eq:SDEAPriori1}, is not obviously valid. However, using the BDG inequality, Theorem \ref{th:BDG}, with $p=1$ and $\tau=T$, we have the following,
	\begin{align*}
		\mbE \left[ \sup_{t \in [0,T]} \left| \int_0^t \langle u_{n;s},B(u_{n;s})  \dd W_s \rangle \right|_H\right] &\leq C \mbE\left[ \left( \sum_{k=1}^n \int_0^T \sqrt{\lambda_{k;t}}\langle u_{n;t},e_k\rangle\dd t  \right)^{\frac{1}{2}}  \right]\\
		&\leq C \mbE \left[\left(\sum_{k=1}^n \sup_{t \in [0,T]} |\langle u_{n;t},e_k\rangle| \int_0^T \sqrt{\lambda_{k;s}}\dd t\right)^{\frac{1}{2}}  \right] \\
		&\leq \frac{1}{4} \mbE\left[\sup_{t \in [0,T]} |u_{n;t}|^2_H\right] + C^2 \mbE\left[ \int_0^T \|B(u_{n;t})\|^2_{L_2} \dd t\right] .
	\end{align*}
	The final inequality here follows from H\"older and Young's product inequality. If we take a supremum over $t\in [0,T]$, before taking the expectation in \eqref{eq:FiniteDimSDEHNorm}, we therefore obtain
	\begin{align*}
		\mbE\left[ \sup_{t \in [0,T]} |u_{n;t}|_H^2\right] &\leq |u_0|^2_H + 2\mbE\left[ \int_0^T |\langle A(u_{n;s}),u_{n;s}\rangle| \dd s \right] + (1+2C^2)\mbE\left[\int_0^T \|B(u_{n;s})\|^2_{L_2}\dd s \right]\\
		&\quad + \frac{1}{2}\mbE\left[\sup_{t \in [0,T]} |u_{n;t}|^2_H\right].
	\end{align*}
	From which it follows,  using \eqref{eq:CoerciveSPDEPrime}, Remark \ref{rem:BLinearBound} and linearity of the expectation, that we  have
	\begin{equation}\label{eq:SDEAPriori2}
		\sup_{n\geq 1 }\mbE\left[ \sup_{t \in [0,T]} |u_{n;t}|_H^2\right]  \leq  |u_0|^2_H    +  \sup_{n\geq 1}\mbE\left[ \int_0^T (1+\tilde{C})\|u_{n;t}\|^2_V\dd t\right],
	\end{equation}
	for a new constant $\tilde{C}>0$ that incorporates the constant $\nu>0$ in \eqref{eq:CoerciveSPDEPrime}. It follows then from \eqref{eq:SDEAPriori1} that the right hand side of \eqref{eq:SDEAPriori2} is finite and so we finally obtain the bound,
	\begin{equation}\label{eq:SDEAPriori3}
		\sup_{n\geq 1} \mbE\left[ \sup_{t\in[0,T]} |u_{n;t}|^2_H + \int_0^T \|u_{n;t}\|^2_{V}\dd t\right] <\infty.
	\end{equation}
	From \eqref{eq:SDEAPriori3}, \eqref{eq:LinearGrowthSPDE} and Remark \ref{rem:BLinearBound} it follows that:
	\begin{enumerate}
		\item $(u_n)_{n\geq 1}$ is uniformly bounded in $L^2(\Omega;L^\infty([0,T];H))\cap L^2(\Omega\times[0,T];V)$,
		\item $(A(u_n))_{n\geq 1}$ is uniformly bounded in $L^2(\Omega\times[0,T];V^\ast)$,
		\item $(B(u_n))_{n\geq 1}$ is uniformly bounded in $L^2(\Omega\times[0,T];L_2)$.
	\end{enumerate}
	Thus, there exist subsequences, which we do not relabel converging weakly in the following senses,
	\begin{equation*}
		\begin{aligned}
			u_n &\rightharpoonup u \quad \text{ in } L^2(\Omega\times [0,T];V),\\
			A(u_n)&\rightharpoonup \xi \quad \text{ in } L^2(\Omega\times [0,T];V^\ast)\\
			B(u_n)&\rightharpoonup \eta \quad \text{ in }L^2(\Omega\times [0,T];L_2).
		\end{aligned}
	\end{equation*}
	Furthermore, we have that $u_n\overset{*}{\rightharpoonup} u$ in $L^2(\Omega;L^\infty([0,T];H))$. So just as in the proof of Theorem \ref{th:MonotonePDE} it is only left to show that $\zeta =A(u)$ and $\eta = B(u)$. The argument in this case is very similar.\\ \par 
	From \eqref{eq:MonotoneSPDEPrime}, for all $v \in L^2(\Omega\times [0,T];V)$ and $n\geq 1$, we have that
	\begin{equation*}
		\mbE\left[\int_0^T 2 \langle A(u_{n;t}) -A(v_t),u_{n;t}-v_t\rangle + \|B(u_{n;t})-B(v_t)\|^2_{L_2}\dd t \right] \leq 0.
	\end{equation*}
	We then claim that we have the inequality,
	\begin{equation}\label{eq:SPDELimInf}
		\hspace{-1em}\begin{aligned}
			\mbE\bigg[ \int_0^T& \langle \zeta_t ,u_t\rangle \dd t + \int_0^T\|\eta_t\|^2_{L_2}\dd t\bigg]  \leq \liminf_{n\rightarrow \infty} 	\mbE\left[\int_0^T 2 \langle A(u_{n;t}) ,u_{n;t}\rangle + \|B(u_{n;t})\|^2_{L_2}\dd t \right].
		\end{aligned}
	\end{equation}
	To see this, we recall from finite dimensional It\^o formula, that we have
	\begin{equation*}
		\mbE\left[	|u_{n;T}|_H^2 - |u_{n;0}|^2_H \right]= \mbE \left[\int_0^T \langle A(u_{n;t}),u_t\rangle  + \|B(u_{n;t})\|^2_{L_2}\,\dd t \right],
	\end{equation*}
	and passing to the limit, $n\rightarrow 0$, gives that,
	\begin{equation*}
		\mbE\left[	|u_{T}|_H^2 - |u_{0}|^2_H \right]= \mbE \left[\int_0^T \langle \zeta_t,u_t\rangle  + \|\eta_t\|^2_{L_2}\,\dd t \right].
	\end{equation*}
	By essentially the same reasoning we have that  $u_{n;T}\rightharpoonup u_T$ in $L^2(\Omega;H)$, and so since the map $\rho \mapsto \mbE\left[|\rho|^2_H\right]$ is convex we have that
	\begin{equation*}
		\mbE\left[	|u_{T}|_H^2 - |u_{0}|^2_H \right] \leq \liminf_{n\rightarrow \infty} \mbE\left[	|u_{n;T}|_H^2 - |u_{n;0}|^2_H \right],
	\end{equation*}
	which is equivalent to \eqref{eq:SPDELimInf}. With \eqref{eq:SPDELimInf} in hand, proceeding by a similar passage as in \eqref{eq:PDEZetaMonotone} we see that
	\begin{equation}\label{eq:SPDEZetaEtaMontone}
		\mbE\left[\int_0^T 2\langle \zeta_t - A(v_t),u_t-v_t\rangle + \|\eta_t - B(v_t)\|^2_{L_2}\dd t  \right]\leq 0.
	\end{equation}
	If we set $u=v$ in \eqref{eq:SPDEZetaEtaMontone} we immediately see that $\eta = B(u)$. To show that $\zeta= A(u)$, let $v= u-\theta w$ for $\theta>0$ and $w \in L^2(\Omega\times [0,T];V)$. Diving both sides by $\theta$ we see that
	\begin{equation*}
		\mbE\left[\int_0^T \langle \zeta_t-A(u_t-\theta w_t),w_t \rangle \dd t \right] \leq 0.
	\end{equation*}
	Letting $\theta \rightarrow 0$ and appealing to \eqref{eq:WeakContinuitySPDE} we deduce that
	\begin{equation*}
		\mbE\left[ \int_0^T \langle \zeta_t - A(u_t),w_t\rangle  \dd t \right]\leq 0, \quad \forall \, w \in L^2(\Omega \times [0,T];V),
	\end{equation*}
	from which it follows that $\zeta = A(u)$ in $L^2(\Omega \times [0,T];V^\ast)$.\\ \par
	So it follows that $u \in L^2(\Omega\times [0,T];V)\cap L^2(\Omega;L^\infty([0,T];H))$ is the unique solution to \eqref{eq:AbstractSPDE} in the sense of Definition \ref{def:SPDEWeakSol}. Finally, we prove that in fact, $\mbP$-a.s., $u \in C([0,T];H)$. To see this we recall that since $H$ is a Hilbert space, if $u_n \rightharpoonup u \in H$, and $|u_n|_H\rightarrow |u|\in \mbR$ one has
	\begin{equation*}
		|u_n-u|^2_{H} =\langle u-u_n,u-u_n\rangle = |u|^2_H-2\langle u,u_n\rangle + |u_n|^2 \rightarrow 0.
	\end{equation*}
	Therefore, since the ad hoc It\^o formula, \eqref{eq:SPDEL2Ito} show that for a sequence $t_k\rightarrow t$, it holds that $|u_{t_k}|_H\rightarrow |u_t|_H$, it suffices to show that $u_{t_k}\rightharpoonup u_t\in H$. Let $h\in H$ be arbitrary and $(h_n)_{n\geq 1}\subset V$ be a sequence converging to $h$ strongly in $H$. Let $\varepsilon>0$ and $n_\varepsilon\geq 1$ be large enough such that,
	\begin{equation*}
		\sup_{t \in [0,T]} |u_t|_H||h-h_n|_H\,\leq
		\, \frac{\varepsilon}{2}, \quad \text{for all }n\geq n_\varepsilon.
	\end{equation*}
	So it follows that
	\begin{align*}
		|\langle u_t,h\rangle -\langle u_{t_k},h\rangle| &\leq |\langle u_t,h-h_{n_\varepsilon}\rangle|+|\langle u_t-u_{t_k},h_{n_\varepsilon}\rangle| + |\langle u_{t_k},h-h_{n_\varepsilon}\rangle|\\
		&\leq \varepsilon + \|u_{t}-u_{t_k}\|_{V^\ast}\|h_{n_\varepsilon}\|_{V}.
	\end{align*}
	By definition of the weak form of the equation, we have that $u_{t_k} \rightarrow u_t$ strongly in $V^\ast$ and so we have that,
	\begin{equation*}
		\limsup_{n\geq n_\varepsilon} |\langle u_t,h\rangle - \langle u_{t_k},h\rangle| \,\leq \,\varepsilon.
	\end{equation*}
	Hence, by the arbitrariness of $\varepsilon>0$ the necessary convergence holds and we conclude the proof.
\end{proof}
\section{Examples}
We present some particular examples. We refer to \cite{daprato_zabczyk_14,prevot_rockner_07,pardoux_21_spde_introduction,dalang_koshnevisan_mueller_nualart_xiao_09,liu_rockner_15_spde_introduction} for more examples of SPDE, some of which are amenable to the method presented here.
\begin{example}[A Transport Noise SPDE]\label{ex:TransportSPDE}
Let us fix $\Gamma = \mbT^d$ and $H= L^2(\mbT^d)$, $V= \mcH^{1}(\mbT^d),\,V^\ast = \mcH^{-1}(\mbT^d)$. By the Poincar\'e inequality, \cite[Subsec 5.8.1]{evans_10}, we have the bound,
\begin{equation}
\|u-u_{\mbT^d}\|^2_{L^2}\leq \|\nabla u\|^2_{L^2},\quad u_{\mbT^d} := \frac{1}{|\mbT^d|}\int_{\mbT^d} u(x)\dd x.
\end{equation}
Therefore we may in principle describe the space $\mcH^1(\mbT^d)$ to be the set of all functions $u:\mbT^d\rightarrow \mbR$ with square integrable gradient. We see below that this suffices for our application.\\ \par
Then, assume we have a sequence of $L^2(\mbT^d)$ basis functions, $(e_k)_{k\geq 1}\subset C^\infty(\mbT^d;\mbR^d)$, such that $\sup_{k\geq 1}\|e_k\|_{L^2}=1$. Furthermore, we assume that we have a sequence of non-negative real numbers $(\sigma_k)_{k\geq 1} \subset \mbR_{\geq 0}$, then we consider the SPDE,
\begin{equation}\label{eq:TransportSPDE2}
	\begin{cases}
		\dd u_t = \Delta u_t \dd t + \sum_{k\geq 1} \sqrt{\sigma_k}e_k \cdot \nabla u_t \dd \beta_t^k ,&\\
		u\tzero =u_0,
	\end{cases}
\end{equation}
where $(\beta^k)_{k\geq 1}$ a family of i.i.d, real, Brownian motions. So under the assumption that $\sum_{k\geq 1}\sigma_k = \|\sigma\|_{\ell^1}<\infty$, writing $B(u_t):= \sum_{k\geq 1} \sqrt{\sigma_k} e_k\cdot \nabla u_t$ we have,
\begin{equation*}
	\|B(u_t)\|^2_{L_2} = \sum_{k\geq 1} \sigma_k \|e_k\|^2_{L^2}\|\nabla u_t\|^2_{L^2} = \|\sigma\|_{\ell^1} \|\nabla u_t\|^2_{L^2}.
\end{equation*}
Note that the quantity $u_{\mbT^d}$ is (at least formally) preserved by \eqref{eq:TransportSPDE2}, therefore for simplicity let us set $u_0$ to have zero mean. So we may work in the space of mean free functions and we truly have $\|u\|_{L^2} \leq \|\nabla u\|_{L^2}$.\\ \par
Since \eqref{eq:TransportSPDE2} is linear, by analogy with the arguments of Chapter \ref{ch:PDE}, we only need to establish coercivity in the sense of \ref{eq:CoerciveSPDEPrime}. With $A=\Delta$, for any $\alpha>0$, we directly have
\begin{equation*}
	2\langle \Delta u,u\rangle_{L^2} + \|B(u)\|^2_{L_2} + \alpha \|u\|^2_{H^1} = ( \|\sigma\|_{\ell^1}\|-2 + \alpha)\|\nabla u\|^2_{L^2}.
\end{equation*}
So in order to satisfy \eqref{eq:CoerciveSPDEPrime} we require the bound $\|\sigma\|_{\ell^1} <2$. This mirrors the result shown in Subsection \ref{subsec:WorkedExample} but now in arbitrary dimensions and with infinite dimensional noise. As in that case, if we also have $\div e_k =0$ for all $k\geq 1$ then the martingale term arising from It\^o's formula for $\|u_t\|_{L^2}$, formally becomes,
\begin{align*}
	\sum_{k\geq 1} \sigma_k \int_0^t \int_{\mbT^d} (e_k(x) \cdot \nabla u_{s}(x) )u_s(x)\dd x\dd \beta_s^k  &= \sum_{k\geq 1} \frac{\sigma_k}{2} \int_0^t \int_{\mbT^d} e_k(x) \cdot \nabla (u^2_{s}(x) )\dd x \dd \beta_s^k\\
	&= \sum_{k\geq 1} \frac{\sigma_k}{2} \int_0^t \int_{\mbT^d} \nabla \cdot (e_k(x) u^2_{s}(x) )\dd x \dd \beta_s^k\\
	&= 0.
\end{align*}
So with some care one can show that in this case, the martingale term in fact disappears $\mbP$-a.s., not only in expectation. \qed
\end{example}
\begin{example}\label{ex:stoch_rd_field}
Let $(\Omega,\mcF,\mbP)$ be a fixed probability space, $m\geq 2$ , $\Gamma\subset \mbR^d$ be a smooth bounded domain, $\ell \in \mbN$ and 
\begin{align*}
	B :  \mbR^d\times L^2(\Gamma) &\times \Omega \to \mbR^\ell,\\
	(t,u,\omega)&\mapsto B_t(u,\omega)
\end{align*}
be a stochastic process for which there exists a $C>0$ such that for all $u,\,w \in L^2(\Gamma)$,
\begin{equation*}
	\sup_{x\in \Gamma}|B(u,x,\omega) - B(w,x,\omega)|_{\mbR^\ell} \leq C\|u-w\|_{L^2} \quad \text{and}\quad \sup_{x\in \Gamma}|B(u,x,\omega)|_{\mbR^\ell} \leq C(1+\|u\|_{L^2(\Gamma)})\quad \mbP\text{-a.s.}
\end{equation*}
For concision we will drop the $x\in \Gamma$ and $\omega \in \Omega$ arguments from the notation for $B$. Then, for $\kappa>0$ and $W$ an $\mbR^\ell$ valued $\mbP$-Brownian motion, consider the stochastic field equation
	\begin{equation}\label{eq:stoch_rd_field}
	\begin{cases}
		\dd u_t = -\kappa |u|^{m-2}u\dd t + B(u)\cdot \dd W_t,&\\
		u\tzero =u_0.
	\end{cases}
\end{equation}
We fix our functional setting by choosing $V= L^m(\Gamma)$, $H=L^2(\Gamma)$ and $V^* = L^{\frac{m}{m-1}}(\Gamma)$. Note that in this case we view $B_t(u) : \mbR^\ell \to \mbR$ and so by our global Lipschitz assumption, we have the $\mbP$-a.s. bound
\begin{equation*}
	\|B(u)\|^2_{L^\infty_\Gamma L_2} \coloneqq \|B(u)\|^2_{L^\infty_\Gamma L_2(\mbR^\ell;\mbR)} \leq C(1+\|u\|_{L^2(\Gamma)}^2).
\end{equation*}
We then define the operator
\begin{align*}
	A : L^m(\Gamma)&\to L^{\frac{m}{m-1}}(\Gamma)\\
	u &\mapsto -u|u|^{m-2}.
\end{align*}
We first check \eqref{eq:CoerciveMSPDE}. To wit, by our growth assumption on we have
\begin{align*}
	\langle A(u),u\rangle  + \|B(u)\|^2_{L_2} =& -\kappa \langle |u|^{m-2} u,u\rangle + \|B(u)\|^2_{L^\infty_\Gamma L_2}  \\
	= &- \kappa \|u\|^m_{L^m(\Gamma)}  + \|B(u)\|^2_{L^\infty_\Gamma  L_2}\\
	 \leq &- \kappa \|u\|^m_{L^m(\Gamma)}  + C(1+\|u\|^2_{L^2(\Gamma)}),
\end{align*}
for some possibly new $C>0$, so that \eqref{eq:CoerciveMSPDE} is verified with $p=m$.\\

Monotonicity, \eqref{eq:MonotonePrime}, is checked in the same way as in Example~\eqref{ex:p_laplace_reac_diffuse} only using the Lipschitz assumption on $B$ in addition. \\

The growth bound \eqref{eq:MGrowthSPDE} is obtained by calculating for $u \in L^m(\Gamma)$,
\begin{equation*}
	\|A(u)\|_{L^{\frac{m}{m-1}}(\Gamma)} = \|u^{m-1}\|_{L^{\frac{m}{m-1}}(\Gamma)} = \|u\|^{m-1}_{L^m(\Gamma)}  \lesssim (1+\|u\|^m_{L^m(\Gamma)}),
\end{equation*}
so that again choosing $p=m$ the condition is satisfied.\\

Weak strong continuity in $L^m(\Gamma)$ spaces by the same arguments as in Example~\ref{ex:p_laplace_reac_diffuse}, see Remark~\ref{rem:l_p_weak_strong_continuous}.
	\qed
\end{example}
\begin{remark}
	Note that \eqref{eq:stoch_rd_field} is not strictly speaking a PDE, since we only take a derivative in time. Instead one can think of it as a field valued SDE. A class of (S)PDE not easily covered by this theory are the reaction diffusion equations of the form
	\begin{equation*}
		\dd u_t = \Delta u_t \dd t - \kappa |u_t|^{m-2} u_t\dd t + Q^{\nicefrac{1}{2}}\dd W_t,
	\end{equation*}
	for $\kappa>0$ and $Q^{\nicefrac{1}{2}}W$ a $Q$-Wiener process. The challenge here is to reconcile the growth bound \eqref{eq:MGrowthSPDE} in triples like $\mcH^1_0\subset L^2 \subset \mcH^{-1}$ with the coercivity property \eqref{eq:CoerciveMSPDE}. However, equations of this form can be treated relatively straightforwardly by a priori energy estimates.
\end{remark}
Before presenting the next example, we recall the Sobolev and Poincar\'e inequalities on bounded domains, we refer in general to \cite[Sec. 5.6 \& Subsec. 5.8.1]{evans_10} for this material. Let $\Gamma\subset \mbR^d$ be a closed, bounded, domain with smooth boundary and for $k \in \mbN$, $p\in [1,\infty]$, let $W^{k,p}(\Gamma)$ denote the space of maps $u:\Gamma\rightarrow \mbR$ whose first $k$ distributional derivatives are bounded in $L^p(\Gamma)$. We retain the special notation $\mcH^{1}(\Gamma):= W^{1,2}(\Gamma)$. We equip these spaces with the norms,
\begin{equation*}
	\|u\|_{W^{k,p}} := \sum_{n\leq k} \|D^{n}u\|_{L^p}. 
\end{equation*}
Let $l<k$, $q \in (p,\infty]$ be such that $\frac{1}{p}-\frac{k}{d} =\frac{1}{q}-\frac{l}{d}$. Then one has the bound,
\begin{equation*}
 \|u\|_{W^{l,q}}\leq \|u\|_{W^{k,p}}.
\end{equation*}
That is lower derivatives, in higher integrabillity are controlled by higher derivatives in lower integrabillity. As a particular case, for $d\geq 2$ and $p> d$, one always has,
\begin{equation*}
	\|u\|_{L^{p^*}} \leq \|u\|_{W^{1,p}},\quad \text{ for } p^* :=  \frac{pd}{d-p}
	\end{equation*}
In the edge case $p=d=2$ one in fact also has $\|u\|_{L^2}\leq \|u\|_{W^{1,2}}$ and in fact this embedding is compact. Furthermore, since $\Gamma$ is compact, we also have $\|u\|_{L^m}\leq \|u\|_{L^p}$ for $m\leq p$ and so in the above setting we have
\begin{equation*}
	\|u\|_{L^m}\leq \|u\|_{W^{1,p}} \quad \text{ for all } m\leq p^*.
	\end{equation*}
From which it follows that $W^{k,p}(\Gamma)\cap L^{m}(\Gamma) = W^{k,p}(\Gamma)$ for all $m \leq \frac{pd}{d-p}$.\\ \par 
The Poincar\'e inequality allows us to control the $L^p$ norms of a function by $L^p$ norms on its gradient. These inequalities are related to the spectral properties of the heat semi-group, see \cite{bakry_gentil_ledoux_14}. We recall that for any $p\in [1,\infty]$ there is a map, known as the trace map,
\begin{align*}
	T: W^{1,p}(\Gamma) &\rightarrow L^p(\partial \Gamma)\\
	u\mapsto u|_{\partial\Gamma}.
\end{align*}
For all $k\geq 1$, we define $W^{k,p}_0(\Gamma) = W^{k,p}(\Gamma)\cap \text{Ker}(T)$. Then, the classical Poincar\'e inequality says that for all $p\in [1,\infty]$ there exists a $C:= C(\Gamma,p)>0$ such that for all $u\in W^{1,p}_0(\Gamma)$,
\begin{equation*}
	\|u\|_{L^p} \leq C\|\nabla u\|_{L^p}.
\end{equation*}
Thus for example we have $\|u\|_{W^{1,p}}\leq (1+C) \|\nabla u\|_{L^p}$ for all $p\in [1,\infty]$. By considering $u = D^{k-1}v$ we also have $\|u\|_{W^{k-1,p}}\leq (1+kC)\|D^ku\|_{L^p}$. Of particular relevance for us will be the bound $\|u\|_{\mcH^1}:= \|u\|_{L^2}+ \|\nabla u\|_{L^2} \leq (1+C)\|\nabla u\|_{L^2} =:\|u\|_{\mcH^1_0}$. We furthermore equip $\mcH_0^1(\Gamma)$ with the inner product,
\begin{equation*}
	\langle u,v\rangle_{\mcH^1_0} := \int_{\Gamma} \nabla u(x)\nabla v(x)\dd x.
\end{equation*}
\vspace{0.3em}
\begin{example}[A Stochastic Porous Medium Equation]\label{ex:StochPorous}
	Keep $\Gamma \subset \mbR^d$, closed, bounded and with smooth boundary. We now let $m\in [2,\infty)$ and consider the Banach space $V = L^{m}(\Gamma)$ and will work with  $(\mcH^{1}_0(\Gamma))^\ast = \mcH^{-1}(\Gamma)$ as our Hilbert space and $V^\ast = (L^m(\Gamma))^\ast$ as the second Banach space. Since we are on a bounded domain, we have $L^2(\Gamma )\subset L^{\frac{m}{m-1}}(\Gamma)$ with $\frac{m}{m-1}=\frac{m}{m-1} \in (1,2]$. So it follows that $\mcH^1_0(\Gamma) \subset L^2(\Gamma)\subset L^m(\Gamma)$ and so $L^{\frac{m}{m-1}}(\Gamma ) \cong (L^m(\Gamma))^\ast \subset \mcH^{-1}(\Gamma)$. Let $Q \in L_1(L^2(\Gamma))$ and consider the stochastic porous medium equation, with $Q$-Wiener process $W$,
		\begin{equation*}\label{eq:stoch_porous_media}
		\begin{cases}
			\dd u_t = \nabla \cdot (|u_t|^{m-2}\nabla u_t) \dd t + Q^{1/2}\dd W_t,& \text{in }\Gamma,\\
			u_t|_{\partial \Gamma}=0, &\text{for all }t\geq 0,\\
			u\tzero	=u_0, &  \text{in }\Gamma.
		\end{cases}	
	\end{equation*}
Consider the operator $A(u)= \frac{1}{m-1}\Delta (|u|^{m-2}u) = \nabla \cdot (|u_t|^{m-2}\nabla u_t)$. It follows from the Riesz representation theorem, that for $\mcH^{1}_0$ equipped with the inner product $\langle \nabla u,\nabla v\rangle_{L^2}$, for any $\varphi \in (\mcH^{1}_0)^\ast$ there exists a $u\in \mcH^{1}_0$ such that for all $v \in \mcH^{1}_0$,
\begin{equation}
	\varphi(v) = \langle u_\varphi,v\rangle_{\mcH^{1}_0}= \langle -\Delta u_\varphi,v\rangle_{L^2}.
\end{equation}
Therefore, by linearity the map $\varphi \mapsto \Delta u_\varphi$ defines an isomorphism $(\mcH^{1}_0)^\ast\rightarrow \mcH^{1}_0$. However, since it is also the case that every element $u \in \mcH^{1}_0$ defines a linear functional we have that the map $\Delta :\mcH^1_0 \rightarrow (\mcH^1_0)^\ast$ is an isomorphism. The fact that $(\mcH^1_0(\Gamma))^\ast =\mcH^{-1}(\Gamma)$ as described in Example \ref{ex:DirLaplace} is an easy exercise. Thus we identify $\mcH^{1}_0$ with its dual via the map $-\Delta$ and we have the Gelfand triple $L^{m}\subset \mcH^{-1}\cong \mcH^{1}_0 \subset (L^{m})^\ast$. Under this identification, we equip $\mcH^{-1}$ with the inner product,
\begin{equation}
	\langle u,v\rangle_{\mcH^{-1}} = -\langle \Delta^{-1/2} u,\Delta^{-1/2}v\rangle_{L^2},
\end{equation}
where $(-\Delta)^{-1}:\mcH^{-1}\rightarrow \mcH^1_0$ is the inverse isomorphism. Note that when we identify $\mcH^{-1}\cong \mcH^{1}_0$ through the inner product, $\langle  u,  v\rangle_{\mcH^1_0}$, we cannot simultaneously identify $(L^{m}(\Gamma))^\ast \cong L^{\frac{m}{m-1}}(\Gamma)$, since the associated Riesz maps are not the same in each identification.\\ \par 
Before establishing \eqref{eq:CoerciveMSPDE}-\eqref{eq:WeakContinuitySPDE}, for $A(u)$, we show that $-\Delta$ can be continuously extended as an operator $\Delta :L^{\frac{m}{m-1}}(\Gamma)\rightarrow (L^{m}(\Gamma))^\ast$. Since $|u|^{m-2}u\in L^m$ for all $u \in L^{\frac{m}{m-1}}$ it follows that $A(u):=\Delta (|u|^{m-2}u): L^m(\Gamma)\rightarrow (L^{m}(\Gamma))^\ast$ is a continuous operator. To show the former claim, let $u \in \mcH^1_0$, so that $\Delta u \in \mcH^{-1} \subset (L^m)^\ast$ and $v\in L^m\subset \mcH^{-1}$, so that by the above discussion, we have
\begin{equation*}
\langle \Delta u,v\rangle_{(L^m)^\ast;L^m} = \langle \Delta u,v\rangle_{\mcH^{-1}} = \int_{\Gamma} u \Delta v = \langle u,v\rangle_{L^2}\leq \|u\|_{L^{\frac{m}{m-1}}(\Gamma)}\|v\|_{L^m(\Gamma)}.
\end{equation*}
So it follows that, under the above identification, $\|\Delta u\|_{(L^m)^\ast}\leq \|u\|_{L^{\frac{m}{m-1}}(\Gamma)}\leq \|u\|_{L^m(\Gamma)}$. Therefore the map $\Delta: \mcH^{1}_0\rightarrow (L^{m})^\ast$ extends continuously to a linear isometry $\Delta :L^{\frac{m}{m-1}}\rightarrow (L^{m})^\ast$. It follows in the same way, that the map $u\mapsto \Delta(|u|^{m-2}u)$ extends continuously to a map $L^{m}(\Gamma)\rightarrow  (L^{m}(\Gamma))^\ast$ and we are in the required setting.\\ \par
Coercivity \eqref{eq:CoerciveMSPDE} (with $p=m$) follows almost by definition, since we have $\langle \Delta(|u|^{m-2}u),u\rangle_{\mcH^{-1}} = -\langle |u|^{m-2},| u|^2\rangle_{L^2}= -\|u\|^m_{L^m}$. Monotonicity follows similarly, using the arguments of Example \ref{ex:p_laplace_reac_diffuse}. Assumption \ref{ass:WeakContinuitySPDE} is satisfied again by the strong-weak continuity of products in $L^m(\Gamma)$ spaces. Finally, the growth bound, \eqref{eq:MGrowthSPDE} (with $p=m$) should be clear.

\qed
\end{example}

		\chapter{Pathwise Approach to SPDE}\label{ch:Extras}
In this chapter we present a brief introduction to the pathwise approach to SPDE. To motivate this consideration, recall Example \ref{ex:StochPorous} from the previous Chapter. For these equations, where the noise enters additively, we had to directly consider a $Q$-Wiener process, with trace class covariance. It is natural to ask however, if we could instead consider a cylindrical Wiener process in these equations. This setting is relevant for example in equations of stochastic quantisation, where noise that is truly white in space and time is the correct object to consider. One approach is to use the mild formulation, provided the operator $A(u)$ is the generator of a sufficiently regular semi-group, see \cite{hairer_09}. However, this approach also has its limitations, as we will see below, when considering non-linear equations whose solutions must lie in a space of genuine distributions, rather than functions. Methods built on the pathwise approach and the theory of rough paths, \cite{lyons_98}, have recently had great success at handling these so called, singular SPDE, see for example the now foundational papers, \cite{hairer_14_RegStruct, gubinelli_imkeller_perkowski_15_GIP}. We do not discuss these more singular cases here, but instead present, by way of an example, some of the ideas behind these approaches, which had already been applied by Giuseppe Da-Prato, Jerzy Zabzyck and collaborators to SPDE problems motivated by quantum field theory, fluid dynamics and interface models, \cite{daprato_debussche_temam_94,daprato_debussche_96,daPrato_debussche_02,daPrato_debussche_03}.
\section{Regularity of Stochastic Processes}
As in the case of finite dimensional stochastic processes, the Kolmogorov continuity criterion is a useful tool for establishing $\mbP$-a.s. regularity statements for infinite dimensional stochastic processes. This result actually holds in the far more general context of complete metric spaces.
\begin{definition}
	Given a metric space $(E,\rho)$, $T>0$ and $\kappa \in (0,1)$, we say that a map $[0,T]\ni t\mapsto X_t\in E$ is $\kappa$-H\"older continuous if,
	\begin{equation*}
		\|X\|_{\mcC^\kappa E}:= \sup_{s\neq t \in [0,T]} \frac{\rho(X_t,X_s)}{|t-s|^{\kappa}}<\infty.
	\end{equation*}
	If $\kappa>1$ we say that $X \in \mcC^{\kappa}E$ if and only if $[0,T]\ni t\mapsto X_t\in E$ is $\floor{\kappa}$ times continuous differentiable the $\floor{\kappa}$ derivative is $\kappa-\floor{\kappa}$-H\"older continuous in the above sense.
\end{definition}
\begin{theorem}[Kolmogorov Criterion]\label{th:Kolmogorov}
	Assume that $(X_t)_{t\in [0,T]}$ is a stochastic process taking values in a complete metric space $(E,\rho)$ and that there exist constants $C>0$, $\varepsilon>0$ and $\delta>1$ such that for all $s,\,t \in [0,T]$,
	\begin{equation*}
		\mbE\left[ \rho(X_t,X_s)^\delta \right] \leq C|t-s|^{1+\varepsilon}.
	\end{equation*}
	Then there exists a version of $X$ (which we do not relabel) such that $\mbP$-a.s. $X\in \mcC^{\kappa}E$, for all $\kappa <\frac{\varepsilon}{\delta}$.
\end{theorem}
\begin{proof}
	The proof is a natural generalisation of the usual argument for finite dimensional stochastic processes. See \cite[Thm. 3.3]{daprato_zabczyk_14} for details.
\end{proof}
A useful property of Gaussian measures is the following corollary of Theorem \ref{th:Fernique}. In short it says that all higher moments of a Gaussian random variable are controlled by its first moment. This is related to a much more powerful result in the study of Gaussian measures, known as Nelson's hypercontractivity estimate, from which it follows that all moments of iterated integrals of a Gaussian random variable can be controlled by the second moment of the same iterated integral. This fact is related to the spectral gap of the Ornstein--Uhlenbeck semi-group, for an overview see \cite[Ch. 1]{nualart_10}. We state the result only for Gaussian random variables and in the context of separable Hilbert spaces even thought it holds almost without modification in separable Banach spaces, see \cite[Prop. 3.14]{hairer_09}.
\begin{proposition}\label{prop:FirstMomentBnd}
	Let $H$ be a separable Hilbert space, $\mu \in \mcP(H)$ be a Gaussian measure and let $M:= \int_H \|h\|_H\dd\mu(h)$. Then there exist constants $\alpha,\,C >0$ such that for any, $f:\mbR_+\rightarrow \mbR_+$ with the property that $f(x) \lesssim e^{\alpha x^2}$, one has
	\begin{equation}\label{eq:HypercontractiveBnd1}
		\int_{H} f\left(\frac{\|h\|_{H}}{M}\right)\dd\mu(h) \lesssim_f C.
	\end{equation}
In particular, for any $n\geq 1$, it holds that
\begin{equation}\label{eq:HypercontractiveBnd2}
	\int_H \|h\|_H^{2n}\dd\mu(h)\leq n! C \alpha^{-n}M^{2n}.
\end{equation} 
\end{proposition}
\begin{proof}
	We recall from the proof of Fernique's theorem, Theorem \ref{th:Fernique}, that we obtained the bound,
		\begin{align*}
		\int_H \exp\left(\frac{\tilde{\alpha}\|h\|^2_H}{\sigma^2} \right)\dd \mu(h) 
		&\leq e^{\tilde{\alpha}} + 2\tilde{\alpha} \int_1^\infty t e^{-\tilde{\alpha} t^2}\dd t, \label{eq:SigmaIndependent}
	\end{align*}
where, for any $\sigma>0$, $\tilde{\alpha}>0$ is such that $\mu(\|h\|_H>x)\leq \exp\left(-\frac{2\,\tilde{\alpha}\, x^2}{\sigma^2}\right)$ for all $x>\sigma$. Therefore, \eqref{eq:HypercontractiveBnd1} follows after setting $\sigma=4M$, for example, and using Chebyshev's inequality. The second bound, \eqref{eq:HypercontractiveBnd2} follows from Taylor's theorem which gives the bound $\frac{\alpha^nx^{2n}}{n!} \leq e^{\alpha x^2}$, for any $\alpha>0$ and $n\geq 1$. 
\end{proof}
\begin{corollary}\label{cor:GaussianHypercontractive} It follows that for an $H$-valued, Gaussian random variable $X$, and any $p>1$, there exists a constant, $C_p>0$ such that, $\mbE[\|X\|^p_H]\leq C_p\mbE[\|X\|^2_{H}]$.
\end{corollary}
\begin{proof}
From Proposition \ref{prop:FirstMomentBnd} and H\"older's inequality it follows that for any $p>1$ there exists a possibly different $C_p>0$ such that $\mbE[\|X\|_H^p]\leq C_p\mbE[\|X\|_H]$. Therefore, controlling the first moment by its second gives the result.
\end{proof}
To see how the Kolmogorov theorem, \ref{th:Kolmogorov} applies along with Corollary \ref{cor:GaussianHypercontractive}, let us consider an example of a Gaussian process solving a linear SPDE.
\begin{example}
	Let $(W_t)_{t\in [0,T]}$ denote a cylindrical Gaussian process, see Subsection \ref{subsec:GenGaussinRV} and $(v_t)_{t\in [0,T]}$ denote the solution to the linear, scalar, equation,
	\begin{equation}\label{eq:SHE}
		\begin{cases}
			\dd v_t = \Delta v_t\dd t + \dd W_t, & \text{ on  }[0,T]\times \mbT^d,\\
			v\tzero = 0, &\text{ on }\mbT^d.
		\end{cases}
	\end{equation}
	We can try to apply the method of Chapter \ref{ch:SPDE}, however, since we already saw that the white-noise $\mbP$-a.s. does not take values in $L^2(\mbT^d)$ this approach becomes tricky. We recall the representation of the cylindrical Wiener process, for some basis $(e_k)_{k\geq 1}$ of $L^2(\mbT^d)$, $W_t= \sum_{k\geq 1}e_k\beta^k_t$, we may decompose \eqref{eq:SHE} into the finite dimensional system of real, scalar, SDEs,
	\begin{equation*}
		\dd v_{k;t} = -|k|^2 v_{k;t} \dd t + \dd \beta^k_t,\quad v_{k;0}=0, \quad n\in \mbZ^d.
	\end{equation*}
	This defines, for each $n\in \mbZ^d$, a real, Ornstein--Uhlenbeck process, which we can solve explicitly,
	\begin{equation*}
		v_{k;t} =   \int_0^t e^{-|k|^2(t-s)}\dd \beta^k_s.
	\end{equation*}
	We would like to define the solution to \eqref{eq:SHE} as
	\begin{equation*}
		v_t := \sum_{k\in \mbZ^d} v_{k;t} e_k.
	\end{equation*}
	However, this raises the question of convergence of the sum and in tun, what space should we expect $t\mapsto v_t$ to live in. We can answer both questions using the Kolmogorov test, Theorem \ref{th:Kolmogorov}. Recall the scale of Sobolev spaces, for $\alpha \in \mbR$,
	\begin{equation*}
		H^\alpha:= H^\alpha(\mbT^d) :=  \left\{ \varphi \in \mcS'(\mbT^d)\,:\, \|\varphi\|^2_{H^\alpha}:= \sum_{k\in \mbZ^d} (1+|k|^2)^\alpha |\langle \varphi,e_k\rangle|^2 <\infty \right\}.
	\end{equation*}
	We observe that for each $n \in \mbZ^d$, $t\mapsto v_{k;t}$ is a Gaussian process and so from Corollary \ref{cor:GaussianHypercontractive}, for any $p\geq 1$, there exists a $C_p>0$ such that, for all $s,\,t \in [0,T]$,
	\begin{equation*}
		\mbE\left[ |v_{k;t}-v_{k;s}|^{2p} \right]\leq C_p \mbE\left[ |v_{k;t}-v_{k;s}|^2  \right]^p.
	\end{equation*}
	We will use this bound to estimate the $L^{2p}(\Omega;H^\alpha)$ norm of $v_t-v_s$. Strictly speaking we should first truncate the sum, at some $N>0$, only considering frequencies with $|k|\leq N$, then obtain uniform bounds in $N$ and take the limit. However, since the calculations are essentially the same we work directly with the formal limit $v$ and leave the precise limiting argument as a technical step. For $p\geq 1$ we have,
	\begin{align*}
		\mbE\left[ \|v_t-v_s\|^{2p}_{H^\alpha} \right] &= \sum_{k_1,\ldots,k_p\in \mbZ^d} \prod_{i=1}^p(1+|k_i|^2)^\alpha \mbE\left[ \prod_{i=1}^p |v_{k_i;t}-v_{k_i;s}|^2 \right]\\
		&\leq \sum_{k_1,\ldots,k_p\in \mbZ^d} \prod_{i=1}^p(1+|k_i|^2)^\alpha  \prod_{i=1}^p \mbE\left[|v_{k_i;t}-v_{k_i;s}|^{2p} \right]^{\frac{1}{p}}\\
		&\leq C_p \sum_{k_1,\ldots,k_p\in \mbZ^d} \prod_{i=1}^p(1+|k_i|^2)^\alpha  \prod_{i=1}^p \mbE\left[|v_{k_i;t}-v_{k_i;s}|^{2} \right].
	\end{align*}
	So it suffices to obtain good bounds on $\mbE\left[|v_{k_i;t}-v_{k_i;s}|^{2} \right]$. Note that for, $0\leq s<t\leq T$ and any $k \in \mbZ$, we have that
	\begin{equation*}
		v_{k;t}-v_{k;s} = (e^{-|k|^2(t-s)}-1)\int_0^s e^{-|k|^2(s-r)}\dd \beta^k_r+ \int_s^t e^{-|k|^2(t-r)}\dd \beta^k_r.
	\end{equation*}
	So by the usual finite dimensional, It\^o isometry, and the independent increments property of the Brownian motion, for each $i=1,\ldots,p$ we have,
	\begin{align*}
		\mbE\left|v_{k_i;t}-v_{k_i;s}|^2 \right] &=  \frac{1}{2|k_i|^2}\left(e^{-|k_i|^2(t-s)}-1\right)^2\left(1-e^{-2|k_i|^2s}\right)  +\frac{1}{2|k_i|^2}\left(1-e^{-2|k_i|^2(t-s)}\right).
	\end{align*}
	We can bound this expression using Taylor, so that for any $\kappa \in [0,1]$, we have,
	\begin{equation*}
		\mbE[|v_{k_i;t}-v_{k_i;s}|^2] \leq C_{\kappa} |k_i|^{2(\kappa-1)} |t-s|^\kappa. 
	\end{equation*}
	Returning to our bound on the $L^{2p}(\Omega;H^\alpha)$ norm of $v_t-v_s$, we now have,
	\begin{align*}
		\mbE\left[ \|v_t-v_s\|^{2p}_{H^\alpha} \right]  &\leq C_{p,\kappa} (t-s)^{\kappa p } \prod_{k_1,\ldots,k_p\in \mbZ^d} \prod_{i=1}^p(1+|k_i|^2)^\alpha   |k_i|^{2(\kappa-1)} \\
		&= C_{p,\kappa} (t-s)^{\kappa p } \sum_{k\in \mbZ^d}  \left((1+|k|^2)^\alpha   |k|^{2(\kappa-1)}\right)^p.
	\end{align*}
	So provided $\alpha <-\frac{d}{2}+1-\kappa$, the sum on the right hand side converges to a finite constant. Applying Theorem \ref{th:Kolmogorov} we establish the existence of a modification, which we do not relabel, such that $\mbP$-a.s. $v \in \mcC^{\kappa_p}H^{\alpha}$ for all $\kappa_p < = \frac{\kappa}{2}-\frac{1}{2p}$. Taking $p>1$ arbitrarily large we obtain that $v \in \mcC^{\kappa'}H^{\alpha-2\kappa'}$ for all $\alpha <-\frac{d}{2}+1$ and $\kappa' \in [0,1)$.\\ \par 
	This approach can be directly adapted to incorporate non-zero initial data, one only need include a finite variation term in the SDEs and analogous estimates hold, provided $v\tzero = v_0 \in H^{\alpha+2\kappa}$. As with the deterministic heat equation, the method can be further extended for less regular initial data, at the expense of measuring the solution in weighted H\"older spaces with prescribed blow-up as $t\searrow 0$.\\ \par 
	Notice that only for $d=1$ is $v_{t}$ a well defined spatial function. For all $d\geq 2$, $v_t$, only defines a distribution on $\mbT^d$. In the next section we will see that this causes an issue when considering  for example, non-linear equations driven by additive space-time white noise in $d\geq 2$.
\end{example}
\section{Pathwise Approach to SPDE: Stochastic Burger's Equation}
The above analysis gives us a way into a different perspective on SPDE, that being the pathwise approach. Instead of viewing solutions to SPDE as infinite dimensional stochastic processes and applying the tools of stochastic calculus, in the pathwise approach, one views solutions to SPDE as solutions to PDE driven by random coefficients. The idea is to build suitable input objects, typically coming from the noise source, a linearised equation or coupled system, and then solve a non-linear PDE involving these inputs for $\mbP$-a.a $\omega \in \Omega$ independently. We give an example of this method applied to the stochastic Burger's equation in one spatial dimension,
\begin{equation}\label{eq:StochBurgers}
	\begin{cases}
		\dd u_t = ( \partial_{xx}u_t+\partial_x (u_t^2))\dd t + \dd W_t, &\text{ on } [0,T]\times \mbT,\\
		u\tzero = u_0, &\text{ on }\mbT^,
	\end{cases}
\end{equation}
where $(\dd W_t)_{t\in [0,T]}$ is a space-time white noise and $u_0$ is a specified initial data. This problem was studied in \cite{daprato_debussche_temam_94} on the unit interval $(0,1)$ with Dirichlet boundary data. Although we work on the torus, mostly for presentational ease, our presentation is very similar to that of \cite{daprato_debussche_temam_94}.\\ \par
We note that \eqref{eq:StochBurgers} does not fall into the class of monotone equations - it can easily be checked that the non-linear transport term, $A(u):=\partial_x (u)^2 = 2 u \partial_x u$, does not satisfy $\langle Au_1 -A u_2,u_1-u_2\rangle_{L^2} \lesssim \|u_1-u_2\|_{L^2}$ for any proportionality constant. However, being a one dimensional analogue of the Navier--Stokes equations, it can be approached using compactness methods, \cite{metivier_viot_88} and this approach has been used to study more singular version of the equation, \cite{gubinelli_perkowski_20}.\\ \par
An alternative approach, however, is to treat \eqref{eq:StochBurgers} as a PDE with random forcing. An advantage of this approach is that in principle, the tools of PDE analysis, that have been well developed for deterministic counterparts of such equations can be brought to bear.\\ \par
We recall that in the previous section, we obtain the solution to \eqref{eq:SHE} as a stochastic process $(v_t)_{t\in [0,T]}$, with $\mbP$-a.s. continuous trajectories taking values in $H^\alpha$ for all $\alpha <\frac{1}{2}$. We do not expect the non-linear term to increase the regularity of $u$ over $v$ and so we see that we should look for solutions $t\mapsto u_t \in H^{1/2-} := \cap_{\alpha<1/2}H^\alpha$. This we can do almost directly by the Duhamel principle, and obtaining a fixed point locally in time. However, anticipating that in order to obtain global well-posedness we aim to obtain $\mbP$-a.s. a priori bounds on the quantities $\|u\|_{C_TL^p(\mbT)}$, we realise that we will need more than $1/2-$ spatial regularity. Therefore we do not work with $u$ directly but instead subtract the solution to the linear equation, \eqref{eq:SHE} in order to leave behind a more regular remainder.\\ \par
Observe that if $u$ solves \eqref{eq:StochBurgers} and we define $w := u-v$, where $v$ solves \eqref{eq:SHE}, then $w$ solves,
\begin{equation}\label{eq:BurgersRemainder}
	\begin{cases}
		\dd w_t = (\partial_{xx} w + \partial_x (w_t+v_t)^2)\dd t, & \text{ on }[0,T]\times \mbT,\\
		w\tzero = w_0, &\text{ on }\mbT.
	\end{cases}
\end{equation}
Since any solution to \eqref{eq:BurgersRemainder} also defines a solution to \eqref{eq:StochBurgers} by the formula, $u=w+v$, we now attempt to find solutions $w$ in a suitable Sobolev space. Our hope is that having subtracted the least regular part from $u$, our solution to \eqref{eq:BurgersRemainder} will live in a Sobolev space of regularity index above one.\\ \par 
It will also be useful to extend the scale of Sobolev spaces introduced above to the scale of Bessel potential spaces. For $\alpha \in \mbR$ and $p\in (1,\infty)$ we define,
\begin{equation}\label{eq:BesselPotentialSpaces}
	W^{\alpha,p}:= W^{\alpha,p}(\mbT) :=  \left\{ \varphi \in \mcS'(\mbT)\,:\, \|\varphi\|^2_{W^{\alpha,p}}:= \sum_{k\in \mbZ} (1+|k|^2)^{\frac{\alpha p}{2}} |\langle \varphi,e_k\rangle|^p <\infty \right\}.
\end{equation}
We make some remarks regarding the space $W^{\alpha,p}$.
\begin{itemize}
	\item The definition extends naturally for $d > 1$.
	\item For $\alpha \in \mbN$ these spaces agree with the usual Sobolev spaces.
	\item The analogous definition when $p=\infty$ defines the scales of H\"older spaces on $\mbT$ for $\alpha \in \mbR\setminus \mbZ$, and when $\alpha \in \mbN$ and $p=\infty$ these spaces agree with the usual $W^{k,\infty}$ spaces.
	\item For $\alpha =0$ and $p \in (1,\infty)$, we have $W^{0,p}(\mbT)=L^p(\mbT)$. When $\alpha \neq 0$ and $p=2$ we retain the notation $H^\alpha(\mbT)$.
	\item The derivative $\partial_x$ is a bounded linear map from $W^{\alpha,p}(\mbT)\rightarrow W^{\alpha-1,p}(\mbT)$. This is easily seen through Fourier multipliers.
	\item For $ p>q \in (1,\infty)$ and $\alpha \in \mbR$, one has the Sobolev type inequality,
	\begin{equation}\label{eq:SobEmbedding}
		\|\varphi\|_{W^{\alpha,p}(\mbT^d)}\leq C_{\alpha,p,q} \|\varphi\|_{W^{\beta,q}(\mbT^d)},\quad \beta = \alpha + d\left(\frac{1}{q}-\frac{1}{p}\right).
	\end{equation}
	\item The heat semi-group, $e^{t\Delta}$, has the following regularising effect, for all $q>p \in (1,\infty)$ and $\beta -2<\alpha \in \mbR$,
	\begin{equation}\label{eq:HeatFlow}
		\|e^{t\Delta}\varphi\|_{W^{\beta,q}} \leq C_{\alpha,\beta,p,q} t^{-\frac{\beta-\alpha}{2} - d\left(\frac{1}{p}-\frac{1}{q}\right)} \|\varphi\|_{W^{\alpha,p}}.
	\end{equation}
\end{itemize}
For more details on these spaces and the listed properties see for example \cite[Ch. 2]{bahouri_chemin_danchin_11}.\\ \par
Fix $T>0$, a $\mbP$-null set  $\msN\subset\Omega$ and an $\omega \in \Omega\setminus \msN$, such that $v:=v(\omega)\in \mcC_T^{\kappa'} H^{\alpha-2\kappa'}$ for some $\kappa' \in [0,1)$ and $\alpha<1/2$. We drop the explicit dependence on $\omega\in \Omega\setminus\msN$ from now on. By the Sobolev embedding, \eqref{eq:SobEmbedding}, for any $\alpha \in [0,1/2)$, 
we have that
\begin{equation*}
	\|v\|_{C_TL^p} \leq C_p \|v\|_{C_TH^\alpha},\quad p := \frac{2}{1-2\alpha},
\end{equation*}
%
%
Note that we can chose $\alpha \in [0,1/2)$ so as to achieve any  $p\in [0,\infty)$. \\ \par 
By Duhamel's principle, a solution to \eqref{eq:BurgersRemainder} will be given by a fixed point of the map,
\begin{equation}\label{eq:BurgersMild}
	\Psi w_t := e^{t\Delta}w_0+\int_0^t e^{(t-s)\Delta} \partial_x(w_s+v_s)^2 \dd s.
\end{equation}
We introduce the notation, $\lesssim$, to indicate that an inequality holds up to an unimportant constant. If we wish to specify some parameters on which this constant does depend, we write for example, $\lesssim_{\beta,p}$.
\vspace{-0.2em}
\begin{theorem}\label{th:BurgersLocal}
	Let $p \in [2,\infty)$ and $w_0 \in L^p(\mbT)$. Then there exists a $T_\ast \in (0,T]$ such that a unique, mild solution, $w \in C_{T_*}L^p$, exists for \eqref{eq:BurgersRemainder}. Furthermore, if $p>4$, there exists a $\beta>1$ such that for any $t \in (0,T_*]$, $\|w_t\|_{H^\beta}<\infty$.
\end{theorem}
\begin{proof}
	The argument for existence and uniqueness is fairly standard and so we only sketch it. For more details in this particular case see \cite[Lem. 2.1]{daprato_debussche_temam_94}.\\ \par
	First we define the unit ball, for any $T_{**} \in (0,T]$, we set
	\begin{equation*}
		\mfB_{T_{**}}:= \left\{ w \in C([0,T_{**}];L^q(\mbT))\,:\,\|w\|_{C_{T_{**}}L^p}\leq 1  \right\}.
	\end{equation*}
	Then, with $\Psi$ defined as in \eqref{eq:BurgersMild}, we have
	\begin{align*}
		\|\Psi w_t\|_{L^p} &\leq \|e^{t\Delta}w_0\|_{L^p}+\int_0^t \|e^{(t-s)\Delta}\partial_x(w_s+v_s)^2\|_{L^p}\,\dd s\\
		&\lesssim_p \|w_0\|_{L^p}+\int_0^t \|e^{(t-s)\Delta}\partial_x (w_s+v_s)^2\|_{W^{\frac{1}{p},\frac{p}{2}}}\,\dd s\\
		&\lesssim_p\|w_0\|_{L^p}+ \int_0^t (t-s)^{- \frac{1}{2p}-\frac{1}{2}} \|\partial_x (w_s+v_s)^2\|_{W^{-1,\frac{p}{2}}}\,\dd s\\
		&\lesssim_p \|w_0\|_{L^p}+\int_0^t (t-s)^{- \frac{1}{2p}-\frac{1}{2}} \| (w_s+v_s)^2\|_{L^{\frac{p}{2}}}\,\dd s\\
		&\lesssim_p \|w_0\|_{L^p}+\sup_{s \in [0,t]}\left(\|w_s\|^2_{L^p} + \|v_s\|^2_{L^p}\right)t^{\frac{1}{2}-\frac{1}{2p}}.
	\end{align*}
	This shows both that the right hand side of \eqref{eq:BurgersMild} is well defined for any $w \in \mfB_{T_{**}}$ and that choosing $T_{**}\in (0,T]$ sufficiently small we have that $\Psi:\mfB_{T_{**}}\rightarrow \mfB_{T_{**}}$. Considering $w,\,\tilde{w}\in \mfB_{T_{**}}$ and using similar estimates one additionally obtains, for some $T_* \in (0,T_{**}]$,
	\begin{equation*}
		\sup_{t \in T_*} \|\Psi w_t -\Psi w_s\|_{L^p} < \|w-\tilde{w}\|_{C_{T_*}L^p}.
	\end{equation*}
	Applying Banach's fixed point theorem we obtain a unique fixed point of $\Psi$ in $\mfB_{T^\ast}$, which by definition is a mild solution to \eqref{eq:BurgersRemainder}. It is not difficult to show that this fixed point is in fact unique in all of $C_{T_*}L^p$ which completes the proof of local existence and uniqueness. From now on we write $w \in C_{T_*}L^p$ for this solution.\\ \par
	In order to show high regularity at positive times we may use the fact that $\|w\|_{C_{T_*}L^p} \leq 1$ and $\|v\|_{C_TL^p} <\infty$. Let, $p>4$ and fix $\beta \in (1,2-4/p)$, by assumption this interval is non-trivial. Therefore, for any $t \in (0,T_*)$, and with $\alpha \in (0,1/2)$ as above,
	\begin{align*}
		\|w_t\|_{H^\beta} &\leq \|e^{t\Delta}w_0\|_{H^\beta} + \int_0^t \|e^{(t-s)\Delta}\partial_x (w_s+v_s)^2\|_{H^\beta}\dd s\\
		&\lesssim_\beta t^{-\frac{\beta}{2} - \frac{1}{p} + \frac{1}{2}} \|w_0\|_{L^p} +\int_0^t (t-s)^{-\frac{\beta+1}{2} - \frac{2}{p}+\frac{1}{2}}\|\partial_x (w_s+v_s)^2\|_{W^{-1,\frac{p}{2}}}\dd s\\
		&\lesssim_\beta t^{-\frac{\beta}{2} - \frac{1}{p} + \frac{1}{2}}\|w_0\|_{L^p} +\int_0^t (t-s)^{-\frac{\beta+1}{2} - \frac{2}{p}+\frac{1}{2}}\left(\|w_s\|^2_{L^p}+\|v_s\|^2_{L^p}\right)\dd s\\
		&\lesssim_\beta t^{-\frac{\beta}{2} - \frac{1}{p} + \frac{1}{2}}\|w_0\|_{L^p}  + t^{1- \frac{\beta}{2} -\frac{2}{p}}\left(1 + \|v\|^2_{C_TH^\alpha}\right)\\
		&<\infty.
	\end{align*}
\end{proof}
We proceed to establish an a priori bound on $\sup_{t \in T_*}\|w_t\|_{L^p}$, for $p\geq 2$. First, one needs to show that the mild solution constructed above is a weak solution, that is for any $\varphi \in H^1(\mbT)$, $t \in (0,T_*]$, one has
\begin{equation*}
	\langle w_t,\varphi\rangle  = \langle w_0,\varphi\rangle  + \int_0^t \langle \partial_x w_s,\partial_x \varphi\rangle + \langle \partial_x(w_s+v_s)^2,\varphi\rangle \dd s.
\end{equation*}
Since we have already established, $w \in C_{T_*}H^\beta \hookrightarrow C_{T_*}H^1$, one may argue, either by finite dimensional approximations or considering an increasing sequence of partitions of $[0,t]$, that we have the identity,
\begin{equation}\label{eq:RemainderPTest}
	\frac{1}{p}\left(\|w_t\|^p_{L^p}-\|w_0\|^p_{L^p}\right) = -\int_{0}^{t} \langle \partial_x w_s,\partial_xw_s^{p-1}\rangle + \langle  (w_s+v_s)^2,\partial_x w_s^{p-1}\rangle\dd s. 
\end{equation}
For details of this kind of argument see \cite[Sec. 3]{daprato_debussche_temam_94}, \cite[Sec. 6]{mourrat_weber_17_GWP}. Taking the identity \eqref{eq:RemainderPTest} as given we establish the following a priori estimate.
\begin{lemma}\label{lem:BurgersApriori}
	Let $p\geq 2$, $\bar{T}>0$ and $w \in C([0,\bar{T});L^p(\mbT))\cap C((0,\bar{T});H^1(\mbT))$  be a solution to \eqref{eq:BurgersRemainder}. Then, for any $t \in (0,\bar{T})$, with $\alpha \in (0,1/2)$ as above, we have that
	\begin{equation}\label{eq:RemApriori}
		\|w_t\|_{L^p} \lesssim_{p,T} \|w_0\|_{L^p} + \|v\|_{C_TH^\alpha}.
	\end{equation}
\end{lemma}
\begin{proof}
	Integrating by parts on the right hand side of \eqref{eq:RemainderPTest}, it follows that, for $p\geq 2$ and even, one has,
	\begin{align*}
		\frac{1}{p(p-1)}\left(\|w_t\|^p_{L^p}-\|w_0\|^p_{L^p}\right)& = -\int_{0}^{t} \| w_s^{p-2}|\partial_x w_s|^2\|_{L^1} \dd s \\
		&\quad -\int_0^t \langle w_s^{p},\partial_x w_s\rangle + 2\langle w_s^{p-1}v_s,\partial_x w_s\rangle   + \langle v_s^2, w^{p-2}_s\partial_x w_s\rangle \dd s. 
	\end{align*}
	We control the final three terms separately: firstly we have
	\begin{equation*}
		\langle w^p_s,\partial_x w_s\rangle = \frac{1}{p+1}\int_{\mbT} \partial_x w_s^{p+1}  =0;
	\end{equation*}
	for the second, letting $c>1$ be unspecified for now, we have
	\begin{align*}
		\langle w_s^{p-1}v_s,\partial_x w_s\rangle  & \leq \|v_s\|_{L^\infty} |\langle w_s^{p/2},w_s^{p/2-1}\partial_x w_s\rangle|\\
		&\leq \|v_s\|_{L^\infty} \|w_s\|^{p/2}_{L^p}\| w_s^{p-2}|\partial_x w_s|^2\|^{1/2}_{L^1}\\
		&\leq \frac{c}{2}\|v_s\|^2_{L^\infty} \|w_s\|^{p}_{L^p} + \frac{1}{2c}\| w_s^{p-2}|\partial_x w_s|^2\|_{L^1};
	\end{align*}
	and for the third, with $c>1$ as above,
	\begin{align*}
		\langle v_s^2, w^{p-2}_s\partial_x w_s\rangle  &\leq \|v_s\|^2_{L^\infty}|\langle w_s^{p/2-1},w_s^{p/2-1}\partial_x w_s\rangle|\\
		&\leq  \|v_s\|^2_{L^\infty} \|w_s^{p-2}\|^{\frac{1}{2}}_{L^1} \|w_s^{p-2}|\partial_x w_s|^2\|^{1/2}_{L^1}\\
		&\leq \frac{c}{4}\|v_s\|^{2p}_{L^\infty} + \frac{c}{4}\|w_s\|^p_{L^p} + \frac{1}{2c} \|w_s^{p-2}|\partial_x w_s|^2\|_{L^1}.
	\end{align*}
	Putting these together gives that
	\begin{align*}
		\frac{1}{p(p-1)}\left(\|w_t\|^p_{L^p}-\|w_0\|^p_{L^p}\right)& \leq  -\frac{2c-3}{2}\int_{0}^{t} \| w_s^{p-2}|\partial_x w_s|^2\|_{L^1} \dd s + \int_0^t \|w_s\|^p_{L^p}\left(c\|v_s\|^2_{L^\infty}+\frac{c}{4}\right)\dd s\\
		&\quad + \frac{c}{4}\int_0^t \|v_s\|^{2p}_{L^\infty} \dd s.
	\end{align*}
	Choosing $c>3/2$ so that the first term is negative, and applying Gr\"onwall we have that for all $t\in T_*$,
	\begin{equation*}
		\|w_t\|^p_{L^p} \leq \left(\|w_0\|^p_{L^p} + \frac{cp(p-1)}{4}\int_0^t \|v_s\|^{2p}_{L^p}\dd s \right) e^{t\left(c\|v\|^2_{C_TL^\infty}+\frac{c}{4}\right)}.
	\end{equation*}
	Taking the supremum outside the integral and abstracting the unimportant constants gives \eqref{eq:RemApriori} for $p\geq 2$ and even. To obtain the same for $p$ odd one may use the embedding $L^p(\mbT)\hookrightarrow L^{p'}(\mbT)$ for $p >p'$.
	\end {proof}
	With this a priori bound in hand, it is straightforward to demonstrate global well-posedness for \eqref{eq:BurgersRemainder}, at least for $w_0 \in L^p(\mbT)$, with $p>4$. Note that in \cite{daprato_debussche_temam_94}, the same result is obtained by a more refined version of this method, for $p\geq 2$.
	\begin{theorem}
		Let $T>0$, $p>4$ and $w_0 \in L^p(\mbT)$. Then there exists a unique, mild solution, $w \in C_TL^p$, to \eqref{eq:BurgersRemainder}.
	\end{theorem}
	\begin{proof}
		From Theorem \ref{th:BurgersLocal}, it is clear that for any $\bar{T}\in (0,T]$ and a solution $w \in C([0,\bar{T});L^p(\mbT))$ to \eqref{eq:BurgersRemainder}, then either $\lim_{t\rightarrow \bar{T}} \|w_t\|_{L^p} =\infty$ or $w$ is in fact a solution on all of $[0,\bar{T}]$. However, Lemma \ref{lem:BurgersApriori} shows that the former cannot be the case and so we may extend the solution indefinitely.
	\end{proof}
	It follows that $u:= w +v$ defines a global solution to \eqref{eq:StochBurgers}. In \cite{daprato_debussche_temam_94} it was additionally shown that for every $u_0 \in L^p((0,1))$, with $p\geq 2$, there exists an invariant measure, $\nu_{u_0} \in \mcP(L^p((0,1)))$, for \eqref{eq:StochBurgers}. Further results concerning both the Burgers and Navier--Stokes equations with additive noise, including uniqueness of the invariant measure and exponential ergodicity, can be found  in \cite{goldys_maslowski_05,flandoli_maslowski_95, daprato_debussche_04}. Some general introductions to ergodic theory for SPDE can be found in \cite{daprato_zabczyk_96, hairer_09, hairer_16_Ergodic}.

	\bibliographystyle{alpha}
	\bibliography{ThesisBib.bib}	

\newcommand{\etalchar}[1]{$^{#1}$}
\begin{thebibliography}{DKM{\etalchar{+}}09}

\bibitem[AC15]{allez_chouk_15}
Romain Allez and Khalil Chouk.
\newblock The continuous {A}nderson hamiltonian in dimension two.
\newblock arXiv:1511.02718, 2015.

\bibitem[Bat00]{batchelor_00}
George.~K. Batchelor.
\newblock {\em An Introduction to Fluid Dynamics}.
\newblock Cambridge Mathematical Library. Cambridge University Press, 2000.

\bibitem[BCD11]{bahouri_chemin_danchin_11}
Hajer Bahouri, Jean-Yves Chemin, and Rapha{\"e}l Danchin.
\newblock {\em Fourier Analysis and Nonlinear Partial Differential Equations}.
\newblock Springer, 2011.

\bibitem[BGL14]{bakry_gentil_ledoux_14}
Dominique Bakry, Ivan Gentil, and Michel Ledoux.
\newblock {\em Analysis and geometry of Markov diffusion operators}.
\newblock Springer, 2014.

\bibitem[Bre11]{brezis_11}
Haim Brezis.
\newblock {\em Functional analysis, {S}obolev spaces and partial differential
  equations}.
\newblock Universitext. Springer, New York, 2011.

\bibitem[CDG07]{chetrite_delannoy_gawedzki_07}
Raphaël Chetrite, Jean-Yves Delannoy, and Krzysztof Gawedzki.
\newblock Kraichnan flow in a square: An example of integrable chaos.
\newblock {\em Journal of Statistical Physics}, 126(6):1165–1200, Feb 2007.

\bibitem[CE15]{cohen_elliott_15}
Samuel~N. Cohen and Robert~James Elliott.
\newblock {\em {Stochastic Calculus and Applications}}.
\newblock Birkh{\"a}user, 2015.

\bibitem[Daw72]{dawson_72}
Donald.~A. Dawson.
\newblock Stochastic evolution equations.
\newblock {\em Math. Biosci.}, 15:287--316, 1972.

\bibitem[DD96]{daprato_debussche_96}
Giuseppe {Da Prato} and Arnaud Debussche.
\newblock Stochastic {C}ahn--{H}illiard equation.
\newblock {\em {Nonlinear Analysis\: Theory, Methods \& Applications}},
  26(2):241 -- 263, 1996.

\bibitem[DD04]{daprato_debussche_04}
Giuseppe {Da Prato} and Arnaud Debussche.
\newblock {Absolute Continuity of the Invariant Measures for Some Stochastic
  PDEs}.
\newblock {\em Journal of Statistical Physics}, 115(1):451--468, 2004.

\bibitem[DDT94]{daprato_debussche_temam_94}
Guiseppe {Da Prato}, Arnaud Debussche, and Roger Temam.
\newblock {Stochastic Burgers' equation}.
\newblock {\em Nonlinear Differential Equations and Applications NoDEA},
  1(4):389--402, 1994.

\bibitem[DH87]{damgaard_huffel_87}
Poul~H. Damgaard and Helmuth H{\"u}ffel.
\newblock Stochastic quantization.
\newblock {\em Physics Reports}, 152(5):227--398, 1987.

\bibitem[DKM{\etalchar{+}}09]{dalang_koshnevisan_mueller_nualart_xiao_09}
Robert Dalang, Davar Khoshnevisan, Carl Mueller, David Nualart, and Yimin Xiao.
\newblock {\em A minicourse on stochastic partial differential equations},
  volume 1962 of {\em Lecture Notes in Mathematics}.
\newblock Springer-Verlag, Berlin, 2009.

\bibitem[DPD03]{daPrato_debussche_03}
Giuseppe Da~Prato and Arnaud Debussche.
\newblock Strong solutions to the stochastic quantization equations.
\newblock {\em Ann. Probab.}, 31(4):1900--1916, 10 2003.

\bibitem[DPZ96]{daprato_zabczyk_96}
Guiseppe. Da~Prato and Jerzy. Zabczyk.
\newblock {\em {Ergodicity for Infinite Dimensional Systems}}.
\newblock London Mathematical Society Lecture Note Series. Cambridge University
  Press, 1996.

\bibitem[DPZ14]{daprato_zabczyk_14}
Giuseppe Da~Prato and Jerzy Zabczyk.
\newblock {\em {Stochastic Equations in Infinite Dimensions}}.
\newblock Encyclopedia of Mathematics and its Applications. Cambridge
  University Press, 2 edition, 2014.

\bibitem[Eva10]{evans_10}
Lawrence~Craig Evans.
\newblock {\em Partial {D}ifferential {E}quations}.
\newblock American Mathematical Society, 2010.

\bibitem[Far17]{farwig_17}
Reinhard Farwig.
\newblock Jean leray: Sur le mouvement d'un liquide visqueux emplissant
  l'espace.
\newblock {\em Jahresbericht der Deutschen Mathematiker-Vereinigung},
  119(4):249--272, 2017.

\bibitem[FGL21]{flandoli_galeati_luo_21_delayed}
F.~Flandoli, L.~Galeati, and D.~Luo.
\newblock \href{https://doi.org/10.1080/03605302.2021.1893748}{Delayed blow-up
  by transport noise}.
\newblock {\em Communications in Partial Differential Equations}, 0(0):1--39,
  2021.

\bibitem[FL21]{flandoli_luo_21_high}
Franco Flandoli and Dejun Luo.
\newblock High mode transport noise improves vorticity blow-up control in 3d
  navier--stokes equations.
\newblock {\em Probability Theory and Related Fields}, 180(1):309--363, 2021.

\bibitem[Fla11]{flandoli_11}
Franco Flandoli.
\newblock {\em Random perturbation of PDEs and fluid dynamic models}.
\newblock Springer, 2011.

\bibitem[Fla17]{flandoli_17}
Franco Flandoli.
\newblock Random initial conditions and noise in ordinary and partial
  differential equations.
\newblock Avaibale at:
  \url{http://users.dma.unipi.it/flandoli/randomSNS_LN2.pdf}, 2017.

\bibitem[Fle75]{fleming_75}
Wendell~H. Fleming.
\newblock Distributed parameter stochastic systems in population biology.
\newblock In {\em Control theory, numerical methods and computer systems
  modelling ({I}nternat. {S}ympos., {IRIA} {LABORIA}, {R}ocquencourt, 1974)},
  pages 179--191. Lecture Notes in Econom. and Math. Systems, Vol. 107. 1975.

\bibitem[FM95]{flandoli_maslowski_95}
Franco Flandoli and Bohdan Maslowski.
\newblock {Ergodicity of the 2-D Navier-Stokes equation under random
  perturbations}.
\newblock {\em Communications in Mathematical Physics}, 172(1):119--141, 1995.

\bibitem[GH21]{gubinelli_hofmanova_18_pde}
Massimiliano Gubinelli and Martina Hofmanov\'{a}.
\newblock \href{https://doi.org/10.1007/s00220-021-04022-0}{A {PDE}
  construction of the {E}uclidean {$\phi_3^4$} quantum field theory}.
\newblock {\em Comm. Math. Phys.}, 384(1):1--75, 2021.

\bibitem[GIP15]{gubinelli_imkeller_perkowski_15_GIP}
Massimiliano Gubinelli, Peter Imkeller, and Nicolas Perkowski.
\newblock \href{https://doi.org/10.1017/fmp.2015.2}{Paracontrolled
  distributions and singular {PDE}s}.
\newblock {\em Forum Math. Pi}, 3:75, 2015.

\bibitem[GK81a]{gyongy_krylov_81_1}
Istvan. Gy{\"o}ngy and Nicolai.~V. Krylov.
\newblock On stochastic equations with respect to semimartingales. {I}.
\newblock {\em Stochastics}, 4(1):1--21, 1981.

\bibitem[GK81b]{gyongy_krylov_81_2}
Istvan. Gy{\"o}ngy and Nicolai.~V. Krylov.
\newblock On stochastics equations with respect to semimartingales. {II}.
  {I}t\^{o} formula in {B}anach spaces.
\newblock {\em Stochastics}, 6(3-4):153--173, 1981.

\bibitem[GM05]{goldys_maslowski_05}
Benjamin. Goldys and Brian. Maslowski.
\newblock Exponential ergodicity for stochastic {Burgers} and 2{D}
  {Navier--Stokes} equations.
\newblock {\em Journal of Functional Analysis}, 226(1):230--255, 2005.

\bibitem[GP20]{gubinelli_perkowski_20}
Massimiliano Gubinelli and Nicolas Perkowski.
\newblock The infinitesimal generator of the stochastic {B}urgers equation.
\newblock {\em Probab. Theory Related Fields}, 178(3-4):1067--1124, 2020.

\bibitem[GY21]{gess_yaroslavtsev_21}
Benjamin Gess and Ivan Yaroslavtsev.
\newblock Stabilization by transport noise and enhanced dissipation in the
  {K}raichnan model.
\newblock arXiv:2104.03949, 2021.

\bibitem[Gy{\"o}82]{gyongy_82_3}
Istvan. Gy{\"o}ngy.
\newblock On stochastic equations with respect to semimartingales. {III}.
\newblock {\em Stochastics}, 7(4):231--254, 1982.

\bibitem[Hai08]{hairer_16_Ergodic}
M.~Hairer.
\newblock {Ergodic Theory for Stochastic PDEs}.
\newblock Technical report, The University of Warwick, 2008.
\newblock Available at: \url{http://hairer.org/notes/Imperial.pdf}.

\bibitem[Hai09]{hairer_09}
M.~Hairer.
\newblock {An Introduction to Stochastic PDEs}.
\newblock Technical report, The University of Warwick / Courant Institute,
  2009.
\newblock Available at: \url{http://hairer.org/notes/SPDEs.pdf}.

\bibitem[Hai14a]{hairer_14_RegStruct}
M.~Hairer.
\newblock \href{https://doi.org/10.1007/s00222-014-0505-4}{A theory of
  regularity structures}.
\newblock {\em Invent. Math.}, 198(2):269--504, 2014.

\bibitem[Hai14b]{hairer_14_Phi43}
M.~Hairer.
\newblock Regularity structures and the dynamical $\phi^4_3$ model.
\newblock {\em Current Developments in Mathematics}, 2014(1):1--49, 2014.

\bibitem[HPS16]{harbrecht_peters_siebenmorgen_16}
Helmut. Harbrecht, Mike. Peters, and Markus. Siebenmorgen.
\newblock Analysis of the domain mapping method for elliptic diffusion problems
  on random domains.
\newblock {\em Numerische Mathematik}, 134(4):823--856, 2016.

\bibitem[{Kaz}68]{kazantzev_68}
Alexander.~P. {Kazantsev}.
\newblock Enhancement of a magnetic field by a conducting fluid.
\newblock {\em Soviet Journal of Experimental and Theoretical Physics},
  26:1031, May 1968.

\bibitem[{Kra}68]{kraichnan_68}
Robert~H. {Kraichnan}.
\newblock {Small-Scale Structure of a Scalar Field Convected by Turbulence}.
\newblock {\em Physics of Fluids}, 11(5):945--953, May 1968.

\bibitem[Kup18]{kupiainen_16}
Antti Kupiainen.
\newblock Quantum fields and probability.
\newblock Available at:
  \url{https://courses.helsinki.fi/sites/default/files/course-material/4594153/MathPhys2018final.pdf},
  2018.

\bibitem[Lab19]{labbe_19}
Cyril Labbé.
\newblock The continuous {A}nderson hamiltonian in $d\le 3$.
\newblock {\em Journal of Functional Analysis}, 277(9):3187--3235, 2019.

\bibitem[Ler34]{leray_34}
Jean Leray.
\newblock Sur le mouvement d'un liquide visqueux emplissant l'espace.
\newblock {\em Acta Mathematica}, 63(1):193--248, 1934.

\bibitem[Lio11]{lions_11}
Jacques.~L. Lions.
\newblock {\em Quelques Problem{\'e}s de la Theorie des Equations Non
  Lin{\'e}aires D'{\'e}volution}, pages 189--342.
\newblock Springer Berlin Heidelberg, Berlin, Heidelberg, 2011.

\bibitem[LR15]{liu_rockner_15_spde_introduction}
Wei Liu and Michael R\"ockner.
\newblock {\em Stochastic partial differential equations: an introduction}.
\newblock Universitext. Springer, Cham, 2015.

\bibitem[LR17]{lototsky_rozovsky_17}
Sergey~V Lototsky and Boris~L'vovic Rozovsky.
\newblock {\em Stochastic partial differential equations}.
\newblock Springer, 2017.

\bibitem[Lyo98]{lyons_98}
Terry Lyons.
\newblock Differential equations driven by rough signals.
\newblock {\em Revista Matemática Iberoamericana}, pages 215--310, 1998.

\bibitem[MK99]{majda_kramer_99}
Andrew~J. Majda and Peter~R. Kramer.
\newblock Simplified models for turbulent diffusion: Theory, numerical
  modelling, and physical phenomena.
\newblock {\em Physics Reports}, 314(4):237--574, 1999.

\bibitem[Mof02]{moffatt_02}
Keith Moffatt.
\newblock {G.K. Batchelor} and the homogenization of turbulence.
\newblock {\em Annu. Rev. Fluid Mech}, 34:19--35, 01 2002.

\bibitem[MV88]{metivier_viot_88}
Michel Metivier and Michel Viot.
\newblock On weak solutions of stochastic partial differential equations.
\newblock In {\em Stochastic Analysis}, pages 139--150, Berlin, Heidelberg,
  1988. Springer Berlin Heidelberg.

\bibitem[MW17]{mourrat_weber_17_GWP}
Jean-Christophe Mourrat and Hendrik Weber.
\newblock \href{https://doi.org/10.1214/16-AOP1116}{Global well-posedness of
  the dynamic $\Phi^{4}$ model in the plane}.
\newblock {\em The Annals of Probability}, 45(4):2398--2476, 2017.

\bibitem[Nua10]{nualart_10}
David Nualart.
\newblock {\em {Malliavin Calculus and Related Topics}}.
\newblock Springer, 2010.

\bibitem[{\O}13]{oksendal_13}
Bernt {\O}.
\newblock {\em {Stochastic Differential Equations}}.
\newblock Springer, 1 edition, 2013.

\bibitem[OP17]{ozanksi_pooley_17}
Wojciech~S. {Ożański} and Benjamin~C. Pooley.
\newblock Leray's fundamental work on the navier-stokes equations: a modern
  review of "sur le mouvement d'un liquide visqueux emplissant l'espace".
\newblock arXiv:1708.09787, 2017.

\bibitem[Par72]{pardoux_72}
{\'E}tienne Pardoux.
\newblock Sur des {\'e}quations aux d{\'e}riv{\'e}es partielles stochastiques
  monotones.
\newblock {\em C. R. Acad. Sci. Paris S\'{e}r. A-B}, 275:A101--A103, 1972.

\bibitem[Par79]{pardoux_79}
{\'E}tienne. Pardoux.
\newblock Stochastic partial differential equations and filtering of diffusion
  processes.
\newblock {\em Stochastics}, 3(2):127--167, 1979.

\bibitem[Par21]{pardoux_21_spde_introduction}
\'Etienne Pardoux.
\newblock {\em Stochastic partial differential equations---an introduction}.
\newblock SpringerBriefs in Mathematics. Springer, Cham, [2021] \copyright
  2021.

\bibitem[PD02]{daPrato_debussche_02}
Giuseppe~Da Prato and Arnaud Debussche.
\newblock Two-dimensional {N}avier--{S}tokes equations driven by a space–time
  white noise.
\newblock {\em Journal of Functional Analysis}, 196(1):180 -- 210, 2002.

\bibitem[PR07]{prevot_rockner_07}
Claudia Pr{\'e}v{\^o}t and M.~R{\"o}ckner.
\newblock {\em A concise course on stochastic partial differential equations}.
\newblock Springer, 2007.

\bibitem[Ser14a]{seregin_14}
Gregory Seregin.
\newblock {\em Lecture Notes on Regularity Theory for the {Navier--Stokes}
  Equations}.
\newblock World Scientific, 2014.

\bibitem[Ser14b]{seregin_14_Free}
Gregory Seregin.
\newblock Lecture notes on regularity theory for the {Navier--Stokes}
  equations.
\newblock Available at:
  \url{http://www.maths.ox.ac.uk/system/files/attachments/Lecture%20Notes%2014.01.pdf},
  2014.

\bibitem[Tzv16]{tzvetkov_16}
Nikolay Tzvetkov.
\newblock Randmon data wave equations.
\newblock Available at:
  \url{http://php.math.unifi.it/users/cime/Courses/2016/02/201624-Notes.pdf},
  2016.

\bibitem[vNVW07]{neerven_veraar_weis_07}
Jan. M. A.~M. van Neerven, Mark.~C. Veraar, and Lutz. Weis.
\newblock Stochastic integration in {UMD} banach spaces.
\newblock {\em The Annals of Probability}, 35(4):1438--1478, 2007.

\bibitem[XT06]{xiu_tartakovsky_06}
Dongbin Xiu and D.~M. Tartakovsky.
\newblock Numerical methods for differential equations in random domains.
\newblock {\em SIAM Journal on Scientific Computing}, 28(3):1167--1185, 2006.

\bibitem[Zak69]{zakai_69}
Moshe Zakai.
\newblock On the optimal filtering of diffusion processes.
\newblock {\em Z. Wahrscheinlichkeitstheorie und Verw. Gebiete}, 11:230--243,
  1969.

\bibitem[Zam17]{zambotti_15}
Lorenzo Zambotti.
\newblock {\em Random obstacle problems}, volume 2181 of {\em Lecture Notes in
  Mathematics}.
\newblock Springer, Cham, 2017.
\newblock Lecture notes from the 45th Probability Summer School held in
  Saint-Flour, 2015.

\end{thebibliography}

\end{document}